\documentclass[12pt]{amsart}
\usepackage[T1]{fontenc}
\usepackage{amsfonts, amsthm, amsmath, amssymb,array}
\usepackage{times}
\usepackage{color}

\renewcommand{\[}{\begin{equation}}
\renewcommand{\]}{\end{equation}}
\newcommand{\ip}[2]{\langle{#1},{#2}\rangle}
\newcommand{\ipp}[2]{\langle\!\langle{#1},{#2}\rangle\!\rangle}
\newcommand{\br}[1]{{\langle{#1}\rangle}} 
\newcommand{\pbr}[1]{\langle{#1}\rangle^\backprime}

\newcommand{\Br}[1]{{\{{#1}\}}} 
\newcommand{\pBr}[1]{{\{{#1}\}^{\backprime}\!\hsp}} 
\newcommand{\cBr}[1]{\{{#1}\}^\circ} 
 
\newcommand{\oeq}[1]{\overset{#1}{=}}

\newcommand{\hsp}{\hspace{-1pt}} 
\newcommand{\hs}{\hspace{1pt}}

\def\K{{\mathbb K}}
 
\def\Z{\mathbb{Z}} 
\def\N{\mathbb{N}} 
\def\lra{\longrightarrow} 
\def\id{\mathrm{id}} 

\newcommand{\msum}[2]{\underset{#1}{\overset{#2}{\mbox{$\sum$}}}}

\def\ra{\triangleleft} 
\def\la{\triangleright} 
\def\eps{\varepsilon}

\def\ot{{\otimes}} 
\def\uc{\underline{\circ}}
\def\um{\underline{m}}

 \def\vare{\varepsilon}
 \def\ust{\hs\underline{*}\, }

 \def\uD{\underline{\Delta}}

 \def\stac#1{\raise-.2cm\hbox{$\stackrel{\displaystyle\otimes}{\scriptscriptstyle{#1}}$}}

\def\cten#1{\raise-.2cm\hbox{$\stackrel{\displaystyle\widehat{\otimes}}{\scriptscriptstyle{#1}}$}}

\def\Label#1{\label{#1}\ifmmode\llap{[#1] }\else 
\marginpar{\smash{\hbox{\tiny [#1]}}}\fi} 
\def\Label{\label} 
\newtheorem{proposition}{Proposition}[section]
\newtheorem{lemma}[proposition]{Lemma} 
\newtheorem{corollary}[proposition]{Corollary} 
\newtheorem{theorem}[proposition]{Theorem} 
\theoremstyle{definition} 
\newtheorem{definition}[proposition]{Definition}
 
\newtheorem{example}[proposition]{Example} 
 
\theoremstyle{remark} 
\newtheorem{remark}[proposition]{Remark} 
 
\def\ot{\otimes}
\def\uot{\ul{\otimes}}

\def\ul{\underline}

\newcounter{mnotecount}[section]
\renewcommand{\themnotecount}{\thesection.\arabic{mnotecount}}
\newcommand{\mnote}[1]
{\protect{\stepcounter{mnotecount}}$^{\mbox{\footnotesize
$
\bullet$\themnotecount}}$ \marginpar{
\raggedright \tiny\em
$\bullet$\themnotecount: #1} }

\usepackage{chngcntr}
\counterwithout{equation}{section}

\begin{document} 
\title[Duality for infinite-dimensional braided bialgebras]{Duality for 
infinite-dimensional braided bialgebras and their (co)modules} 
\author[E. Wagner]{Elmar Wagner}
\address{Instituto de F\'isica y Matem\'aticas, Universidad Michoacana de San Nicol\'as de Hidalgo, 
Edificio C-3, Ciudad Universitaria, 
58040 Morelia, Michoac\'an, M\'exico}
\email{elmar.wagner@umich.mx}

\begin{abstract} 
The paper presents a detailed description of duality for braided algebras, coalgebras, 
bialgebras, Hopf algebras and their modules and comodules in the infinite setting. 
Assuming that the dual objects exist, it is shown how a given braiding induces 
compatible braidings for the dual objects, and how actions (resp.\ coactions) 
can be turned into coactions (resp.\ actions) of the dual coalgebra (resp.\ algebra), 
with an emphasis on braided bialgebras and their braided (co)module algebras. 
\end{abstract} 

\maketitle

\section{Introduction} 

The objective of this paper is to give a detailed description of duality for braided algebras, coalgebras, 
bi- and Hopf algebras, and their modules and comodules in the infinite setting.  
The interest in duality for these structures goes back to Majid  \cite{Maj92, Maj94a},  
who did much of the pioneering work on braided Hopf algebras 
\cite{Maj91,Maj91a,Maj93,Maj93a,Maj95}, 
and to Takeuchi \cite{T99, T}. Focusing on specific applications, duality 
for braided Hopf algebras has been studied by 
Lyuvashenko \cite{L} and by Guo and Zhang \cite{GZ} in the context of integrals, 
and by Da Rocha, J.A. Guccione and J. J. Guccione \cite{RGG} 
in connection with crossed products. 
The theory of braided Hopf crossed products reveals in particular the 
relevance of braided (co)module algebras \cite{GG,GGV} which 
serve as a guiding principle for this paper. 
Duality for infinite braided Hopf algebras, 
under the assumption that the braiding is symmetric, 
has been considered in the Yetter-Drinfeld module category
in the light of the Blattner-Montgomery duality theorem 
by Han and Zhang \cite{HZ}, and by  Cheng, Xu and  Zhang \cite{CXZ}. 

In the finite setting, duality is most conveniently studied in a rigid 
monoidal category \cite{Maj94a,T99}. The problem in the infinite setting is the lack of a 
so-called coevaluation map. To avoid this problem, we do not follow a 
categorical approach but define the dual objects by a set of conditions,  
similar to the path taken by Takeuchi in \cite{T}. 
In particular, we will neither prove the existence nor the uniqueness of a 
dual space with the desired properties, so we will not define 
a functor into a dual category. This may be especially useful 
in a topological setting, where it is not always practical 
to work with full dual space, for instance if there are unbounded (braiding) 
operators involved. However, the entire paper is kept completely  algebraic 
even though the more interesting examples arise in a topological framework. 
In this sense, one may consider all tensor products as algebraic 
tensor products of linear spaces over a field~$\K$. 

Our first aim is to establish a duality theory for infinite-dimensional 
braided bialgebras and Hopf algebras. This will be done in Section \ref{sec4}. 
The guiding principle emanates from the definition of 
a dual pairing between braided bialgebras in Definition \ref{dualp}. 
To take into account the braided setting, we include in this definition 
a braiding between dual spaces so that the dual pairing 
of two-fold tensor products can be realized by 
evaluating simultaneously adjacent tensor factors. 
The fundamental idea of our approach is that all structures 
on dual spaces should be induced from the given ones, 
including the braiding appearing in the dual pairing.  
In this sense, our method is constructive, 
only the existence of a dual space with the required properties will be assumed, 
wheras all algebraic properties will be deduced from the original source.

To develop the theory step by step, we start by elaborating a 
duality theory for infinite-dimensional algebras and coalgebras 
in Propositions \ref{lust} and \ref{lcop}.  Before doing so, 
we show in Lemma \ref{lembraided} how a given braiding induces braidings on a dual space  
and between the space and its dual which are compatible with 
additional algebraic structures like multiplication or comultiplication. 
Our definition of a product or coproduct on the dual space 
is intimately related to the compatibility conditions of a dual pairing. 
The construction of a dual braided bialgebra will be achieved in Theorem~\ref{dH} 
under certain assumptions on the chosen dual space which 
guarantee that the induced braidings define bijective maps into the 
correct tensor products and that the product and coproduct are well-defined. 
Proposition \ref{propdual} shows that the construction is reflexive in the sense 
that taking twice the dual gives back the same braided bialgebra. 
The extension of these results to braided Hopf algebras requires only a minor 
condition regarding the antipode. 

In Section \ref{sec5}, we address duality for braided modules and comodules. 
The starting point is again to induce new braidings for dual spaces 
from a  given braiding between a (co)algebra and a (co)module in such a way that 
the compatibility properties are maintained. 
This will be done  step by step in Lemmas \ref{dbvs}-\ref{bUW}, 
each time replacing one of the two involved spaces by a dual space. 
Since a left braided vector space induces the structure of 
a right braided vector space on duals, and vice versa, we will frequently use the 
inverse of a braiding to recover a left-handed or right-handed version. 
In fact, one of the purposes of this paper is to single out the correct braidings and 
formulas so that it may serve as a reference for others. 

Theorems \ref{mcom} and \ref{mmm} are the main results of Section \ref{sec5}. 
There it is shown how to transform a braided comodule into a braided module of 
a dual (bi)algebra, and a braided module into a braided comodule of 
a dual co- or bialgebra. Interestingly, in the latter case, it will not exactly yield  
a comodule of a dual bialgebra but a version of it that corresponds to taking twice 
the braided opposite and co-opposite bialgebra. These structures, where 
the product or coproduct is flipped by a power of the braiding, 
are discussed at greater length in Section~\ref{sec3} since such considerations 
do play a role in subsequent results. For instance, they justify to present 
only one version of a braided dual (co)algebra in Propositions \ref{lust} and \ref{lcop}, 
other versions can be obtained by combining the constructed (co)product with 
powers of the braiding. Moreover, twisting products, coproducts, actions or coactions 
with a braiding may give rise to whole families of new structures as illustrated  
in Proposition \ref{lemmD} and Corollary \ref{corind}. 
In Section~\ref{sec3}, we also review the significance of the antipode 
in the braided setting for turning left (co)actions into right (co)actions and vice versa. 

In Proposition \ref{mco}, we dualize a coaction on a comodule 
to an action of a dual algebra on a dual of the comodule,  
and in Proposition \ref{com}, we dualize 
an action on a module to a coaction of a dual coalgebra on a dual of the module.
Despide the fact that our main interest lies in module and comodule algebras 
of braided bialgebras, Propositions \ref{mco} and \ref{com} do not consider these 
topics since we would then have to introduce the dual objects, namely 
module and comodule coalgebras. To keep the length of the paper reasonable, 
we refrain from introducing (co)module coalgebras. 
With a detailed description of braided (co)module algebras at hand, 
it should be clear how to dualize these notions to braided (co)module coalgebras. 

Also for brevity, we do not elaborate an example in full detail but sketch only that 
graded braided (co)algebras, bi- and Hopf algebras and their graded (co)modules 
provide infinite-dimensional examples. For the same reason, we completely avoid braid diagrams 
since the proofs presented by braid diagrams would occupy considerably more space. 
Instead of braid diagrams, we introduce a Sweedler-type notation and annotate the 
employed relations over the equality signs. Once the reader gets used to this notation, 
it shouldn't be a problem to draw the corresponding braid diagrams, 
one only has to be careful with the chosen crossings. For instance, 
the crossings of the induced braidings 
have to be compatible with the given braidings, and 
the braidings obtained from an inverse braiding 
should be denoted differently than those induced from a given one. Although the proofs 
are rather straightforward, we present most of them in order to show where 
the involved braidings and compatibility relations are used.

\section{Preliminaries on braided bialgebras and their (co)modules} 

In this section, we give a working definition for braided bi- and Hopf algebras 
without using braided tensor categories. 
The reason is that we want to establish a duality theory for 
infinite-dimensional braided bialgebras. 
The categorical approach works well in rigid monoidal categories, 
but for infinite-dimensional examples, there is a problem with the rigorous definition of 
a co-evaluation map, see e.g.~\cite{RGG,T}.  
Moreover, a dual braided bialgebra may not exist, and if it exists, it may not be unique, 
therefore we do not aim at defining a functor into a dual category.  

Throughout this paper, the letter $\K$ stands for an arbitrary field. 
A braiding for a vector space $V$ 
is a bijective linear map $\Psi_{VV}: V\ot V \to V\ot V$ fulfilling the Yang--Baxter equation 
\[   \label{YBE} 
(\Psi_{VV} \ot \id) \circ (\id \ot  \Psi_{VV}) \circ ( \Psi_{VV} \ot \id) = ( \id \ot \Psi_{VV}) \circ (\Psi_{VV} \ot \id) \circ (\id \ot  \Psi_{VV}) . 
\] 
Let $V$ be a braided vector space. 
A left $V$-braided vector space is a vector space $W$ together with a bijective linear map $\Psi_{VW}: V\ot W \to W\ot V$ 
such that 
\[   \label{VW} 
(\Psi_{VW} \ot \id) \circ (\id \ot  \Psi_{VW}) \circ ( \Psi_{VV} \ot \id) = ( \id \ot \Psi_{VV}) \circ (\Psi_{VW} \ot \id) \circ (\id \ot  \Psi_{VW}) . 
\] 
Similarly, $V$ is called a right $W$-braided vector space if $W$ is a braided vector space and the bijective linear map
$\Psi_{VW}: V\ot W \to W\ot V$ satisfies 
\[   \label{WV} 
 ( \id \ot \Psi_{VW}) \circ (\Psi_{VW} \ot \id) \circ (\id \ot  \Psi_{WW})
= (\Psi_{WW} \ot \id) \circ (\id \ot  \Psi_{VW}) \circ ( \Psi_{VW} \ot \id) . 
\] 
The archetypal example, also in the case $V=W$, is given by the flip:  
\[ \label{flip} 
\tau : V\ot W \lra W\ot V, \quad \tau (v\ot w):= w\ot v. 
\]

Let $V$ be an algebra with multiplication $m_V :V\ot V \to V$. 
If $W$ is a left $V$-braided vector space, or if $V$ is a right $W$-braided vector space, 
then we say that the braiding $\Psi_{VW}$ is compatible with the multiplication if 
\[       \label{AWm} 
\Psi_{VW}\circ (m_V \ot \id) = (\id \ot m_V) \circ (\Psi_{VW} \ot \id) \circ (\id \ot \Psi_{VW})
\]
and, if $1\in V$, 
\[ \label{1W}
\Psi_{VW}(1 \ot w) = w\ot 1 , \qquad w\in W. 
\]
Assume now that $W$ is an algebra with multiplication $m_W :W\ot W \to W$ 
and that $W$ is a left $V$-braided vector space or that $V$ is a right $W$-braided vector space. 
Then we say that the braiding $\Psi_{VW}$ is compatible with the multiplication if 
\[   \label{VAm} 
 \Psi_{VW}\circ (\id\ot m_W) = (m_W\ot \id ) \circ (\id \ot \Psi_{VW}) \circ (\Psi_{VW} \ot \id) 
\]
and 
\[\label{V1}
 \Psi_{VW} (v \ot 1) = 1\ot v, \qquad v\in V. 
\]

A braided algebra is an algebra $(A,m)$  which is a braided vector space such that 
the braiding is compatible with the multiplication. 
In this case, it follows from  \eqref{AWm} and \eqref{VAm} that 
\[     \label{Pmm}
\Psi_{AA}  \circ (m\ot m) 
= (m\ot m) \circ (\id \ot \Psi_{AA} \ot \id) \circ (\Psi_{AA} \ot \Psi_{AA} ) \circ (\id \ot \Psi_{AA} \ot \id). 
\] 
Clearly, each algebra $A$ becomes a braided algebra with the usual flip 
defined in \eqref{flip} as braiding isomorphism. 

Suppose now that $(V,\Delta,\eps)$ is a coalgebra and $W$ is a left $V$-braided vector space
or $V$ is a right $W$-braided vector space. 
We say that the 
braiding $\Psi_{VW} : V\ot W \to W\ot V$ is compatible with the comultiplication  of $V$ if 
\[ \label{PHW} 
(\id \ot \Delta)\circ   \Psi_{VW} =  (\Psi_{VW} \ot \id)  \circ (\id \ot \Psi_{VW}) \circ  (\Delta \ot \id), \ \ 
( \id\ot \eps)\circ  \Psi_{VW} = \eps\ot \id. 
\]
If $(W,\Delta,\eps)$ is a coalgebra and $V$ is a right $W$-braided vector space
or $W$ is a left $V$-braided vector space, then 
the analogous definitions read 
\[  \label{PVH} 
(\Delta \ot \id)\circ \Psi_{VW} =   (\id \ot \Psi_{VW}) \circ (\Psi_{VW} \ot \id)  \circ (\id \ot \Delta), \ \ 
(\eps \ot \id)\circ  \Psi_{VW} = \id \ot \eps. 
\]  
A braided coalgebra is a coalgebra $(H,\Delta,\eps)$  which is a braided vector space such that 
the braiding is compatible with the comultiplication. In this case, combining  
\eqref{PHW} and \eqref{PVH} yields 
\[  \label{PDD}
(\Delta\ot \Delta)\circ \Psi_{HH} = (\id \ot \Psi_{HH}  \ot \id ) \circ (\Psi_{HH} \ot \Psi_{HH} ) \circ 
(\id \ot \Psi_{HH}  \ot \id )\circ (\Delta\ot \Delta). 
\]
As for algebras, each coalgebra becomes a braided coalgebra with the braiding defined 
by the flip $\tau$ given in \eqref{flip}.

The compatibility conditions permit to extend the (co)algebra structures to tensor products. 
If  $(A,m_A)$ and $(B,m_B)$ are algebras such that 
$A$ is a right $B$-braided vector space, $B$ is a left $A$-braided vector space, and 
the braiding $\Psi_{BA}: B\ot A\to A\ot B$ is compatible with the multiplications of $A$ and $B$, then 
\[  \label{ulm}
\pmb{m} : (A\ot B)\ot (A\ot B) \lra A\ot B, \quad \pmb{m}:= (m_A\ot m_B)\circ (\id\ot \Psi_{BA} \ot \id), 
\] 
defines a product on $A\ot B$ turning it into an associative algebra denoted by $A\,\uot\, B$. 
If $A$ and $B$ are unital, then $1 \ot 1$  yields the unit of $A\,\uot\, B$. 

If  $H$ and $G$ are coalgebras such that 
$H$ is a right $G$-braided vector space, $G$ is a left $H$-braided vector space, and 
the braiding $\Psi_{HG}: H\ot G\to G\ot H$ is compatible with the comultiplications of $H$ and $G$,  
then the coproduct 
\[  \label{ulD} 
\pmb{\Delta}  : H \ot G \lra  (H\ot G)\ot (H\ot G), \quad 
\pmb{\Delta}  := (\id \ot \Psi_{HG} \ot \id) \circ (\Delta \ot \Delta),  
\]
turns $H\ot G$ into a coalgebra with counit $\pmb{\eps} := \eps\ot \eps$.

Recall that, for a coalgebra $(H, \Delta,\eps)$ and a unital algebra $(A,m)$, 
the space $L(H,A)$ of all linear mappings from $H$ to $A$ becomes an 
associative unital algebra under the convolution product  
\[   \label{conv} 
(\phi\ast \psi)(h):= m\circ (\phi\ot \psi) \circ \Delta (h), \quad h\in H, \ \ \phi, \psi\in L(H,A), 
\]
and with unit $h \mapsto \eps(h) 1$. In this picture, 
the antipode of a Hopf algebra can be viewed as 
the convolution inverse of the identity map $\id : H \to H$. 

The central objects of this paper are braided bialgebras which will be defined in the definition below. 
We include there the definition of a braided Hopf algebra although the existence of an antipode will 
play rather a minor role in our presentation. 

\begin{definition}[\cite{Maj91,T}]   \label{defab}
A braided bialgebra is a unital algebra $(H,m)$ together with a coalgebra structure $(H, \Delta,\eps)$  
and a braiding  $\Psi_{HH} : H\ot H \to H$ such that the following com\-pa\-ti\-bility conditions hold: 
$H$ is a braided vector space, the braiding $\Psi_{HH}$ is compatible 
with the multiplication and the comultiplication of $H$, 
and the coproduct is an algebra homomorphism $\Delta : H \to H\,\uot\, H$, i.e.,  
\[   \label{Dcm}
\Delta \circ  m = (m\ot m) \circ (\id \ot \Psi_{HH} \ot \id) \circ (\Delta\ot \Delta). 
\] 
A braided Hopf algebra is a braided bialgebra $H$ such that the identity map has a 
convolution inverse  $S:H\to H$ called the antipode. 
\end{definition} 

For a braided Hopf algebra $H$,  it can be shown that 
\begin{align} 
\begin{split}\label{Sbraid} 
&  \Psi_{HH}  \circ (S \ot \id ) = (\id \ot S) \circ \Psi_{HH}, \quad 
\Psi_{HH}  \circ (\id \ot S) = (S \ot \id ) \circ \Psi_{HH}, \\ 
&S \circ m = m \circ  \Psi_{HH} \circ (S \ot S), \quad \Delta \circ S 
= \Psi_{HH}  \circ (S \ot S) \circ  \Delta, \quad  \eps \circ S = \eps, 
\end{split} 
\end{align} 
see e.g.\ \cite{Maj94a,T}. 

If $A$ and $B$ are braided bialgebras with a braiding $\Psi_{AB}$ satisfying the compatibility conditions 
on the multiplication and comultiplication, then $A\,\uot\, B$ becomes a 
braided bialgebra with braiding 
\[ \label{braidAB} 
\Psi_{A\ot B,\, A\ot B}  
:=   (\id \ot \Psi_{AB}\ot \id) \circ (\id\ot \id\ot\Psi_{BB} ) \circ  (\Psi_{AA}\ot \id\ot\id ) \circ (\id \ot \Psi_{AB}^{-1}\ot \id), 
\] 
and multiplication and comultiplication defined in \eqref{ulm} and \eqref{ulD}, respectively. 

The following presentation of actions and coactions in the braided setting is taken from \cite{RGG}. 
Let $(A,m)$ be a braided (unital) algebra, and let $W$ be a left $A$-braided vector space
such that the braiding is compatible with the multiplication. 
We say that $W$ is a braided left $A$-module 
if there is a map $\nu_L : A \ot W \to W$ satisfying 
\begin{align}  
& \nu_L\circ (\id \ot \nu_L) = \nu_L \circ (m \ot \id) , \quad \nu_L ( 1\ot w) = w, \ \ w\in W,  \label{anu} \\
&\Psi_{AW} \circ (\id \ot \nu_L) = (\nu_L \ot \id)\circ (\id \ot \Psi_{AW} ) \circ (\Psi_{AA} \ot \id). \label{bnu} 
\end{align}
Equation \eqref{anu} says that $\nu_L$ is an algebra action in the usual sense, 
and \eqref{bnu} means that $\nu_L$ is compatible with the braiding. 
A braided right $A$-module $V$ is defined analogously, i.e., 
$V$ is a left $A$-braided vector space, the braiding is compatible with the multiplication, and 
the right action $\nu_R : V \ot A \to V$ satisfies 
\begin{align}
& \nu_R\circ (\nu_R \ot \id) = \nu_R \circ (\id \ot m) , \quad \nu_R ( v \ot 1) = v, \ \ v\in V, \label{amu} \\
&\Psi_{VA} \circ (\nu_R \ot\id ) = (\id \ot\nu_R )\circ (\Psi_{VA} \ot\id  ) \circ (\id \ot\Psi_{AA} ).  \label{bmu} 
\end{align}
Applying the inverse braidings to the compatibility relations yields 
\begin{align} \label{nulinv}
&(\id \ot \nu_L)\circ (\Psi_{AA}^{-1} \ot \id)\circ (\id \ot \Psi_{AW}^{-1})  = \Psi_{AW}^{-1} \circ (\nu_L \ot \id), \\
&(\nu_R \ot\id )\circ (\id \ot\Psi_{AA}^{-1})\circ (\Psi_{VA}^{-1} \ot\id)  \, = \, \Psi_{VA}^{-1}\circ (\id \ot\nu_R ).
\label{nurinv}
\end{align}
To shorten notation, we often write 
\[ \label{na}
\nu_L(a \ot w)  :=  a\la w , \quad  \nu_R(v\ot a) := v \ra a, \quad w\in W, \ \, v\in V, \ \, a\in A. 
\]

Let $(H,\Delta, \eps)$ be a braided coalgebra.   
Let $V$ be a right $H$ braided vector space such that the braiding is compatible with 
the comultiplication. 
Recall that a left coaction on a vector space $V$ is a linear map 
$\rho_L : V\to H\ot V$ satisfying 
\[  \label{Droh} 
(\id\ot \rho_L )\circ \rho_L  = (\Delta \ot\id)\circ \rho_L , \quad (\vare \ot \id)\circ \rho_L = \id. 
\]
We say that the coaction is compatible with the braiding, if 
\[  \label{brL} 
(\id \ot \rho_L)\circ \Psi_{VH} =   ( \Psi_{HH}\ot \id) \circ (\id\ot \Psi_{VH} )  \circ (\rho_L \ot \id ). 
\] 
In this case, $V$ is called a braided left $H$-comodule. 

Analogously, for a braided right $H$-comodule $W$ with a right coaction $\rho_R :W\to W\ot H$, 
we require that $W$ is a left $H$ braided vector space and that 
\begin{align}   \label{rohD} 
&(\rho_R\ot \id)\circ \rho_R  = (\id \ot\Delta)\circ \rho_R , \quad ( \id \ot \vare)\circ \rho_R = \id,  \\
  \label{brR} 
&(\rho_R \ot \id)\circ \Psi_{HW} =   (\id \ot \Psi_{HH}) \circ (\Psi_{HW} \ot \id)  \circ (\id \ot \rho_R). 
\end{align}  

Note that we defined a braided \emph{left} $H$-comodule for a \emph{right} $H$ braided vector space 
and vice versa. By Lemma \ref{Linv} below, the inverse braidings 
turn a  right (resp.\ left) $H$ braided vector space into a left (resp.\ right) $H$ braided vector space. 
For the inverse braidings, the compatibility conditions read 
\begin{align}
&(\id\hsp\ot\hsp\Psi_{VH}^{-1} ) \hsp \circ\hsp ( \Psi_{HH}^{-1}\hsp\ot\hsp\id) \hsp\circ\hsp (\id\hsp\ot\hsp\rho_L) 
= (\rho_L \hsp\ot\hsp\id )\hsp\circ\hsp \Psi_{VH}^{-1},\\
 \label{invPHV} 
&(\Psi_{HW}^{-1}\hsp\ot\hsp \id)\hsp\circ\hsp (\id\hsp\ot\hsp \Psi_{HH}^{-1})\hsp\circ\hsp (\rho_R\hsp\ot\hsp \id)
= (\id\hsp\ot\hsp \rho_R)\hsp\circ\hsp \Psi_{HW}^{-1}.
\end{align}

The objects of our interest are $H$-module algebras and $H$-comodule algebras for  
a braided bialgebra $H$, so we will highlight them in a separate definition. 
\begin{definition}     \label{modalg} 
Let $H$ be a braided bialgebra and let $B$ be a braided algebra
such that  $B$ is a left  $H$-braided vector space and the 
braiding $\Psi_{HB}$  is compatible  
with the multiplications of $H$ and $B$
and with the comultiplication of $H$. 

Assume that $B$ is a braided left $H$-module with left action $\nu_L : H \ot B \to B$. 
We say that $B$ is a braided left $H$-module algebra if 
the left action $\nu_L$ and the multiplication $m_B$  of $B$ satisfy the compatibility condition 
\[ \label{num} 
\nu_L \circ (\id \ot m_B) = m_B\circ (\nu_L \ot \nu_L)\circ (\id \ot \Psi_{HB} \ot \id)\circ (\Delta \ot \id\ot \id). 
\] 
If $1\in B$, then it is additionally required that 
\[   \label{nu1} 
\nu_L (f \ot 1) = \eps(f) 1,\quad f\in H. 
\]

Assume that $B$ is a braided right $H$-comodule with right coaction $\rho_R :B\to B\ot H$. 
Then $B$ is called a braided right $H$-comodule algebra, if 
\[  \label{rhoRm} 
\rho_R \circ m_B = (m_B \ot m_H)     \circ (\id \ot \Psi_{HB} \ot \id )\circ (\rho_R \ot \rho_R),  
\]
If $1\in B$, then it is additionally required that 
\[ \label{rR1} 
\rho_R(1)=1\ot 1. 
\]

A braided right $H$-module algebra and braided left $H$-comodule algebra are defined 
analogously by flipping the tensor products and replacing $\Psi_{HB}$ by $\Psi_{BH}$. 
\end{definition}  

\begin{remark}     \label{rem1} 
The dual notions of $H$-module algebra and $H$-comodule algebra are 
$H$-module coalgebra and $H$-comodule coalgebra, respectively. The compatibility 
conditions are the dual versions of those in Definition \ref{modalg}. 
To keep the size of the paper in reasonable limits, we will not discuss these structures here. 
\end{remark}

Throughout our presentation of duality, we will make 
frequent use of the inverse of a given braiding. 
For later reference, we finish this section with a lemma 
that summarizes some properties of inverse braidings.  
It is proven by applying repeatedly the corresponding inverse morphism on 
both sides of the defining relations. 
\begin{lemma} \label{Linv}
Let $H$ be a braided vector space with braiding $\Psi_{HH}$. 
Then $\Psi_{HH}^{-1}$ defines a braiding on $H$. 
If \,$V$ carries the structure of 
a left (resp.\ right) $H$-braided vector space with respect to $\Psi_{HH}$, 
then it does so with respect to $\Psi_{HH}^{-1}$, and it becomes a 
right (resp.\ left) $H$-braided vector space with respect to $\Psi_{HV}^{-1}$ (resp.\ $\Psi_{VH}^{-1}$). 
If $H$ or $V$ is an algebra and $\Psi_{HV}$ (resp.\ $\Psi_{VH}$) is compatible with the multiplication, 
then $\Psi_{HV}^{-1}$ (resp.\ $\Psi_{VH}^{-1}$) is also compatible with the multiplication. 
If $H$  or $V$ is a coalgebra and $\Psi_{HV}$ (resp.\ $\Psi_{VH}$) is compatible with the comultiplication, 
then  $\Psi_{HV}^{-1}$ (resp.\ $\Psi_{VH}^{-1}$) is also compatible with the comultiplication. 
In particular, if $H$ is a braided (co)algebra with respect to $\Psi_{HH}$, then it is also one 
with respect to $\Psi_{HH}^{-1}$. 
\end{lemma}

\section{Braided products, coproducts and (co)actions} \label{sec3}

The purpose of this section is to discuss generalizations of opposite algebras and coalgebras 
in the braiding setting. It is also shown 
how to use the braiding, and possibly the inverse of the antipode, to turn right (co)actions into 
left (co)actions and vice versa. 

Given a braided bialgebra $H$, there exists a different braided bialgebra structure on $H$ 
with the multiplication $m$ and the coproduct $\Delta$ replaced by 
\[ 
m_1:= m\circ  \Psi_{HH} : H\ot H \lra H, \qquad  \Delta_{-1}:= \Psi_{HH}^{-1}\circ \Delta : H\lra H\ot H. 
\] 
This bialgebra will be denoted by $H^{(1,-\hsp1)}$, differing from the standard notation $H^{\mathrm{op,cop}}$ 
for reasons that will become clear below. 
If $H$ is a braided Hopf algebra, then it follows from \eqref{Sbraid} that 
$H^{(1,-\hsp1)}$ is also one with the same antipode $S$. 

The opposite algebra, say $H^{(-\hsp1,0)}$, with $m$ replaced by $m_{-1}:= m\circ  \Psi_{HH}^{-1} $ and 
the opposite coalgebra, say  $H^{(0,-\hsp1)}$, with $\Delta$ replaced by $\Delta_{-1}:= \Psi_{HH}^{-1}\circ \Delta $ 
will also yield braided bialgebras, but with respect to the inverse braiding. 
These constructions may yield an infinite family of braided bialgebras 
provided that $\Psi_{HH}$ is infinite cyclic.  
For later use, and since we did not find it in the literature, we will state the result in the following 
proposition. In the subsequent corollary, it is shown that, for a braided Hopf algebra $H$ with bijective antipode, 
$H$ is isomorphic to its $H^{\mathrm{op,cop}}$-version. 

\begin{proposition}  \label{lemmD}
Let $(H,m)$ be a braided (unital) algebra and let $k\in \Z$. 
Then 
\[ \label{mkDn1}
m_k:= m\circ  \Psi_{HH}^{k} : H\ot H \lra H, 
\] 
defines a product on $H$ turning it again into a braided (unital) algebra (with the same unit element). 
If \hs$V$ is a left or right $H$-braided vector space such that the braiding is compatible with 
the multiplication $m$ on $H$, then the braiding is also compatible with 
the multiplication~$m_k$. 

Let $(H,\Delta,\eps)$ be a braided coalgebra and let $n\in \Z$. 
Then 
\[ \label{mkDn2}
 \Delta_n: =\Psi_{HH}^{n}\circ\Delta:H\lra H\ot H, 
\]
defines a coproduct on $H$ turning it again into a braided coalgebra 
with the same counit. 
If \hs$V$ is a left or right $H$-braided vector space such that the braiding is compatible with 
the comultiplication of $H$, then the braiding is also  is compatible with 
the coproduct $\Delta_n$. 

Assume that $H$ is a braided bialgebra and 
let $H^{(k,n)}$ denote the linear space $H$ equipped with the product $m_k$, 
the unmodified unit element, the coproduct $\Delta_n$ and the unmodified counit.  
Then, for all $n\in\Z$,  $H^{(n,-n)}$ is a braided bialgebra with respect to the braiding $\Psi_{HH}$, 
and $H^{(n-1,-n)}$  is a braided bialgebra  
with respect to the braiding $\Psi_{HH}^{-1}$. 

If $H$ is a braided Hopf algebra with antipode $S$, then all $H^{(n,-n)}$ are braided 
Hopf algebras with the unmodified antipode $S$.  
If $S$ is invertible, then all $H^{(n-1,-n)}$ are braided Hopf algebras 
with antipode $S^{-1}$. 
\end{proposition} 
\begin{proof} 

We begin by showing that $m_1$ defines an associative product: 
\begin{eqnarray*} 
m_1\circ (\id \ot m_1) 
\!\!&\oeq{\eqref{mkDn1}}&\!\!  m \circ \Psi_{HH}\circ (\id \ot m) \circ (\id \ot \Psi_{HH}) \\
\!\!&\oeq{\eqref{VAm}}&\!\! m \circ (m \ot \id) \circ (\id \ot  \Psi_{HH}) \circ ( \Psi_{HH} \ot \id)\circ (\id \ot \Psi_{HH}) \\
\!\!&\oeq{\eqref{YBE}}&\!\! m \circ (\id \ot m)\circ ( \Psi_{HH} \ot \id)\circ (\id \ot \Psi_{HH})  \circ ( \Psi_{HH}\ot \id ) \\
\!\!&\oeq{\eqref{AWm}}&\!\! m \circ \Psi_{HH}\circ (m \ot \id) \circ ( \Psi_{HH}\ot \id) = m_1\circ ( m_1\ot \id). 
\end{eqnarray*} 
If $1\in H$, then clearly $m_1(1\ot h) = h = m_1(h\ot 1)$ for all $h\in H$ by \eqref{1W} and \eqref{V1}. 

To show the compatibility of $\Psi_{HH}$ with $m_1$, we compute that 
\begin{eqnarray*}
\Psi_{HH}\circ (m_1 \ot \id)   
\!\!&\oeq{\eqref{AWm},\eqref{mkDn1}}&\!\!     
(\id\ot m) \circ (\Psi_{HH} \ot \id) \circ (\id\ot \Psi_{HH}) \circ (\Psi_{HH} \ot \id) \\
\!\!& \oeq{\eqref{YBE}}&\!\!  
(\id\ot m) \circ (\id\ot \Psi_{HH}) \circ (\Psi_{HH} \ot \id) \circ (\id\ot \Psi_{HH}) \\
\!\!&\oeq{\eqref{mkDn1}}&\!\! (\id\ot m_1) \circ (\Psi_{HH} \ot \id) \circ (\id\ot \Psi_{HH}), 
\end{eqnarray*} 
which proves \eqref{AWm}. The proof of \eqref{VAm} is completely analogous. 
Moreover, if $1\in H$, then \eqref{1W} and \eqref{V1} are trivially satisfied. 
Therefore $H$ is again a braided algebra with respect to the 
multiplication $m_1$ and the braiding $\Psi_{HH}$. 

Since $m_{k+1} = m_k\circ \Psi_{HH}$, we conclude by induction that the same 
holds for all $k\in \N$. From Lemma \ref{Linv} and the previous computations, 
it follows that $m_{-1}$ turns $H$ also into a braided algebra, 
and again by induction, we obtain the result for all $m_{-k}$, \,$k\in \N$. 

Next we prove that $(H, \Delta_1,\eps)$ yields a braided coalgebra. 
The coassociativity follows from    
\begin{eqnarray*}
((\Psi_{HH}\hsp\circ\hsp \Delta)\hsp\ot\hsp \id)\circ (\Psi_{HH}\hsp\circ\hsp \Delta) 
\!\!&\oeq{\eqref{PVH}}&\!\! (\Psi_{HH} \hsp\ot\hsp \id)\circ (\id \hsp\ot\hsp \Psi_{HH})\circ (\Psi_{HH} \hsp\ot\hsp \id)\circ 
(\id \hsp\ot\hsp \Delta)\circ \Delta \\
\!\! &\oeq{\eqref{YBE}}&\!\!  (\id \hsp\ot\hsp \Psi_{HH})\circ (\Psi_{HH} \hsp\ot\hsp \id)\circ (\id \hsp\ot\hsp \Psi_{HH})\circ 
( \Delta \hsp\ot\hsp \id )\circ \Delta \\
\!\! &\oeq{\eqref{PHW}}&\!\!   (\id\hsp\ot\hsp(\Psi_{HH}\circ \Delta))\circ (\Psi_{HH}\circ \Delta). 
\end{eqnarray*} 
Furthermore, $(\id \ot \eps) \circ (\Psi_{HH}\hsp\circ\hsp \Delta) 
= \id = (\eps\ot \id)\circ (\Psi_{HH}\hsp\circ\hsp \Delta)$ by the 
second relations in \eqref{PHW} and \eqref{PVH}. 
To verify the compatibility with the braiding, we compute that 
\begin{eqnarray*}
(\id\hsp \ot \hsp \Delta_1) \hsp \circ \hsp \Psi_{HH} 
\!\! &\oeq{\eqref{PHW},\eqref{mkDn2}}&\!\!     
(\id \hsp\ot\hsp\Psi_{HH}) \hsp  \circ \hsp (\Psi_{HH}\hsp \ot \hsp\id) \hsp \circ \hsp (\id \hsp \ot  \hsp \Psi_{HH})
\hsp \circ \hsp( \Delta\hsp \ot \hsp\id) \\
\!\! &\oeq{\eqref{YBE}}&\!\! 
(\Psi_{HH}\hsp \ot \hsp\id) \hsp \circ \hsp (\id \hsp \ot  \hsp \Psi_{HH})
\hsp \circ \hsp (\Psi_{HH}\hsp \ot \hsp\id) \hsp \circ \hsp( \Delta\hsp \ot \hsp\id) \\ 
\!\! &\oeq{\eqref{mkDn2}}&\!\!  (\Psi_{HH}\hsp \ot \hsp\id) \hsp \circ \hsp (\id \hsp \ot  \hsp \Psi_{HH}) 
\hsp \circ \hsp( \Delta_1\hsp \ot \hsp\id).  
\end{eqnarray*} 
This shows the first relation of \eqref{PHW}. The first relation of \eqref{PVH} is proven analogously, and 
second relations in  \eqref{PHW} and \eqref{PVH} are trivially satisfied. 
Therefore $(H, \Delta_1,\eps)$ is a braided coalgebra. 

By Lemma \ref{Linv}, the same arguments show that 
$(H, \Delta_{-1},\eps)$ yields also a braided coalgebra. 
Similar to the above, 
since $\Delta_{k\pm 1}=\Delta_{k}\circ \Psi_{HH}^{\pm 1}$, 
we can now proceed by induction to conclude 
that $(H, \Delta_k,\eps)$ is a braided coalgebra for all $k\in \Z$. 

Let $H$ be a braided bialgebra. We first show that 
$H^{(-\hsp1,0)}$ is a braided bialgebra with respect to $\Psi_{HH}^{-1}$. 
From the first part of the proof and Lemma \ref{Linv}, we know 
that $\Psi_{HH}^{-1}$ is compatible with the multiplication $m_{-1}$ and 
the comultiplication $\Delta$, so it remains to verify \eqref{Dcm}. 
Again by Lemma \ref{Linv}, we conclude that $\Psi_{HH}^{-1}$ satisfies \eqref{PDD}.  
Therefore, 
\begin{align*} 
\Delta \circ m_{-1}   
&\oeq{\eqref{Dcm}}   
(m\ot m) \circ (\id \ot \Psi_{HH} \ot \id) \circ (\Delta\ot \Delta) \circ\Psi_{HH}^{-1} \\
 &\oeq{\eqref{PDD}}   
(m\hsp \ot\hsp  m)\hsp  \circ\hsp  ( \Psi_{HH}^{-1}\hsp\ot\hsp\Psi_{HH}^{-1})\hsp  \circ\hsp  
 (\id\hsp  \ot\hsp \Psi_{HH}^{-1} \hsp \ot\hsp  \id)\hsp \circ\hsp  (\Delta \hsp\ot\hsp\Delta) \\
 &\oeq{\eqref{mkDn1}}  (m_{-1}\ot m_{-1})\circ (\id\ot \Psi_{HH}^{-1} \ot \id)\circ (\Delta\ot\Delta). 
\end{align*} 
This finishes the proof that $H^{(-\hsp1,0)}$ is a braided bialgebra with respect to $\Psi_{HH}^{-1}$. 

To conclude the same for $H^{(0,-1)}$\hsp, note that  \eqref{Pmm} remains valid if 
we replace  $\Psi_{HH}$ by $\Psi_{HH}^{-1}$. 
Thus 
\begin{align*} 
\Delta_{-1} \circ m   
&\oeq{\eqref{Dcm}}   
\Psi_{HH}^{-1}\circ (m\ot m) \circ (\id \ot \Psi_{HH} \ot \id) \circ (\Delta \ot\Delta) \\
&\oeq{\eqref{Pmm}}   
 (m\ot m) \circ (\id \ot \Psi_{HH}^{-1} \ot \id) \circ (\Psi_{HH}^{-1} \ot \Psi_{HH}^{-1}) \circ  (\Delta \ot\Delta) \\
  &\oeq{\eqref{mkDn2}}  (m\ot m)\circ (\id\ot \Psi_{HH}^{-1} \ot \id)\circ (\Delta_{-1}\ot\Delta_{-1}),  
\end{align*} 
hence \eqref{Dcm} is satisfied. Together with the previous results, 
it follows that $H^{(0,-\hsp1)}$ is a braided bialgebra 
with respect to~$\Psi_{HH}^{-1}$. 

Now we proceed by induction. Let $n\in\N$ and assume 
that $H^{(n-1,-n)}$ is a braided bialgebra with respect to $\Psi_{HH}^{-1}$. 
From what has already been shown and since $\Psi_{HH}= (\Psi_{HH}^{-1})^{-1}$, 
we conclude that $H^{(n,-n)}$ with $m_n=m_{n-1}\circ \Psi_{HH}$ and 
$\Delta_{-n}$ is a braided bialgebra with respect to $\Psi_{HH}$. 
Likewise, if $H^{(-n,n-1)}$ is a braided bialgebra with respect to $\Psi_{HH}^{-1}$, then 
$H^{(-n,n)}$ with $m_{-n}$ and 
$\Delta_{n}=\Delta_{n-1} \circ \Psi_{HH}$ is a braided bialgebra with respect to $\Psi_{HH}$. 
Continuing in this way, if $H^{(-n,n)}$ is a braided bialgebra with respect to $\Psi_{HH}$, 
we can replace in above calculations $H$ by $H^{(-n,n)}$ and see that 
$H^{(-(n+1),n)}$ with $m_{-(n+1)}=m_{-n}\circ \Psi_{HH}^{-1}$ and $\Delta_{-n}$ 
is a  braided bialgebra with respect to $\Psi_{HH}^{-1}$. 
Finally, if $H^{(n,-n)}$ is a braided bialgebra with respect to $\Psi_{HH}$,
it follows that $H^{(n,-(n+1))}$ with $m_n$ and 
$\Delta_{-(n+1)}=\Delta_{-n} \circ \Psi_{HH}^{-1}$ 
is braided bialgebras with respect to $\Psi_{HH}^{-1}$. 
Thus the usual induction argument yields the result. 

If $H$ is braided Hopf algebra, then, by \eqref{Sbraid}, 
$$
m_n\circ (S\ot \id) \circ \Delta_{-n} 
= m\circ \Psi_{HH}^n\circ (S\ot \id) \circ  \Psi_{HH}^{-n}\circ\Delta
= \left\{ \begin{array}{l} m\circ (S\ot \id) \circ \Delta,\ \ n\in 2\Z,  \\
m\circ (\id\ot S) \circ \Delta,\ \  n\in 2\Z+1. 
\end{array} \right.
$$
Thus $m_n\circ (S\ot \id) \circ \Delta_{-n}=1\hs\eps$,  
and similarly, $m_n\circ (\id\ot S) \circ \Delta_{-n}=1\hs\eps$.  
Therefore $H^{(n,-n)}$ is a Hopf algebra with antipode $S$. 
If $S^{-1}$ exists, we obtain from \eqref{Sbraid} for $n\in 2\Z$ that 
\begin{align*}
&m_{n-1}\circ (\id\ot S^{-1}) \circ \Delta_{-n} 
= m\circ \Psi_{HH}^{n-1}\circ (\id\ot S^{-1}) \circ  \Psi_{HH}^{-n}\circ\Delta \\
&= m\circ (S^{-1}\ot \id) \circ  \Psi_{HH}^{-1}\circ\Delta 
= m\circ (\id \ot S) \circ (S^{-1}\ot S^{-1}) \circ \Psi_{HH}^{-1}\circ\Delta \\
& = m\circ (\id \ot S) \circ \Delta \circ S^{-1}  = 1\, \eps\circ S^{-1}  = 1\,\eps. 
\end{align*}  
The remaining cases, which prove that $S^{-1}$ yields an antipode for $H^{(n-1,-n)}$, 
are shown analogously. 

Finally, let $V$ be a left $H$-braided vector space. If $\Psi_{HV}$ is compatible with the 
multiplication on $H$, then 
\begin{eqnarray*} 
\Psi_{HV}\circ (m_k \ot \id) \!\!&\oeq{\eqref{AWm},\eqref{mkDn1}}&\!\!  
 (\id \ot m) \circ (\Psi_{HV} \ot \id) \circ (\id \ot \Psi_{HV}) \circ (\Psi_{HH}^k \ot \id) \\
\!\!&\oeq{\eqref{VW}}&\!\!  (\id \ot m)\circ ( \id\ot\Psi_{HH}^k)  \circ (\Psi_{HV} \ot \id) \circ (\id \ot \Psi_{HV}) \\ 
\!\!&\oeq{\eqref{mkDn1}}&\!\! (\id \ot m_k) \circ (\Psi_{HV} \ot \id) \circ (\id \ot \Psi_{HV}), 
\end{eqnarray*}
which shows the compatibility of $\Psi_{HV}$ with the multiplication $m_k$. 
Likewise, if $\Psi_{HV}$ is compatible with the comultiplication on $H$, then 
\begin{eqnarray*} 
(\id \ot \Delta_n)\circ   \Psi_{HV} \!\!&\oeq{\eqref{PHW},\eqref{mkDn2}} &\!\!  
(\id \ot \Psi_{HH}^{n}) \circ(\Psi_{HV} \ot \id)  \circ (\id \ot \Psi_{HV}) \circ  (\Delta \ot \id) \\
\!\!&\oeq{\eqref{VW}}&\!\! (\Psi_{HV} \ot \id)  \circ (\id \ot \Psi_{HV}) \circ (\Psi_{HH}^{n}\ot\id)\circ(\Delta \ot \id) \\
\!\!&\oeq{\eqref{mkDn2}}&\!\!  (\Psi_{HV} \ot \id)  \circ (\id \ot \Psi_{HV}) \circ  (\Delta_n \ot \id)
\end{eqnarray*} 
proves the compatibility of $\Psi_{HV}$ with the comultiplication $\Delta_n$.  
The relations regarding the unmodified unit or counit remain trivially true.  
The proof for a right $H$-braided vector space is completely analogous. 
\end{proof} 

Since, by \eqref{Sbraid}, $(S\ot S) \circ \Psi_{HH} = \Psi_{HH}\circ (S\ot S)$ and 
$$
S^k \circ m = m \circ  \Psi_{HH}^k \circ (S^k\ot S^k) = m_k \circ (S^k\ot S^k), \ \ 
(S^k\ot S^k) \circ \Delta =  \Psi_{HH}^{-k} \circ \Delta \circ S^k = \Delta_{-k} \circ S^k, 
$$ 
we obtain immediately the following corollary. 

\begin{corollary}  \label{corS} 
Let $H$ be a braided Hopf algebra and $n\in \Z$.  Then the antipode $S$ defines 
braided  Hopf algebra homomorphisms $S^k : H^{(n,-n)} \to H^{(n+k,-(n+k))}$
and braided bialgebra homomorphisms  
$S^k : H^{(n-1,-n)} \to H^{(n+k-1,-(n+k))}$, 
where $k\in \Z$ \hs if \hs $S$ is invertible and $k\in \N$ otherwise. 
For invertible $S$, all these homomorphisms are isomorphisms of braided Hopf algebras.  
\end{corollary}

In the unbraided case, a left  action of an algebra yields a right  action 
of the opposite algebra with flipped multiplication, and a 
left coaction of a coalgebra defines a right coaction 
of the opposite coalgebra with the flipped coproduct. 
Evidently, the same holds if left and right are interchanged. 
However, the usual flip is in general 
not compatible with the braiding.  
A proper version in the braided setting is given in the next proposition. 

\begin{proposition}  \label{biprop} 
Let $(H,m)$ be a braided algebra and $V$ a braided left $H$-module. 
Then 
\[ \label{nuRc} 
\nu_R^\circ  : V\ot H \lra H, \quad \nu_R^\circ  :=  \nu_L  \circ  \Psi_{HV}^{-1} 
\] 
turns $V$ into a braided right $H$-module with respect to the 
multiplication $m_{-1}:= m\circ  \Psi_{HH}^{-1}$ and the braiding $\Psi_{HH}^{-1}$ on $H$. 
If $H$ is a braided bialgebra and $V$ is a braided left $H$-module algebra, 
then $\nu_R^\circ$ transforms $V$ into a right $H^{(-1,0)}$-module algebra. 

Analogously, if $V$ is a braided right $H$-module, then 
$$
\nu_L^\circ  : H\ot V \lra H, \quad \nu_L^\circ  :=  \nu_R  \circ  \Psi_{VH}^{-1} 
$$ 
turns $V$ into a braided left $H$-module with respect to the multiplication 
$m_{-1}$ and the braiding $\Psi_{HH}^{-1}$ on $H$.  
Furthermore, a braided right $H$-module algebra becomes a 
left $H^{(-1,0)}$-module algebra. 

Given a coalgebra $H$ with coproduct $\Delta$ and a braided right $H$-comodule $V$, 
the left coaction 
$$
\rho_L^\circ : V\lra H\ot V, \quad \rho_L^\circ := \Psi_{HV}^{-1}  \circ \rho_R 
$$
turns $V$ into a braided left $H$-comodule with respect to the 
coproduct $\Delta_{-1}:= \Delta\circ  \Psi_{HH}^{-1}$ and the braiding $\Psi_{HH}^{-1}$ on $H$. 
If $H$ is a braided bialgebra and $V$ is a braided right $H$-comodule algebra, 
then $\rho_L^\circ$ transforms $V$ into a left $H^{(0,-1)}$-comodule algebra. 

For a braided left $H$-comodule $V$, the right coaction 
$$
\rho_R^\circ : V\lra V\ot H, \quad \rho_R^\circ := \Psi_{VH}^{-1}  \circ \rho_L 
$$
turns $V$ into a braided left $H$-comodule with respect to the coproduct $\Delta_{-1}$ 
and the braiding $\Psi_{HH}^{-1}$ on $H$, 
and a braided left $H$-comodule algebra becomes a 
right $H^{(0,-1)}$-comodule algebra.  
\end{proposition} 
\begin{proof}
The compatibility of the inverse braidings with algebraic structures can be deduced from Lemma \ref{Linv}. 
In particular, the braiding $\Psi_{HV}^{-1}$ turns a $V$ into right $H$-braided vector space and, 
by Proposition \ref{lemmD}, 
is compatible with the multiplication $m_{-1}$ . Since 
\begin{eqnarray*} 
\nu_R^\circ \circ (\nu_R^\circ \ot \id ) \!\!&\oeq{\eqref{nuRc}} &\!\!
\nu_L  \circ  \Psi_{HV}^{-1}  \circ (\nu_L \ot \id )\circ (\Psi_{HV}^{-1} \ot \id)  \\
\!\!&\oeq{\eqref{nulinv}} &\!\!
\nu_L  \circ (\id \ot \nu_L )\circ (\Psi_{HH}^{-1} \ot \id)\circ  (\id\ot \Psi_{HV}^{-1}) \circ (\Psi_{HV}^{-1} \ot \id) \\
\!\!&\oeq{\eqref{VW},\eqref{anu}}&\!\!
 \nu_L  \circ (m\ot \id) \circ  (\id\ot \Psi_{HV}^{-1}) \circ (\Psi_{HV}^{-1} \ot \id) \circ ( \id \ot\Psi_{HH}^{-1})\\
 \!\!&\oeq{\eqref{AWm}}&\!\!
  \nu_L\circ \Psi_{HV}^{-1} \circ (\id \ot m)\circ  ( \id \ot\Psi_{HH}^{-1}) 
  \oeq{\eqref{nuRc}} \nu_R^\circ \circ (\id \ot m_{-1}), 
\end{eqnarray*} 
and $\nu_R^\circ (v\ot 1) = v$ by \eqref{1W} and \eqref{anu}, 
it follows that $\nu_R^\circ $ defines a right $H$-action with respect to the multiplication $m_{-1}$. 
Moreover, 
\begin{eqnarray*} 
\Psi_{HV}^{-1}\circ (\nu_R^\circ \ot \id ) 
\!\!&\oeq{\eqref{nuRc}}&\!\! \Psi_{HV}^{-1}\circ (\nu_L \ot \id )\circ (\Psi_{HV}^{-1} \ot \id) \\
\!\!&\oeq{\eqref{nulinv}}&\!\!
(\id \ot \nu_L )\circ (\Psi_{HH}^{-1} \ot \id)\circ  (\id\ot \Psi_{HV}^{-1})\circ (\Psi_{HV}^{-1} \ot \id) \\
\!\!&\oeq{\eqref{VW},\eqref{nuRc}}&\!\!
(\id \ot \nu_R^\circ )\circ (\Psi_{HV}^{-1} \ot \id) \circ ( \id \ot\Psi_{HH}^{-1})
\end{eqnarray*} 
proves \eqref{bmu}. Therefore $\nu_R^\circ$ equips $V$ 
with the structure of a braided right $H$-module with respect to the multiplication $m_{-1}$ 
and the braidings $\Psi_{HH}^{-1}$ and $\Psi_{HV}^{-1}$. 

If $H$ is a braided bialgebra and $V$ is a braided left $H$-module algebra, then 
\begin{align*}
&\nu_R^\circ  \circ (m_V\ot \id) \oeq{\eqref{VAm},\eqref{nuRc}}
\nu_L  \circ (\id \ot m_V) \circ (\Psi_{HV}^{-1} \ot \id)\circ (\id\ot \Psi_{HV}^{-1})\\ 
& \oeq{\eqref{num}}
 m_V\circ (\nu_L \ot \nu_L)\circ (\id \ot \Psi_{HV} \ot \id)\circ (\Delta \ot \id\ot \id) 
 \circ (\Psi_{HV}^{-1} \ot \id)\circ (\id\ot \Psi_{HV}^{-1})\\
& \oeq{\eqref{PHW}} 
m_V\circ (\nu_L \ot \nu_L)\circ (\Psi_{HV}^{-1}\ot \Psi_{HV}^{-1})
\circ  (\id\ot \Psi_{HV}^{-1}\ot \id) \circ  (\id\ot \id \ot \Delta) \\
&\oeq{\eqref{nuRc}} 
m_V\circ (\nu_R^\circ \ot \nu_R^\circ) 
\circ  (\id\ot \Psi_{HV}^{-1}\ot \id) \circ  (\id\ot \id \ot \Delta).  
\end{align*} 
This shows that $\nu_R^\circ$ satisfies the compatibility condition of  a 
braided right $H$-module algebra with respect to the unmodified coproduct $\Delta$ 
and the braiding $\Psi_{HV}^{-1}$, i.e., 
$\nu_R^\circ$ equips $V$ with the structure of a braided right $H^{(-1,0)}$-module algebra. 

The proof of the opposite version and the proofs for the coactions are similar and left to the reader. 
\end{proof}

In the last proposition, 
we had to replace the product of the braided algebra by the opposite one
in order to interchange left and right actions. 
If $H$ is a braided Hopf algebra with bijective antipode, 
we can use the anti\-pode to turn a left action into a right action of the 
\emph{same} algebra, but with a modified coproduct. To see this, it suffices to observe that, 
if $V$ is a braided left $H$-module and $\varphi: H_0\to H$ is a Hopf algebra homomorphism, 
then  $V$ becomes a braided left $H_0$-module in the obvious way. Thus, setting 
$H_0:= H^{(1,-1)}$ and $\varphi:= S^{-1} : H^{(1,-1)} \to H$, we obtain from Corollary \ref{corS} 
and Proposition \ref{biprop} a right action of the Hopf algebra $H^{(0,-1)}$ 
with the unmodified product, the coproduct $\Delta_{-1}$ and the braiding $\Psi_{HH}^{-1}$. 
Similar arguments can be applied to right actions and left or right coactions. 
We summarize these observations in the next corollary for (co)module algebras. 
\begin{corollary} 
Let $H$ be a braided Hopf algebra with invertible antipode $S$. 
If $V$ is a left $H$-module algebra, then the right action 
$$
\nu_{R,S} : V\ot H \lra V, \quad \nu_{R,S}:=  \nu_L  \circ  \Psi_{HV}^{-1} \circ (\id \ot S^{-1}),  
$$
turns $V$ into a right $H^{(0,-1)}$-module algebra. Analogously, 
a right $H$-module algebra $V$ becomes a left $H^{(0,-1)}$-module algebra 
for the left action defined by 
$$
\nu_{L,S} : V\ot H \lra V, \quad \nu_{L,S}:=  \nu_R  \circ  \Psi_{VH}^{-1} \circ (S^{-1}\ot\id).  
$$

Given a right $H$-comodule algebra $V$, the left coaction 
$$
\rho_{L,S} : V\lra H\ot V, \quad \rho_{L,S} :=(S^{-1}\ot \id)\circ  \Psi_{HV}^{-1}  \circ \rho_R, 
$$
turns $V$ into a left $H^{(-1,0)}$-comodule algebra,  
and a left $H$-comodule algebra $V$ 
becomes a right $H^{(-1,0)}$-comodule algebra for the right coaction defined by 
$$
\rho_{R,S} : V\lra V\ot H, \quad \rho_{R,S} := (\id\ot S^{-1})\circ  \Psi_{VH}^{-1}  \circ \rho_L. 
$$
\end{corollary}

\section{Duality for infinite-dimensional braided algebras, coalgebras, bialgebras and Hopf algebras}  
\label{sec4}

This section provides a detailed description of duality for braided algebras, braided coalgebras,  
and both structures together, i.e., braided bialgebras and Hopf algebras. 
As dual objects may not exist in a braided monoidal category, 
specifically in the infinite-dimensional setting (cf.\ \cite{T99}), 
we continue with our non-categorical approach. That is, we assume 
the existence of a dual space with certain properties without proving 
its existence or uniqueness, which means that we will not define a functor into a dual category. 
Furthermore, our definitions will be rather constructive in the sense that 
they are expressed by explicit formulas derived from the given structures. 

A dual pairing between two vector spaces $U$ and $H$ is a linear map 
$\ip{\cdot}{\cdot} : U\ot H \lra \K$. 
Let $H'$ denote the dual space of $H$. 
Given a subspace $U\subset H'$, we define a dual pairing 
between $U$ and $H$ by 
\[  \label{ipdot}
\ip{\cdot}{\cdot} : U\ot H \lra \K, \quad \ip{f}{a}:= f(a). 
\] 
Identifying by a slight abuse of notation $H$ with 
its image $\iota(H) \subset H''$ under 
the cannonical embedding $\iota : H\to H''$, $\iota(a)(f):= f(a)$, 
the dual pairing \eqref{ipdot} becomes symmetric in the sense that 
$\ip{f}{a}=\ip{a}{f}$. 
A subspace $U\subset H'$ is called non-degenerate, 
or synonymously the dual pairing is called non-degenerate, 
if the associated bilinear map 
$\ip{\cdot\,}{\cdot} : U\times H \to \K$ is non-degenerate. 

The dual pairing $\ip{\cdot}{\cdot}$ defined in \eqref{ipdot} is actually the restriction 
of the fundamental evaluation map $\mathrm{ev}: H'\ot H \to \K$, 
$\mathrm{ev}(f\ot a):= f(a)$. The problem in the infinite setting is
that the so called coevaluation map $\mathrm{coev} : \K \to H\ot H'$ 
may not exist, see e.g.\ \cite{T99}. In other words, the 
braided monoidal category may not be rigid. Nevertheless, the evaluation map, 
or rather its restriction $\ip{\cdot}{\cdot}$, will play a fundamental role 
in the dual pairing between tensor spaces. 
In particular, the dual pairing between $n$-fold tensor product spaces 
will  entirely be traced back to the 
evaluation map on adjacent tensor factors. That is,  
given linear spaces $H_j$ and subspaces 
$U_j\subset H_j'$, $j=1,\ldots, n$, we define 
\begin{align}  \label{braidip} 
\begin{split}
&\ipp{\,\cdot\,}{\hs\cdot\,}:  (U_n\ot \dots \ot  U_1)\ot (H_1\ot  \dots \ot H_n) \lra \K, \\
& \ipp{\,\cdot\,}{\hs\cdot\,} 
 := \ip{\,\cdot\,}{\,\cdot\,} \circ \big(\id \ot \ip{\,\cdot\,}{\,\cdot\,} \ot \id\big) 
\circ  \dots \circ \big( \id \ot \dots \ot \ip{\,\cdot\,}{\,\cdot\,} \ot \dots \ot \id\big).   
\end{split}
\end{align} 
This definition is consistent with the representation of braided monoidal categories 
by braided strings. Once a convention for a dual pairing between 
tensor product spaces  is agreed upon, it should be avoided 
to use isomorphisms between tensor spaces in the dual pairing that do not arise from braidings. 
For instance, to pair the second leg in $H'\ot H'$ with the second leg 
in $H\ot H$, it is more appropriate to apply first a braiding $\Psi_{H'H}$ and 
to consider 
$$
(\ip{\cdot\,}{\cdot} \ot \ip{\cdot\,}{\cdot}) \circ (\id \ot \Psi_  {H'H} \ot \id) : 
H'\ot H' \ot H\ot H \lra \K. 
$$
On the other hand, we will also make use of the 
embedding $U_1\ot \cdots \ot U_n\subset (H_1\ot  \dots \ot H_n)'$ 
(mind the order). In this case, we write 
$$
(f_1\ot \cdots \ot f_n)(a_1\ot \cdots \ot a_n) := f_1(a_1)\cdots f_n(a_n).  
$$ 

Now let $H$ be a braided vector space with braiding $\Psi_{HH}$. 
Our first aim is to show that, for appropriate subspaces $U\subset H'$, 
$\Psi_{HH}$ induces braidings on $U \ot U$ and between $U$ and $H$. 
Moreover, the braidings between $U$ and $H$ will be compatible 
with the multiplication and comultiplication on $H$ if these structures  
are compatible with $\Psi_{HH}$. According to our non-categorical approach, we will 
not assume that $U$ is unique nor prove that it always exists.  

To begin, consider the linear map 
\[  \label{UU}
\Psi_{H'H'}:  H' \ot  H'\to (H\ot H)', \ \ 
\Psi_{H'H'}(f\ot g)(b\ot a):= \ipp{f\ot g}{\Psi_{HH}(a\ot b)}, 
\]
where $a,b\in H$ and $f,g\in H'$.  
Note that we do not assume that $\Psi_{H'H'}(H' \ot  H') \subset H' \ot  H'$. 
Similarly, using the fact that the canonical pairing $\ip{\cdot\,}{\cdot} : H'\ot H \to \K$ is non-degenerate, 
we define  
\begin{align}    \label{PsiHH}
&\Psi_{H'H}:  H' \ot  H \lra (H' \ot  H)' ,\ \  
\Psi_{H'H}(g\!\ot\! a)(f\!\ot\! b):= \ipp{f\!\ot\! g}{\Psi_{HH}(a\!\ot\! b)},  \\
&\Psi_{H'H}^\circ :  H' \ot  H \lra  (H' \ot  H)', \ \ 
\Psi_{H'H}^\circ (g\!\ot\! a)(f\!\ot\! b):= \ipp{f\!\ot\! g}{\Psi_{HH}^{-1}(a\!\ot\! b)},\label{PsiHHcirc}\\ 
&\Psi_{HH'}:  H \ot  H' \lra (H \ot  H')', \ \ 
\Psi_{HH'}(b\ot f)(a\ot g):= \ipp{f\ot g}{\Psi_{HH}(a\ot b)} ,  \\  \label{PHU}
&\Psi_{HH'}^\circ :  H \ot  H' \lra (H \ot  H')', \ \ 
\Psi_{HH'}^\circ (b\ot f)(a\ot g):= \ipp{f\ot g}{\Psi_{HH}^{-1}(a\ot b)}.  
\end{align} 

It will be convenient to introduce some Sweedler-type notation. 
As usual, a coproduct is written $\Delta(a) = a_{(1)} \ot a_{(2)}$ with increasing numbers for 
multiple coproducts. Analogously, left and right coactions are written 
$\rho_L(v) = v_{(-1)} \ot v_{(0)}$ and $\rho_R(v) = v_{(0)} \ot v_{(1)}$, respectively. 
For a given braiding on a vector space $H$, we employ the notation  
$$
\Psi_{HH}(a\ot b) = b^{\br{1}} \ot  a^{\br{2}} , \quad a,b\in H. 
$$ 
For multiple braidings, an index will be used to indicate the chronological order. 
As an example, a combination of \eqref{YBE} and \eqref{PHW} gives for $a$ and $b$ from a braided coalgebra 
\[ \label{bcop}
b^{\br{1}_1}\ot a^{\br{2}_1}{}_{(2)}{}^{\br{1}_2}\ot a^{\br{2}_1}{}_{(1)}{}^{\br{2}_2} 
=  b^{\br{1}_2\br{1}_3}   \ot   a_{(2)}{\!}^{\br{1}_1\br{2}_3}  \ot a_{(1)}{\!}^{\br{2}_1\br{2}_2}. 
\]
We use a back-prime to denote the inverse of  $\Psi_{HH}$, i.e., 
\[ \label{bpr} 
\Psi_{HH}^{-1}(a\ot b) = b^{\pbr{1}} \ot  a^{\pbr{2}} , \quad a,b\in H. 
\]
Then clearly  
\[    \label{brpbr} 
   a^{\br{2} \pbr{1}}  \ot  b^{\br{1} \pbr{2}}  = a \ot  b = a^{\pbr{2} \br{1}}  \ot  b^{\pbr{1} \br{2}}, \quad a,b\in H. 
\]
In a similar vein,  \eqref{YBE} and  \eqref{brpbr} yield the identity 
\[      \label{YBe}
c^{\br{1}\pbr{1}_1} \ot b^{\br{2} \pbr{1}_2} \ot a^{\pbr{2}_1\pbr{2}_2}  
= c^{\pbr{1}_2\br{1}} \ot b^{\pbr{1}_1\br{2}} \ot a^{\pbr{2}_1 \pbr{2}_2}, 
\quad  a, b,c\in  H. 
\] 
The same notations will be used for subspaces $U\subset H'$ such that $\Psi_{H'H'}(U\ot U) \subset U\ot U$. 
Then \eqref{UU} reads for instance 
\[ \label{fagb}
f^{\br{2}}(a)\hs g^{\br{1}}(b) = g(b^{\br{1}})\hs f(a^{\br{2}}), \quad  f,g\in U, \ \ a,b\in H. 
\]
If $\Psi_{H'H}(g\ot a),\,\Psi_{H'H}^\circ (g\ot a)\in H\ot U\subset (H'\ot H)'$ 
for $g\ot a\in U\ot H$, we write 
\[ \label{cb}
\Psi_{H'H}(g\ot a) := a^{\Br{1}} \ot  g^{\Br{2}}, \qquad 
\Psi_{H'H}^\circ(g\ot a) := a^{\Br{1}^\circ} \ot g^{\Br{2}^\circ}, 
\]
and a similar notation will be employed for $\Psi_{HH'}$ and $\Psi_{HH'}^\circ$. 
Under the assumption that all maps belong to the tensor products of the corresponding spaces, 
\eqref{UU}--\eqref{PHU} yield in Sweedler-type notation 
\begin{align}       \label{phh}
&g(b^{\br{1}})\hs f(a^{\br{2}}) =  g^{\br{1}}(b)\hs f^{\br{2}}(a) = f(a^{\Br{1}}) \hs  g^{\Br{2}}(b) 
= f^{\Br{1}}(a) \hs  g(b^{\Br{2}}),  \\ 
& g(b^{\pbr{1}})\hs f(a^{\pbr{2}}) = g^{\pbr{1}}(b)\hs f^{\pbr{2}}(a)
= f(a^{\Br{1}^\circ}) \hs  g^{\Br{2}^\circ}\!(b) = f^{\Br{1}^\circ}\!(a) \hs  g(b^{\Br{2}^\circ}) .\label{cphh}
\end{align}
Furthermore, if $U\subset H'$ is non-degenerate, 
we conclude from \eqref{phh} and \eqref{cphh} that 
\begin{align} \label{phcph}
&f(a^{\Br{1}}) \hs  g^{\Br{2}} =f^{\br{2}}\hsp (a)\hs g^{\br{1}}, &
&a^{\Br{1}}\hs  g^{\Br{2}}(b) = g(b^{\br{1}})\, a^{\br{2}}, \\
&f(a^{\Br{1}^\circ}) \hs  g^{\Br{2}^\circ}= f^{\pbr{2}}\hsp (a)\hs g^{\pbr{1}}, &
&a^{\Br{1}^\circ} \hs  g^{\Br{2}^\circ}(b) = g(b^{\pbr{1}})\, a^{\pbr{2}} ,   \label{ophcph}
\end{align}
for all $f,g\in U$ and $a,b\in H$. 

The next lemma shows that, under suitable conditions on $U\subset H'$, the 
braiding $\Psi_{HH}$ induces braidings on $U \ot U$ and between $U$ and $H$ 
which are compatible with possibly additional structures on $H$.

\begin{lemma}   \label{lembraided} 
Let $H$ be a braided vector space and $U\subset H'$ a non-degenerate subspace. 
Assume that 
\[ \label{pbbraid}
\Psi_{UU}\!:=\! \Psi_{H'H'}\!\!\upharpoonright_{U\ot \hs U}\;  : U\ot U \lra  U\ot U 
\] 
is bijective. Then $\Psi_{UU}$ defines a braiding on $U$. If 
\[
\Psi_{UH}\!:=\! \Psi_{H'H} \!\!\upharpoonright_{U\ot H} \; :  U\ot H\lra H\ot U  \label{brUH}
\]
is bijective, then $\Psi_{UH}$ turns $H$ into a left $U$-braided vector space and 
$U$ into a right $H$-braided vector space 
with respect to the braidings 
$\Psi_{HH}$ and $\Psi_{HH}^{-1}$ on $H$, and $\Psi_{UU}$ and 
$\Psi_{UU}^{-1}$ on $U$. 

In case 
\[  \label{ipbb}
\Psi_{UH}^\circ \!:=\! \Psi_{H'H}^\circ \!\!\upharpoonright_{U\ot H} \; :  U\ot H\lra H\ot U 
\] 
is bijective, it also turns $H$ into a left $U$-braided vector space and 
$U$ into a right $H$-braided vector space with respect to the braidings 
$\Psi_{HH}$ and $\Psi_{HH}^{-1}$ on $H$, and $\Psi_{UU}$ and 
$\Psi_{UU}^{-1}$ on $U$. 

If $H$ is a braided (unital) algebra, then the braidings $\Psi_{UH}$ and $\Psi_{UH}^\circ$   
are compatible with the multiplication on $H$. 
If $H$ is a braided coalgebra, then  the braidings  $\Psi_{UH}$ and $\Psi_{UH}^\circ$  
are compatible with the comultiplication on $H$. 

The analogous statements hold for the opposite versions with respect to the braidings  
$$
\Psi_{HU}\!:=\! \Psi_{HH'} \!\!\upharpoonright_{H\ot\hs U} \; :  H\ot U\lra U\ot H, 
$$
and
$$
\Psi_{HU}^\circ \!:=\! \Psi_{HH'}^\circ \!\!\upharpoonright_{H\ot \hs U} \; :  H\ot U\lra U\ot H. 
$$
 \end{lemma} 
\begin{proof} Let $f\ot g\ot h\in U\ot U\ot U$. 
Since $H\ot H \ot H$ separates the points of $U\ot U\ot U\subset H'\ot H' \ot H'$, we can 
prove \eqref{YBE} by evaluating both sides on all $x\ot y\ot z\in H\ot H \ot H$. 
In Sweedler-type notation, we get 
\begin{align}  \label{ib} 
\begin{split}
&h^{\br{1}_2 \br{1}_3 }(x) \hs g^{\br{1}_1\br{2}_3 }(y)\hs f^{\br{2}_1\br{2}_2 }(z) 
\oeq{\eqref{fagb}} 
h^{  }(x^{\br{1}_1\br{1}_2})\hs g^{ }(y^{\br{2}_1\br{1}_3})\hs f(z^{\br{2}_2\br{2}_3}) \\
&\oeq{\eqref{YBE}} 
h(x^{\br{1}_2\br{1}_3})\hs g^{ }(y^{\br{1}_1\br{2}_3})\hs f(z^{\br{2}_1\br{2}_2})
\oeq{\eqref{fagb}}  
h^{\br{1}_1 \br{1}_2 }(x) \hs g^{\br{2}_1\br{1}_3 }(y)\hs f^{\br{2}_2\br{2}_3 }(z). 
\end{split}
\end{align} 
As $\ip{\cdot\,}{\cdot} : U\ot H \to \K$ is non-degenerate, 
we conclude that $\Psi_{UU}$ satisfies \eqref{YBE}, and so does $\Psi_{UU}^{-1}$ 
by Lemma \ref{Linv}. 

To prove \eqref{VW}, we use again the non-degeneracy of the pairing $\ip{\cdot\,}{\cdot} : U\ot H \to \K$ 
and evaluate $a^{\Br{1}_1 \Br{1}_2 }\ot g^{\br{1}\Br{2}_2 } \ot f^{\br{2}\Br{2}_1 }\in H\ot U\ot U$ on all 
$h\ot y\ot z\in U\ot H\ot H$, where $a\in H$  and $f,g\in U$. 
This gives 
\begin{align} \label{ibr}
&h(a^{\Br{1}_1 \Br{1}_2 })\hs g^{\br{1}\Br{2}_2 }(y)\hs f^{\br{2}\Br{2}_1 }(z) 
\oeq{\eqref{phh}}  h^{\br{2}_2}(a^{\br{2}})\hs g^{\br{1}_1\br{1}_2 }(y)\hs f^{\br{2}_1}(z^{\br{1} }) \\ \nonumber
&\oeq{\eqref{fagb}}   h^{}(a^{\br{2}_1 \br{2}_2})\hs g^{}(y^{\br{1}_2 \br{1}_3 })\hs f^{}(z^{\br{1}_1\br{2}_3 }) 
\oeq{\eqref{YBE}}   h^{}(a^{\br{2}_2 \br{2}_3})\hs g^{}(y^{\br{1}_1 \br{1}_2 })\hs f^{}(z^{\br{2}_1\br{1}_3 }) \\ \nonumber
&\oeq{\eqref{fagb}}   h^{\br{2}}(a^{\br{2}_2 })\hs g^{}(y^{\br{1}_1 \br{1}_2 })\hs f^{\br{1}}(z^{\br{2}_1}) 
\oeq{\eqref{phh}}  h^{}(a^{\Br{1}_1\Br{1}_2})\hs g^{\Br{2}_1}(y^{\br{1}})\hs f^{\Br{2}_2}(z^{\br{2}}) \\\nonumber
&\oeq{\eqref{fagb}}  h^{}(a^{\Br{1}_1\Br{1}_2})\hs g^{\Br{2}_1 \br{1}}(y^{})\hs f^{\Br{2}_2 \br{2}}(z^{}) . 
\end{align} 
Since $U\ot H\ot H$ separates the points of $H\ot U\ot U$, these calculations show that \eqref{VW}
is satisfied. 

Much in the same way, for all $f\ot a\ot b\in U\ot H \ot H$ and $g\ot h \ot z\in U\ot U\ot H$, we compute 
\begin{align*}
& g(b^{\br{1}\Br{1}_1}) \, h(a^{\br{2}\Br{1}_2}) \, f^{\Br{2}_1 \Br{2}_2}(z) 
 \oeq{\eqref{phh}} g(b^{\br{1}_1\br{2}_3}) \, h(a^{\br{2}_1\br{2}_2}) \, f^{}(z^{\br{2}_2\br{1}_3 }) \\
&\oeq{\eqref{YBE}} g(b^{\br{2}_1\br{1}_3}) \, h(a^{\br{2}_2\br{2}_3}) \, f^{}(z^{\br{1}_1\br{1}_2 }) 
\oeq{\eqref{phh} } g^{\br{1}}(b^{\Br{1}_2}) \, h^{\br{2}}(a^{\Br{1}_1}) \, f^{\Br{2}_1\Br{2}_2 }(z^{}) \\
&\oeq{\eqref{fagb}} g^{}(b^{\Br{1}_2\br{1}}) \, h^{}(a^{\Br{1}_1\br{2}}) \, f^{\Br{2}_1\Br{2}_2 }(z^{}) , 
\end{align*} 
which proves \eqref{WV}. 
This finishes the proof of first part of the lemma regarding the braidings 
$\Psi_{HH}$, $\Psi_{UU}$ and $\Psi_{UH}$. By Lemma  \ref{Linv}, 
the same holds with respect to the braidings 
$\Psi_{HH}^{-1}$, $\Psi_{UU}^{-1}$ and $\Psi_{UH}$. 
Replacing in above calculations $\Psi_{HH}$ by $\Psi_{HH}^{-1}$ and $\Psi_{UU}$ by 
$\Psi_{UU}^{-1}$  shows the analogous results for $\Psi_{UH}^\circ$

Let $H$ be a braided algebra. 
Using the fact that $\Psi_{HH}$ satisfies \eqref{AWm} and \eqref{1W}, we get 
for all $f,g\in U$ and $a,b\in H$, 
\begin{align} \begin{split} \label{psicomp}
f((ab)^{\Br{1}})\hs g^{\Br{2}}(c)&\oeq{\eqref{phh}}f((ab)^{\br{2}})\hs g(c^{\br{1}})
\oeq{\eqref{AWm}} f(a^{\br{2}_2}b^{\br{2}_1})\hs g(c^{\br{1}_1\br{1}_2})\\
&\oeq{\eqref{phh}} f(a^{\Br{1}_1}b^{\Br{1}_2})\hs g^{\Br{2}_1\Br{2}_2}(c)
\end{split}
\end{align} 
and 
\[ \label{1comp}
f(1^{\Br{1}}) g^{\Br{2}}(a) \oeq{\eqref{phh}} f(1^{\br{2}}) g(a^{\br{1}}) 
\oeq{\eqref{1W}} f(1) g(a). 
\]
This implies the compatibilty of $\Psi_{UH}$ with the multiplication on $H$. 

By Lemma \ref{Linv}, 
$\Psi_{HH}^{-1}$ is also compatible with the multiplication and on $H$. 
Replacing $\Br{k}$ by $\Br{k}^\circ$ and 
$\br{k}$ by $\pbr{k}$ in \eqref{psicomp} and \eqref{1comp} shows the compatibilty of
$\Psi_{UH}^\circ$ with the multiplication on $H$. 
Similarly, if $H$ is a braided coalgebra, we obtain 
for $f,g,h\in U$ and $a,b\in H$ that 
\begin{align*}
f(a^{\Br{1}}{\!}_{(1)} )\hs   g(a^{\Br{1}}{\!}_{(2)})\hs    h^{\Br{2}}(b) 
&\oeq{\eqref{phcph}} f(a^{\br{2}}{\!}_{(1)} )\hs   g(a^{\br{2}}{\!}_{(2)})\hs    h(b^{\br{1}}) 
\oeq{\eqref{PHW}} f(a_{(1)}{\!}^{\br{2}_2})\hs   g(a_{(2)}{\!}^{\br{2}_1})\hs    h(b^{\br{1}_1\br{1}_2}) \\
&\oeq{\eqref{phh}} f(a_{(1)}{\!}^{\Br{1}_1})\hs   g(a_{(2)}{\!}^{\Br{1}_2})\hs    h^{\Br{2}_1\Br{2}_2}(b)
\end{align*} 
and 
$$
\vare(a^{\Br{1}})\hs  h^{\Br{2}}(b) \oeq{\eqref{phcph}} \vare(a^{\br{2}})\hs h(b^{\br{1}}) 
\oeq{\eqref{PHW}} \vare(a)\hs h(b), 
$$
which proves \eqref{PVH} for $\Psi_{UH}$. 

The same proof with the notational changes mentioned above 
shows the compatibility of $\Psi_{UH}^\circ$ with the multiplication or 
the comultiplication (as applicable) on $H$. 
The statements of the opposite versions are proven analogously. 
\end{proof}

The following definition of a dual pairing between braided bialgebras 
is the central definition of this section because it will also serve as a  
guiding principle for duality between braided algebras and braided coalgebras. 
Similar definitions can be found in \cite{RGG,GZ,HZ,Maj96}.  
\begin{definition}    \label{dualp}
Let $U$ and $H$ be braided bialgebras and let $\Upsilon_{UH} : U \ot  H\to H\ot U$ be a braiding such that 
$H$ is a left $U$-braided vector space, $U$ is a right $H$-braided vector space, and the braiding is 
compatible with the multiplications and comultiplications of $U$ and $H$. 
A dual pairing between $U$ and $H$  is a linear map $\ip{\cdot\,}{\cdot} : U\ot H \to \K$ such that 
\begin{align}   \label{mD}
&\ip{\cdot\,}{\cdot} \circ (m \ot \id ) 
= (\ip{\cdot\,}{\cdot} \ot \ip{\cdot\,}{\cdot}) \circ (\id \ot \Upsilon_{UH} \ot \id) \circ (\id\ot \id\ot \Delta), \\
&\ip{\cdot\,}{\cdot} \circ (\id \ot m )     \label{Dm}
= (\ip{\cdot\,}{\cdot} \ot \ip{\cdot\,}{\cdot}) \circ (\id \ot \Upsilon_{UH} \ot \id) \circ (\Delta\ot\id\ot \id ), \\ 
&\ip{1}{a}=\eps(a),\quad \ip{h}{1}=\eps(h), \quad a\in H, \ \ h\in U.    \label{1a}
\end{align} 
For a dual pairing between braided Hopf algebras, 
it is additionally required that 
\[
\ip{\cdot\,}{\cdot} \circ (S \ot \id ) = \ip{\cdot\,}{\cdot} \circ (\id \ot S). 
\]
If the dual pairing is non-degenerate, 
$U$ is called a left dual of $H$, and $H$ is called a right dual of $U$. 
\end{definition}  

Given a braided coalgebra $H$, the convolution product \eqref{conv} turns $H'$ into an associative algebra. 
In general, this product will not be compatible with the dual pairing of Definition~\ref{dualp}. 
On the other hand, under the assumptions of Definition \ref{dualp}, the product on $U\subset H'$ is 
uniquely determined by \eqref{mD} since $H$ separates the points of $H'$. 
For this reason, we will consider an alternative 
convolution product on $H'$ such that the equality in \eqref{mD} 
is automatically met for subalgebras $U\subset H'$ 
satisfying the assumptions of Lemma \ref{lembraided}.  

\begin{proposition}  \label{lust}
Let $H$ be a braided coalgebra and set 
\begin{align}   \label{ust}
&\ust : H' \ot H' \lra H', \quad f\ust g(a) := \ipp{f\ot g\,}{\,\Psi_{HH}\!\circ\!\Delta(a)}
=f({a_{(1)}}^{\!\br{2}})\hs g({a_{(2)}}^{\!\br{1}}),& 
\end{align}
where $f,g\in H'$ and $a\in H$. 
Then  \eqref{ust} turns $H'$ into an associative unital algebra with 
the unit element given by the counit of~$H$. 

Suppose that $U\subset H'$ is a (unital) subalgebra separating the points of $H$.  
If $\Psi_{UU}$ defined in \eqref{pbbraid} is bijective, 
then it turns $U$ into a braided algebra. 
In case $\Psi_{UH}$ or $\Psi_{UH}^\circ$ satisfies the assumptions of Lemma \ref{lembraided},  
then it defines a braiding that is compatible with the multiplication on $U$.  
The same remains true for the opposite version with $\Psi_{UH}$ and $\Psi_{UH}^\circ$ 
replaced by $\Psi_{HU}$ and $\Psi_{HU}^\circ$, respectively. 
\end{proposition}

\begin{proof} 
The associativity of the product $\ust$ 
is equivalent to coassociativity of $\Delta_1:= \Psi_{HH}\circ\Delta$ which was proven in Proposition \ref{lemmD}. 
Moreover, by \eqref{ust} and 
the second identity in \eqref{PVH}, 
\begin{align*}
f\ust \varepsilon(a)= f\circ (\varepsilon \ot \id) \circ \Psi_{HH}\circ \Delta(a) 
=  f\circ (\id \ot \varepsilon) \circ \Delta(a) =f(a) 
\end{align*} 
and similarly $\varepsilon\ust f(a)=f(a) $ for all $f\in H'$ and $a\in H$. 
Therefore $\varepsilon$ yields the unit element in $H'$ with respect to the product $\ust$.

To show the compatibilty of the multiplication with the braiding, 
we have to verify Equations \eqref{AWm}--\eqref{V1} for $\Psi_{UU}$ 
but only  \eqref{AWm} and \eqref{1W} for $\Psi_{UH}$. 
Since $U$  separates the points of $H$ and vice versa, 
we may again prove the required relations 
by evaluating both sides on elements from $U$ and $H$. 
Let $a, b, c\in H$ and $f,g,h\in U$. 
In Sweedler-type notation, the proof of \eqref{AWm} reads as follows: 
\begin{align}\begin{split} \label{SwAWm}
&h^{\br{1}}(b)\hs (f\ust g)^{\br{2}}(a) \oeq{\eqref{fagb},\eqref{ust}} 
h(b^{\br{1}_1})\hs g({a^{\br{2}_1}}_{(2)}{}^{\!\br{1}_2})\hs f({a^{\br{2}_1}}_{(1)}{}^{\!\br{2}_2}) \\
&\oeq{\eqref{bcop}} h(b^{\br{1}_2\br{1}_3})\hs g(a_{(2)}{\!}^{\br{1}_1\br{2}_3}) \hs f(a_{(1)}{\!}^{\br{2}_1\br{2}_2}) 
\oeq{\eqref{fagb},\eqref{ust}} h^{\br{1}_1\br{1}_2}(b) \hs (f^{\br{2}_2}\ust g^{\br{2}_1})(a). 
\end{split}
\end{align} 
Similarly, 
\begin{align}\begin{split} \label{SwVAm}
&(f\ust g)^{\br{1}}(b)\hs h^{\br{2}}(a)  \oeq{\eqref{fagb},\eqref{ust}} 
 g({b^{\br{1}_1}}_{(2)}{}^{\!\br{1}_2})\hs f({b^{\br{1}_1}}_{(1)}{}^{\!\br{2}_2})\hs h(a^{\br{2}_1}) \\
&\oeq{\eqref{YBE},\eqref{PVH} }  
g(b_{(2)}{\!}^{\br{1}_1\br{1}_2}) \hs f(b_{(1)}{\!}^{\br{2}_1\br{1}_3})\hs h(a^{\br{2}_2\br{2}_3})
\oeq{\eqref{fagb},\eqref{ust}}  (f^{\br{1}_1}\ust g^{\br{1}_2})(b)\hs h^{\br{2}_1\br{2}_2}(a),  
\end{split}
\end{align} 
which yields \eqref{VAm} for $\Psi_{UU}$. Moreover, 
\begin{align}\label{Sw1}
f^{\br{1}}(b) \hs \vare^{\br{2}}(a)\oeq{\eqref{fagb}} f(b^{\br{1}})\hs \vare(a^{\br{2}}) 
\oeq{\eqref{PHW}}\vare(a)\hs f(b)
\oeq{\eqref{PVH}} \vare(a^{\br{1}})\hs f(b^{\br{2}})
\oeq{\eqref{fagb}} \vare^{\br{1}}(a)\hs f^{\br{2}}(b)
\end{align} 
shows that $\Psi_{UU}$ fulfills \eqref{1W} and \eqref{V1}. 
This finishes the proof that $\Psi_{UU}$ is compatible with the multiplication of $U$. 

Applying \eqref{phh} to both sides of \eqref{SwVAm} and
the right side of \eqref{Sw1} gives 
\begin{align*}
h(a^{\Br{1}})\hs (f\ust g)^{\Br{2}}(b)
= h(a^{\Br{1}_1\Br{1}_2}) \hs (f^{\Br{2}_2}\ust g^{\Br{2}_1})(b), \quad  
f(b^{\Br{1}})\hs \vare^{\Br{2}}(a) =  f(b) \hs \vare(a), 
\end{align*} 
so that $\Psi_{UH}$ is also compatible with the multiplication on $U$. 

By Lemma \ref{Linv}, 
$\Psi_{HH}^{-1}$ is compatible with the comultiplication on $H$. 
Similar as above,    
\begin{align}
&h(a^{\Br{1}^\circ})\hs (f\ust g)^{\Br{2}^\circ}(b) \oeq{\eqref{cphh},\eqref{ust}} 
h(a^{\pbr{2}})\hs g({b^{\pbr{1}}}_{\!(2)}{}^{\!\br{1}})\hs f({b^{\pbr{1}}}_{\!(1)}{}^{\!\br{2}}) \nonumber \\ 
\begin{split} \nonumber 
&\oeq{\eqref{PVH}} h(a^{\pbr{2}_1\pbr{2}_2})\hs g(b_{(2)}{}^{\!\pbr{1}_2\br{1}})\hs f(b_{(1)}{}^{\!\pbr{1}_1\br{2}})
\oeq{\eqref{YBe}} h(a^{\pbr{2}_1\pbr{2}_2})\hs g(b_{(2)}{}^{\!\br{1}\pbr{1}_1})\hs f(b_{(1)}{}^{\!\br{2}\pbr{1}_2})\\
&\oeq{\eqref{cphh}} h^{\pbr{2}}(a^{\Br{1}^\circ})\hs g^{\Br{2}^\circ}(b_{(2)}{}^{\!\br{1}})\hs f^{\pbr{1}}(b_{(1)}{}^{\!\br{2}})
\oeq{\eqref{cphh}} h(a^{\Br{1}_1^\circ\Br{1}_2^\circ})\hs 
g^{\Br{2}_1^\circ}(b_{(2)}{}^{\!\br{1}})\hs f^{\Br{2}_2^\circ}(b_{(1)}{}^{\!\br{2}})
\end{split}\\ \nonumber 
&\oeq{\eqref{ust}} h(a^{\Br{1}_1^\circ\Br{1}_2^\circ}) \hs (f^{\Br{2}_2^\circ}\ust g^{\Br{2}_1^\circ})(b), 
\end{align} 
hence $\Psi_{UH}^\circ$ satisfies \eqref{AWm}. 
By Lemma \ref{Linv} and the second relation of \eqref{PVH}, 
$$
f(a^{\Br{1}^\circ})\hs \vare^{\Br{2}^\circ}\!(b)= f(a^{\pbr{2}})\hs \vare(b^{\pbr{1}})
=f(a)\hs \vare(b), 
$$
which shows that $\Psi_{UH}^\circ$ satisfies also \eqref{1W}. 
This proves the compatibility of $\Psi_{UH}^\circ$ 
with the multiplication on $U$.  

Is opposite version is shown in the same way. 
\end{proof} 

Recall that, for any unital algebra $H$, the dual space $H'$ 
contains a largest coalgebra $H^\circ$ such that  
$\Delta : H^\circ \to H^\circ \ot H^\circ$, $\Delta(f)(a\ot b):=f(ab)$ for $a,b\in H$, 
and $\varepsilon(f):=f(1)$ (see e.g.\ \cite{KS}). However, the compatibility condition \eqref{Dm} of the dual 
pairing requires to consider a modified coproduct, say $\Delta^{\!\backprime}$, on suitable subspaces of $H'$. 
To state an explicit formula, assume that $H$ is a braided algebra  
and suppose that $U$ is a linear subspace of $H'$ satisfying the assumptions of 
Lemma \ref{lembraided}. Then, 
for $\Delta^{\!\backprime}$ on $U$ and $\Upsilon_{UH}=\Psi_{UH}$, 
\eqref{Dm} is equivalent to 
$$
\ipp{\Delta^{\!\backprime}(f)}{\Psi_{HH}(a \ot b)} = \ip{f}{m(a\ot b)},\quad f\in U, \ \ a,b\in H, 
$$
which leads to 
\[ \label{Dbraid}
\ipp{\Delta^{\!\backprime}(f)}{a \ot b} = \ip{f}{ m\circ \Psi_{HH}^{-1}(a \ot b)}\quad f\in U, \ \ a,b\in H. 
\]
If we replace $\Psi_{UH}$ by $\Psi_{UH}^\circ$ in \eqref{Dm}, 
then $\Psi_{HH}^{-1}$ needs to be replaced by $\Psi_{HH}$ in \eqref{Dbraid}. 

Note that \eqref{Dbraid} determines uniquely $\Delta^{\!\backprime}(f)\in (H\ot H)'$ 
but $\Delta^{\!\backprime}(f)$ defined by the right hand side of \eqref{Dbraid} 
may not belong to $H'\ot H'$. 
The next proposition shows that, similar to the unbraided case, $H'$ contains 
a largest coalgebra such that the coproduct is given as in \eqref{Dbraid} 
and any subcoalgebra $U$ satisfying the assumptions of Lemma \ref{lembraided} 
becomes a braided coalgebra such that the comultiplication is compatible with the braidings 
$\Psi_{UH}$ and $\Psi_{UH}^\circ$. 

\begin{proposition}   \label{lcop}
Let $H$ be a braided unital algebra and consider 
\[   \label{brcop}
\uD : H' \lra (H\ot H)' , \quad \ipp{\uD(f)}{a\ot b}:= \ip{f}{\,m\!\circ\! \Psi_{HH}(a\ot b)}
=f(b^{\br{1}}  a^{\br{2}}), 
\] 
where $f\in H'$, $a, b\in H$. Then  there exists a largest 
coalgebra $H^{\uc}$ in $H'$ such that the coproduct is given by \eqref{brcop} 
and $\vare(f):=f(1)$. 

Suppose that $U\subset H^{\uc}$ is a subcoalgebra separating the points of $H$.  
If $\Psi_{UU}$ defined in \eqref{pbbraid} is bijective, 
then it turns $U$ into a braided coalgebra. 
In case $\Psi_{UH}$ or $\Psi_{UH}^\circ$ satisfies the assumptions of Lemma \ref{lembraided},  
then it defines a braiding that is compatible with the comultiplication on $U$.  
The same remains true for the opposite version with $\Psi_{UH}$ and $\Psi_{UH}^\circ$ 
replaced by $\Psi_{HU}$ and $\Psi_{HU}^\circ$, respectively. 
\end{proposition}
\begin{proof} 
By Proposition \ref{lemmD}, $m_1:=m\!\circ\! \Psi_{HH}$ defines an associative multiplication on $H$ with 
unit element $1\in H$.  Therefore the existence of a largest coalgebra follows from the known result 
of the unbraided case, i.e., there exists a largest coalgebra, say $H^{\uc}$, 
in $H'$ with the coproduct given by \eqref{brcop} 
and $\vare(f):=f(1)$, 
see e.g.\ \cite[Section 1.2.8]{KS}. 

Assume that the subcoalgebra $U\subset H^{\uc}$ satisfies the assumptions of the proposition 
which guarantee that $\Psi_{UU}$ and $\Psi_{UH}$ are well-defined. 
Then, for all $f,g\in U$ and $a,b,c\in H$, we have 
\begin{align}  \nonumber 
& g^{\br{1}}{\!}_{(1)}(c) \hs g^{\br{1}}{\!}_{(2)}(b)\hs f^{\br{2}}\hsp (a)  
\oeq{\eqref{brcop}}  g^{\br{1}_2}(c^{\br{1}_1} b^{\hsp \br{2}_1}) \hs f^{\br{2}_2}\hsp (a)
\oeq{\eqref{phh}}    g((c^{\br{1}_1} b^{\br{2}_1})^{\br{1}_2})\hs f(a^{\br{2}_2}) \\ \label{prPVH} 
 &\oeq{\eqref{VAm}} g(c^{\br{1}_1\br{1}_2} b^{\br{2}_1\br{1}_3}) \hs  f(a^{\br{2}_2\br{2}_3}) 
 \oeq{\eqref{YBE}} g(c^{\br{1}_2\br{1}_3} b^{\br{1}_1\br{2}_3})\hs f(a^{\br{2}_1\br{2}_2})  \\ \nonumber 
 &\oeq{\eqref{brcop}} g_{(1)}(c^{\br{1}_2})\hs g_{(2)}(b^{\br{1}_1}) \hs  f(a^{\br{2}_1\br{2}_2}) 
 \oeq{\eqref{phh}}   g_{(1)}{\!}^{\br{1}_1} (c)\, g_{(2)}{\!}^{\br{1}_2}(b) \hs f^{\br{2}_1\br{2}_2}(a)  
\end{align} 
which proves that $\Psi_{UU}$ satisfies the first relation of \eqref{PVH}. 
The second relation of  \eqref{PVH} follows from
$$
\vare(g^{\br{1}})\hs  f^{\br{2}}(a) = g^{\br{1}}(1) \hs f^{\br{2}}(a) 
\oeq{\eqref{phh}}   g(1^{\br{1}})  \hs f(a^{\br{2}})
\oeq{\eqref{V1}}   g(1) \hs f(a)  = \vare(g) \hs f(a). 
$$
In exactly the same way, one shows that $\Psi_{UU}$ satisfies \eqref{PHW}, 
hence $U$ is a braided coalgebra with respect to braiding $\Psi_{UU}$.  

Applying $\uD$ to \eqref{phcph}, evaluating on $b\ot c \in H\ot H$ 
and using \eqref{prPVH} gives 
\begin{align*} 
f(a^{\Br{1}}) \hs  g^{\Br{2}}{\!}_{(1)}(c) \hs  g^{\Br{2}}{\!}_{(2)}(b)
&\oeq{\eqref{phcph}} f^{\br{2}}\hsp (a) \hs g^{\br{1}}{\!}_{(1)}(c) \hs g^{\br{1}}{\!}_{(2)}(b)
\oeq{\eqref{prPVH}} f^{\br{2}_1\br{2}_2}(a) \hs g_{(1)}{\!}^{\br{1}_1} (c)\, g_{(2)}{\!}^{\br{1}_2}(b) \\
&\oeq{\eqref{phh}} f(a^{\Br{1}_1\Br{1}_2}) \hs  g_{(1)}{\!}^{\Br{2}_2} (c)\hs g_{(2)}{\!}^{\Br{2}_1}(b). 
\end{align*} 
Moreover, 
$$
f\big(a^{\Br{1}}\hs\vare(g^{\Br{2}})\big) = g^{\Br{2}}\hsp(1) \hs  f(a^{\Br{1}})
\oeq{\eqref{phh}} g(1^{\br{1}}) \hs f(a^{\br{2}}) \oeq{\eqref{V1}}  g(1) \hs f(a) 
=  f\big(\vare(g)\hs a\big). 
$$
From the last two computations, we conclude that $\Psi_{UH}$ satisfies \eqref{PHW}, i.e., 
the braiding $\Psi_{UH}$ is compatible with the comultiplication on $U$. 

Much in the same way, by applying Lemma \ref{Linv} to $\Psi_{HH}$, 
\begin{align*} 
&f(a^{\cBr{1}}) \hs  g^{\cBr{2}}{\!\!}_{(1)}(c) \hs  g^{\cBr{2}}{\!\!}_{(2)}(b)
\oeq{\eqref{ophcph}} f^{\pbr{2}}\hsp (a) \hs g^{\pbr{1}}{\!}_{(1)}(c) \, g^{\pbr{1}}{\!}_{(2)}(b) 
\oeq{\eqref{brcop}}   f^{\pbr{2}}\hsp (a) \hs g^{\pbr{1}}(c^{\br{1}} b^{\br{2}}) \\
&\oeq{\eqref{cphh}}  f(a^{\pbr{2}_2}) \hs  g((c^{\br{1}_1} b^{\br{2}_1})^{\pbr{1}_2})
 \oeq{\eqref{VAm}}  f(a^{\pbr{2}_1\pbr{2}_2}) \hs g(c^{\br{1}\pbr{1}_1} b^{\br{2}\pbr{1}_2})  \\
 &\oeq{\eqref{YBe}} f(a^{\pbr{2}_1\pbr{2}_2})  \hs g(c^{\pbr{1}_2\br{1}} b^{\pbr{1}_1\br{2}}) 
 \oeq{\eqref{brcop}}  f(a^{\pbr{2}_1\pbr{2}_2})  \hs g_{(1)}(c^{\pbr{1}_2})\hs g_{(2)}(b^{\pbr{1}_1})\\
 & \oeq{\eqref{cphh}} f^{\pbr{2}_2}(a^{\pbr{2}_1})  \hs g_{(1)}{\!}^{\pbr{1}_2}(c)\hs g_{(2)}(b^{\pbr{1}_1}) 
 \oeq{\eqref{cphh}}  f(a^{\cBr{1}_1\cBr{1}_2}) \hs g_{(1)}{\!}^{\cBr{2}_2}(c)\hs g_{(2)}{\!}^{\cBr{2}_1}(b), 
\end{align*} 
and 
$f\big(a^{\Br{1}}\hs\vare(g^{\cBr{2}})\big) = g^{\cBr{2}}\hsp(1) \hs  f(a^{\cBr{1}})
\oeq{\eqref{cphh}} g(1^{\pbr{1}}) \hs f(a^{\pbr{2}}) \oeq{\eqref{1W}}  g(1) \hs f(a) =  f\big(\vare(g)\hs a\big)$. 
Therefore, if the braiding $\Psi_{UH}^\circ$ exists, it is compatible 
with the comultiplication on $U$.  

The opposite version is proven analogously. 
\end{proof} 

Given a  braided bialgebra $H$, Lemma \ref{lembraided}, Proposition \ref{lust} and Proposition \ref{lcop} 
tell us how to define
braidings, a product and a coproduct, respectively, on appropriate subspaces of $H'$.    
The next theorem shows that 
these structures fit well together, i.e., they can be used to 
obtain a braided bialgebra  $U\subset H'$ such that the canonical pairing yields 
a dual pairing between braided bialgebras. 
It is noteworthy to mention that the braiding $\Upsilon_{UH}$ 
in Definition \ref{dualp} will not be implemented by $\Psi_{UH}$, but by $\Psi_{UH}^\circ$. 

\begin{theorem}  \label{dH} 
Let $H$ be a braided bialgebra and 
consider the product $\um :  H' \ot H' \to H'$ defined by 
\[   \label{us}
\ip{\um(f\ot g)}{a} := \ipp{f\ot g}{\Psi_{HH}^{-1}\! \circ\! \Delta(a)}
=f({a_{(1)}}^{\!\pbr{2}})\hs g({a_{(2)}}^{\!\pbr{1}}), \quad a\in H, \ f,g\in H'.  
\]
Assume that  $U \subset H'$ is a unital subalgebra which is also a 
subcoalgebra of  \hs$H^{\uc}$\, for the coproduct $\uD$ defined in Proposition \ref{lcop}. 

Suppose that $U$ satisfies the left-handed version of Lemma \ref{lembraided} 
such that $\Psi_{UU}$ and $\Psi_{UH}^\circ$ define braidings.  Then $U$ is a 
braided bialgebra with respect to the braiding $\Psi_{UU}$, and the 
canonical pairing $\ip{\cdot\,}{\cdot} : U\ot H \to \K$ \,defines 
a pairing between braided bialgebras 
with respect to the braiding $\Psi_{UH}^\circ$ such that $U$ becomes a 
left dual of $H$.  

In case $U$ satisfies the right-handed version of Lemma \ref{lembraided} 
such that $\Psi_{UU}$ and $\Psi_{HU}^\circ$ yield braidings, 
the canonical pairing $\ip{\cdot\,}{\cdot} : H\ot U \to \K$ \,defines 
a pairing between braided bialgebras 
with respect to the braiding $\Psi_{HU}^\circ$ and $U$ becomes a 
right dual of $H$.  

If $H$ is a braided Hopf algebra and $f\circ S\in U$ for all $f\in U$, then 
$U$ is a braided Hopf algebra with antipode $S(f):=f\circ S$.  
\end{theorem}  
\begin{proof}
By Lemma \ref{Linv}, $\Psi_{HH}^{-1}$ defines a braiding on $H$ which is compatible with the 
comultiplication on $H$. Therefore, by Proposition \ref{lust}, the product \eqref{us} is well-defined. 
Note that the braiding in \eqref{us} is the inverse of that in \eqref{ust}. 
Combining Proposition \ref{lust} with Lemma \ref{Linv} shows that the multiplication on $U$ 
is compatible with $\Psi_{UU}$ so that $U$ becomes a braided algebra. 
From Proposition \ref{lcop}, it follows directly that $U$ is a braided coalgebra 
with respect to $\Psi_{UU}$. Therefore, to complete the proof that $U$ yields a braided bialgebra, 
it suffices to prove~\eqref{Dcm}.  

As in the previous proofs, we will verify \eqref{Dcm} 
by evaluating both sides on $a\ot b\in H\ot H$.  
It was shown in Proposition \ref{lemmD} that $H^{(1,-1)}$ is a braided bialgebra 
with respect to $\Psi_{HH}$. 
Therefore, for all  $f\ot g\in U\ot U$, 
\begin{align*} 
&\ipp{\uD \circ \um(f\ot g) }{a\ot b}
\oeq{\eqref{brcop},\eqref{us}} \ipp{f\ot g }{\Delta_{\hs -1} \circ m_1(a\ot b)} \\
&\quad\oeq{\eqref{Dcm}} 
\ipp{f\ot g }{(m_1\ot m_1) \circ (\id \ot \Psi_{HH} \ot \id) \circ (\Delta_{\hs -1}\ot \Delta_{\hs -1})(a\ot b)} \\
&\quad\oeq{\eqref{brcop}} 
\ipp{ (\uD \ot \uD)(f\ot g)\,}{(\id \ot \Psi_{HH} \ot \id) \circ (\Delta_{\hs -1}\ot \Delta_{\hs -1})(a\ot b)}\\
&\hspace{4.5pt}\oeq{\eqref{UU},\eqref{pbbraid}} 
\ipp{ (\id \ot \Psi_{UU} \ot \id) \circ(\uD \ot \uD)(f\ot g)\,}{\Delta_{\hs -1}\ot \Delta_{\hs -1}(a\ot b)}  \\ 
&\quad\oeq{\eqref{us}} 
\ipp{ (\um \ot\um)\circ (\id \ot \Psi_{UU} \ot \id) \circ(\uD \ot \uD)(f\ot g)} {a\ot b}, 
\end{align*}  
which implies that the product and the coproduct on $U$ satisfy the compatibility 
condition \eqref{Dcm} with respect to the 
braiding $\Psi_{UU}$. Thus we have proven that $U$ is a braided bialgebra. 

Our next aim is to show that the cannonical pairing defines 
a pairing between braided bialgebras with respect to the braiding $\Psi_{UH}^\circ$. 
From Lemma \ref{lembraided}, Proposition  \ref{lust} and 
Proposition~\ref{lcop}, we conclude that $\Psi_{UH}^\circ$ is compatible 
with the multilplications and the comultiplications on $U$ and $H$. 
By Definitionen \ref{dualp}, it remains to prove 
Equations \eqref{mD}--\eqref{1a}. 

Equation \eqref{1a} is trivially satisfied by the stated unit element in Proposition \ref{lust}
and the definition of the counit in Proposition \ref{lcop}. 
Let $f,g\in U$ and $a,b\in H$. Equation \eqref{mD} follows from 
$$
\um(f\ot g)(a) \oeq{\eqref{us}} g({a_{(2)}}^{\pbr{1}})\hs f({a_{(1)}}^{\pbr{2}}) 
  \oeq{\eqref{cphh}}  f({a_{(1)}}^{\cBr{1}}) \hs g^{\cBr{2}\!}(a_{(2)}), 
$$
and  
$$
f(ab) \oeq{\eqref{brpbr}} f(a^{\pbr{2}\br{1}}\hs b^{\pbr{1}\br{2}}) 
\oeq{\eqref{brcop}}  f_{(1)}( a^{\pbr{2}}) \hs f_{(2)}(b^{\pbr{1}})
 \oeq{\eqref{cphh}}  f_{(1)}( a^{\cBr{1}}) \hs f_{(2)}{\!}^{\cBr{2}\!\hsp}(b)
$$
implies \eqref{Dm}. 

Now assume that $H$ is a braided Hopf algebra with antipode $S$ and that $f\circ S\in U$ for all 
$f\in U$. With the definition $S(f):=f\circ S$, we compute for all $a\in H$  
\begin{align*} 
\ip{\um\circ (\id\hsp\ot\hsp S)\circ \uD (f)\hs}{a}  
&\oeq{\eqref{us}} 
\ipp{(\id\hsp\ot\hsp S)\circ \uD (f)\hs}{\Delta_{-1}(a)} \oeq{\eqref{braidip}} 
\ipp{ \uD (f)}{(S\hsp\ot\hsp \id)\circ\Delta_{-1}(a)} \\
& \oeq{\eqref{brcop}}  \ip{f\hs}{m_1\circ (S\hsp\ot\hsp\id)\circ\Delta_{-1}(a)}
= \vare(a) \hs \ip{f}{1} = \ip{ \vare(f)\hs 1}{a}, 
\end{align*} 
where we used the fact from Proposition \ref{lemmD} that 
$H^{(1,-1)}$ is a braided Hopf algebra 
with antipode $S$. 
This yields $\um\circ (\id\hsp\ot\hsp S)\circ \uD (f)= \vare(f) 1$. 
Interchanging the positions of $S$ and $\id$ in above calculations 
shows that also $\um\circ (S\hsp\ot\hsp \id)\circ \uD (f)= \vare(f) 1$, 
therefore the linear mapping $S: U\rightarrow U$ defined above turns $U$ into a braided Hopf algebra.  

Since $f^{\cBr{1}}\!(a)\,g(b^{\cBr{2}}) = f(a^{\pbr{2}})\,g(b^{\pbr{1}}) 
= f(a^{\cBr{1}})\,g^{\cBr{2}}\!(b)$ by \eqref{cphh}, and 
$\ip{\cdot}{\cdot}$ is symmetric in the sense that 
$\ip{a}{f} = f(a) = \ip{f}{a}$, the proof of the opposite version is essentially the same. 
\end{proof} 

Note that the same bialgebra $U$ serves as a left and as a right dual of $H$ by considering 
either the braiding  $\Psi_{UH}^\circ$ or the braiding $\Psi_{HU}^\circ$.   
From now on, we restrict ourselves mainly to the left version, 
the right version differs essentially only in notation. 

Our definitions of the product and the coproduct on $U\subset H'$ 
deviate from those given in \cite{Maj94a} for rigid (="finite") braided monoidal categories. 
By evaluating elements of $B^*\ (\sim U)$ on elements of $B\ (\sim H)$ in \cite{Maj94a}, 
it can be seen that the main difference boils down to omitting 
the braiding $\Upsilon_{UH}$ in Definition \ref{dualp} and using the $\ipp{\cdot}{\cdot}$ instead. 
For a finite dimensional braided bialgebra, 
the dual braided bialgebra in the sense of \cite[Proposition~2.4]{Maj94a} 
is a special case of our construction and corresponds to $U^{(1,-1)}$. 
We present this result in the following corollary, 
its proof is straightforward by applying the corresponding definitions. 

\begin{corollary}
Let $H$ be a braided bialgebra and assume that  $U\subset H'$ fulfills the assumptions 
of Theorem \ref{dH}. 
Then the product $m_{1}$ and the coproduct $\Delta_{-1}$ of \,$U^{(1,-1)}$ satisfy 
$$
\ip{m_1(f\ot g)}{a} = \ipp{f\ot g}{\Delta(a)} , \qquad \ipp{\Delta_{-1}(f)}{a\ot b} = \ip{f}{m(a\ot b)}
$$
for all $f,g\in U$ and $a,b\in H$.
\end{corollary} 

Note that, for a finite-dimensional braided bialgebra $H$, 
the non-degeneracy condition implies $U=H'$, and $H'$ satisfies 
automatically the assumptions on $U$ in above theorem. 
In this case, the left dual $U=H'$ is unique. For an infinite-dimensional braided bialgebra, 
this need not to be the case, see e.g. \cite[Section 11.2.3]{KS} 
with all braidings given by the usual flip. 

Now let  $H$ be a possibly infinite-dimensional braided bialgebra 
and $U\subset H'$ a left dual of $H$. 
If $H\subset U'$ satisfies the conditions in Lemma \ref{lembraided},  
then we can apply Theorem \ref{dH} 
to construct a multiplication $\um$ and a coproduct $\uD$ on $H$ 
such that $H$ becomes a left dual of $U$. 
The next proposition shows that this left dual is isomorphic to the 
braided bialgebra $H$. In this sense the construction is reflexive, i.e., 
taking twice the left dual yields 
the same braided bialgebra. 

\begin{proposition} \label{propdual}
Let $H$ be a braided bialgebra and let $U\subset H'$ be a left dual of $H$ 
with braiding $\Psi_{UU}$ satisying \eqref{pbbraid}. 
Assume that the map $\Psi_{HU}^\circ$ defined in Lemma \ref{lembraided} 
is a bijection.  
Consider the canonical embedding $\iota: H \rightarrow \iota(H)\subset U'$ given by 
$\iota(a)(f):= f(a)$.  
Then $\iota(H)$ with the multiplication $\um$ 
and the comultiplication $\uD$ \,from Theorem \ref{dH} 
becomes a left dual of \,$U$, and $\iota : H \rightarrow \iota(H)$ yields an isomorphism 
of braided bialgebras. 

Analogously, if \,$U\subset H'$ is a right dual of $H$ and $\Psi_{UH}^\circ$ 
in \eqref{ipbb} is a bijection, then $\iota(H)\cong H$ becomes a right dual of \,$U$. 

If $H$ is a braided Hopf algebra, then $\iota$ yields an isomorphism of braided 
Hopf algebras. 
\end{proposition} 
\begin{proof} 
Let $a,b\in H$ and $f,g\in U$. 
Since
$$
\Psi_{\iota(H)\iota(H)}(\iota(a) \hsp\ot\hsp \iota(b))(g\hsp\ot\hsp f)
\oeq{\eqref{UU}} \ipp{a\hsp\ot\hsp b}{\Psi_{UU}(f\hsp\ot\hsp g)} 
\oeq{\eqref{pbbraid}} \ipp{(\iota\hsp\ot\hsp \iota)\circ\Psi_{HH}(a\hsp\ot\hsp b)}{f\hsp\ot\hsp g} ,  
$$
we conclude that $\Psi_{\iota(H)\iota(H)} : \iota(H)\ot \iota(H)\to  \iota(H)\ot \iota(H)$ 
yields a bijection and that $\iota$ intertwines the braidings on $H$ and $\iota(H)$. 
In the same manner, it follows from \eqref{UU} and \eqref{PHU} that 
$\Psi_{\iota(H)U}^\circ\circ(\iota\ot \id)=(\id\ot\iota)\circ\Psi_{HU}^\circ$ 
which shows that $\Psi_{\iota(H)U}^\circ$ is bijective.  Thus, by Lemma \ref{lembraided}, 
it defines a braiding. 
 
By \eqref{UU}, \eqref{brcop} and \eqref{us}, we have
$$
\ip{\um(\iota(a)\hsp\ot\hsp \iota(b)) }{f} = \ipp{a\hsp\ot\hsp b}{\Psi_{UU}^{-1}\circ\uD(f)} 
= \ip{m(a\hsp\ot\hsp b)}{f}= \ip{\iota\circ m(a\hsp\ot\hsp b)}{f}, 
$$
hence $\iota\circ m  = \um\circ (\iota \ot \iota)$.  This implies that 
$\iota(H) \subset U'$ is a unital subalgebra 
and that $\iota$ yields an algebra isomorphism.

Next we prove that $\iota$ determines an coalgebra isomorphism. From  
\[ \label{dud}
\ipp{\uD(\iota(a))}{f\hsp\ot\hsp g} \oeq{\eqref{brcop}} \ip{a\,}{\um \circ \Psi_{UU}(f\hsp\ot\hsp g)} 
\!\oeq{\eqref{UU},\eqref{us}}\! \ipp{ \Delta(a)}{f\hsp\ot\hsp g} 
=  \ipp{(\iota\ot\iota)\circ \Delta(a)}{f\hsp\ot\hsp g} , 
\] 
we conclude that $(\iota\ot\iota)\circ \Delta = \uD\circ\iota$. 
As a consequence, $ \uD(\iota(H)) \subset \iota(H)\ot \iota(H)$. 
Moreover, $\vare(\iota(a))= \ip{a}{1}=\vare(a)$.   
Therefore $\iota(H) \subset U^{\uc}$ is a subcoalgebra isomorphic to $U$.  

Summarizing, we have shown that $\iota(H)$ satisfies the assumption of Theorem \ref{dH} 
so that it becomes a left dual of $U$ with multiplication and comultiplication 
given in \eqref{us} and \eqref{brcop}, 
and that $\iota: H \lra \iota(H)$ yields an isomorphism of braided bialgebras. 
If $H$ is a braided Hopf algebra, then $\iota$ lifts to a braided Hopf algebra isomorphism with 
$S\circ\iota= \iota\circ S$. The opposite version is proven analogously. 
\end{proof}

Given a braided bialgebra $H$, Proposition \ref{lemmD} 
allows us to construct a countable family 
of braided bialgebras. The next proposition gives a description 
of the corresponding dual bialgebras obtained from a left (or right) dual of $H$. 
\begin{proposition}  \label{Pdual}
Let $H$ be a braided bialgebra and let $U$ be a left dual of $H$ 
with respect to the braidings $\Psi_{UU}$ and $\Psi_{UH}^\circ$ 
given in \eqref{pbbraid} and \eqref{ipbb}, respectively.  
For $n\in \Z$\hs, \,$U^{(-n,n)}$ is a left dual of $H^{(n,-n)}$ 
with respect to the braiding $\Psi_{UH}^\circ$, and 
if $\Psi_{UH}$ is bijective, then 
$U^{(-n,n-1)}$ is a left dual of $H^{(n-1,n)}$ 
with respect to the braiding $\Psi_{UH}$.  
The analogous statements are true for  Hopf algebras, 
and  for the opposite versions 
with $\Psi_{UH}^\circ$ and $\Psi_{UH}$ replaced by 
$\Psi_{HU}^\circ$ and $\Psi_{HU}$, respectively. 
\end{proposition} 
\begin{proof} 
Let $n\in \Z$. As 
\begin{align*} 
\ip{\um_{\hs -n}(f\ot g)}{a} &\oeq{\eqref{mkDn1}}  \ip{\um\circ \Psi_{UU}^{-n}(f\ot g)}{a} 
\oeq{\eqref{UU},\eqref{us}}  \ipp{f\ot g}{\Psi_{HH}^{-n}\circ \Psi_{HH}^{-1}\circ  \Delta(a)} \\
&\oeq{\eqref{mkDn2}}  \ipp{f\ot g}{\Psi_{HH}^{-1}\circ \Delta_{-n}(a)} 
\end{align*} 
and 
\begin{align*}   \label{brcop}
\ipp{\uD_{n}(f)}{a\ot b} &\oeq{\eqref{mkDn2}} \ipp{\Psi_{UU}^{n} \circ \uD(f)}{a\ot b}
\oeq{\eqref{UU},\eqref{brcop}} \ip{f}{\,m\!\circ\! \Psi_{HH}\!\circ\! \Psi_{HH}^{n}(a\ot b)} \\
&\oeq{\eqref{mkDn1}}  \ip{f}{\,m_{n}\!\circ\! \Psi_{HH}(a\ot b)}
\end{align*}
for $f,g\in U$ and $a,b\in H$, 
it follows from Proposition \ref{lemmD} and Theorem \ref{dH} 
that $U^{(-n,n)}$ is a left dual of $H^{(n,-n)}$. 

In the case of $H^{(n-1,n)}$, we have to replace the braiding $\Psi_{HH}$ on $H$ by $\Psi_{HH}^{-1}$ 
and therefore $\Psi_{UH}^\circ$ by $\Psi_{UH}$. 
A careful look at Proposition \ref{Pdual}, Lemma \ref{lembraided} 
and Theorem \ref{dH} reveals that the assumption of Theorem~\ref{dH} are still satisfied  
after these substitutions. 
Now the same calculations, but with
with $\uD_{n}$ replaced by $\uD_{n-1}$, show the result.  

The opposite versions are proven similarly. The statement about Hopf algebras follows 
from Proposition \ref{lemmD} and the definition of the antipode on the dual Hopf algebras. 
\end{proof}

\section{Modules and comodules  in the braided setting}   \label{sec5} 
 
A typical application of duality is to turn a comodule of a coalgebra into a module of a dual algebra. 
Adding more structure, a braided comodule algebra of a braided bialgebra should become 
a braided module algebra of a dual braided bialgebra. This will be discussed in Theorem \ref{mcom}.  
The dual version, i.e., turning a braided module into a braided comodule of a dual coalgebra, 
will be presented in Theorem \ref{mmm}. Propositions \ref{mco} and \ref{com} elaborate the same idea 
on duals of (co)modules. Similarly to Lemma  \ref{lembraided}, 
we start by lifting braidings on vector spaces to 
dual spaces. 

As in the previous section, we will make frequent use of the Sweedler-type notation:  
\begin{align*}
&\Psi_{VW}(v\ot w) := w^{\Br{1}} \ot v^{\Br{2}}, & &\Psi_{VW}^{-1}(w\ot v) := v^{\pBr{1}} \ot w^{\pBr{2}}, \\ 
&\Psi_{VW}^{\circ}(v\ot w) := w^{\cBr{1}} \ot v^{\cBr{2}}, &
&\Psi_{VW}^{\circ\hs -1}(w\ot v) := v^{{\cBr{1}}^\backprime} \ot w^{{\cBr{2}}^\backprime}, 
\end{align*}
where $v\in V$, $w\in W$, and $\Psi_{VW}^{\circ}$ denotes a braiding that is constructed 
from the inverse of a given one.  In this notation, we have 
\[ \label{ipBr}
 w^{\pBr{2}\,\hs \Br{1}} \ot v^{\pBr{1}\,\hs\Br{2}} =  w\ot v , 
 \qquad v^{\Br{2}\pBr{1}\,}\ot w^{\Br{1}\pBr{2}\,} =  v \ot w, 
\] 
and the same holds for $\Psi_{VW}^{\circ}$. 
It turns out that the braidings of the type $\Psi_{VW}^{\circ}$ 
will be the correct ones for obtaining our desired result as it happened in Theorem \ref{dH}. 
For this reason, we will give four versions of the auxiliary results on braidings, 
a left version, a right version and the corresponding versions arising from the inverse braidings.

\begin{lemma}  \label{dbvs}
Let $V$ be a right $H$-braided vector space. 
Consider the linear map 
\[ \label{bHV}
\Psi_{HV'} : H \ot V' \lra  (V \ot H')' , \quad  
\Psi_{HV'}(a\ot e)(v\ot f):= \ipp{e\ot f}{\Psi_{VH}(v\ot a)}, 
\] 
for $a\in H$, $v\in V$, $f\in H'$ and $e\in V'$. 
Let \,$W\subset V'$ be a non-degenerate subspace such that 
\[
\Psi_{HW} := \Psi_{HV'}\!\!\upharpoonright_{H\ot W}  \; : H \ot W \lra W\ot H\subset  (V \ot H')'  
\]
is bijective.  Then $\Psi_{HW}$ turns $W$ into a left $H$-braided vector space. 

If $H$ is a braided coalgebra and $\Psi_{VH}$ is compatible with the comultiplication, 
then $\Psi_{HW}$ is also compatible with the comultiplication of $H$. 
If $H$ is a braided (unital) algebra and $\Psi_{VH}$ is compatible with the multiplication, 
then $\Psi_{HW}$ is compatible with the multiplication of $H$. 

For a  left $H$-braided vector space $V$, it is required that 
\[
\Psi_{WH} := \Psi_{V'H}\!\!\upharpoonright_{W\ot H}  \; : W \ot H \lra H\ot W\subset  (H' \ot V)', 
\]
is bijective, where 
\[ \label{bHpV}
\Psi_{V'H} : V' \ot H \lra (H' \ot V)', \quad 
\Psi_{V'H}(e\ot a)(f\ot v):= \ipp{f\ot e}{\Psi_{HV}(a\ot v)}.
\]
In this case,  $\Psi_{WH}$ turns $W$ into a right $H$-braided vector space 
and the other implications remain the same under identical assumptions. 

The analogous statements hold for $\Psi_{WH}^{\circ}$ and  $\Psi_{HW}^{\circ}$ if 
\[ \label{cbWH}
\Psi_{V'H}^{\circ} : V' \ot H \lra  (H' \ot V)' , \quad  
\Psi_{V'H}^{\circ}(e\ot a)(f\ot v):= \ipp{f\ot e}{\Psi_{VH}^{-1}(a\ot v)}, 
\] 
yields a bijective map 
\[
\Psi_{WH}^{\circ} := \Psi_{V'H}^{\circ}\!\!\upharpoonright_{W\ot H}  \; : W \ot H \lra H\ot W\subset  (H' \ot V)',   
\]
and if 
\[ \label{cHV}
\Psi_{HV'}^{\circ}  : H \ot V' \lra  (V \ot H')' , \quad
\Psi_{HV'}^{\circ} (a\ot e)(v\ot f):= \ipp{e\ot f}{\Psi_{HV}^{-1}(v\ot a)}
\] 
yields a bijective map 
\[
\Psi_{HW}^{\circ} := \Psi_{HV'}^{\circ} \!\!\upharpoonright_{H\ot W}  \; : H \ot W \lra W\ot H\subset  (V \ot H')'. 
\]
\end{lemma}

\begin{proof} 
Let $a,b\in H$, $f\in H'$, $v\in V$ and $e\in W$. 
In Sweedler-type notation, we can write \eqref{bHV} in the form 
$e^{\Br{1}}\hsp (v)\, f(a^{\Br{2}}) = e(v^{\Br{2}})\, f(a^{\Br{1}})$, 
which is equivalent to 
\[ \label{eva} 
e^{\Br{1}}\hsp (v)\, a^{\Br{2}} = e(v^{\Br{2}})\, a^{\Br{1}} . 
\]
The proof is now straightforward,   
nevertheless we give parts of the proof in order to show how \eqref{eva} 
enables us to move the action of the new braiding to the spaces where the 
given braiding is defined.  
For instance, 
\begin{align*}
& e^{\Br{1}_1\Br{1}_2}(v) \,  b^{\br{1}\Br{2}_2}  \ot   a^{\br{2}\Br{2}_1}
\oeq{\eqref{eva}} 
e(v^{\Br{2}_1\Br{2}_2}) \,  b^{\br{1}\Br{1}_1} \ot  a^{\br{2}\Br{1}_2}\\
&\oeq{\eqref{WV}}  
e(v^{\Br{2}_1\Br{2}_2}) \,   b^{\Br{1}_2\br{1}} \ot a^{\Br{1}_1\br{2}}
\oeq{\eqref{eva}} 
e^{\Br{1}_1\Br{1}_2}(v) \, b^{\Br{2}_1\br{1}} \ot a^{\Br{2}_2\br{2}}  
\end{align*} 
proves \eqref{VW}. Furthermore, 
\begin{align*}
e^{\Br{1}_1\Br{1}_2}(v)\,  a_{(1)}{\!}^{\Br{2}_2} \ot a_{(2)}{\!}^{\Br{2}_1}
&\oeq{\eqref{eva}} 
e(v^{\Br{2}_1\Br{2}_2}) \, a_{(1)}{\!}^{\Br{1}_1} \ot a_{(2)}{\!}^{\Br{1}_2} 
\oeq{\eqref{PVH}} 
e(v^{\Br{2}}) \, a^{\Br{1}}{\!}_{(1)} \ot a^{\Br{1}}{\!}_{(2)} \\
&\oeq{\eqref{eva}} 
e^{\Br{1}}(v) \, a^{\Br{2}}{\!}_{(1)} \ot a^{\Br{2}}{\!}_{(2)} 
\end{align*} 
implies the first relation of \eqref{PHW}. The second relation follows from 
$$
\vare(a^{\Br{2}}) \hs e^{\Br{1}}\!(v) \oeq{\eqref{eva}} \vare(a^{\Br{1}}) \hs e(v^{\Br{2}}) 
\oeq{\eqref{PVH}} \vare(a) \hs e(v). 
$$
Likewise, 
\begin{align*}
 &  e^{\Br{1}}\hsp(v)  \, (ab)^{\Br{2}} 
 \oeq{\eqref{eva}}  e(v^{\Br{2}})\, (ab)^{\Br{1}} 
 \oeq{\eqref{VAm}} e(v^{\Br{2}_1\Br{2}_2})\, a^{\Br{1}_1}b^{\Br{1}_2}  
  \oeq{\eqref{eva}}  e^{\Br{1}_1\Br{1}_2}(v)\, a^{\Br{2}_2}\hs b^{\Br{2}_1}, 
\end{align*} 
which yields \eqref{AWm}. If $1\in H$, then \eqref{1W} follows from \eqref{V1}  and \eqref{eva} with $a=1$. 

The proof of the opposite version with the braiding $\Psi_{WH}$ uses similar arguments. 
Combining the obtained results with Lemmas \ref{Linv}  proves the statements 
for $\Psi_{WH}^{\circ}$ and $\Psi_{HW}^{\circ}$. 
\end{proof} 

Given an $H$-braided vector space $V$, the previous lemma showed how to define 
braidings between $H$ and appropriate subspaces of $V'$. 
The next lemma fixes $V$ and shows how to induce braidings between $V$ 
and appropriate subspaces $U\subset H'$. 
Moreover, if $U$ inherits a (co)multiplication from $H$, then the braidings 
between $U$ and $V$ inherit the compatibility properties 
from the corresponding braiding between $H$ and $V$. 

\begin{lemma}  \label{LUVH}
Let $V$ be a right $H$-braided vector space and consider the linear map   
\[ \label{bUVW}
\Psi_{H'V} : H' \ot V \lra  (V' \ot H)' , \quad  
\Psi_{H'V}(f\ot v)(e\ot a):=  \ipp{e\ot f}{\Psi_{VH}(v\ot a)}, 
\] 
for $a\in H$, $v\in V$, $f\in H'$ and $e\in V'$. 
Assume that $U\subset H'$ is a non-degenerate subspace such that 
$\Psi_{UU}$  given in \eqref{pbbraid} defines a braiding.  If 
\[     \label{PsiUV}
\Psi_{UV} := \Psi_{H'V}\!\!\upharpoonright_{U\ot V}  \; : U \ot V \lra V\ot U\subset  (V' \ot H)', 
\]
is bijective, then $\Psi_{UV}$ turns $V$ into a left $U$-braided vector space. 

If $H$ is a braided coalgebra,  $\Psi_{VH}$ is compatible with the comultiplication on $H$,    
and $U$ satisfies the assumptions of Proposition~\ref{lust}, 
then $\Psi_{UV}$ is compatible with the multiplication~$\ust$ from \eqref{ust} on \hs$U$. 
If $H$ is a braided unital algebra, $\Psi_{VH}$ is compatible with the multiplication on $H$,    
and $U$ satisfies the assumptions of Proposition~\ref{lcop}, 
then $\Psi_{UV}$ is compatible with the comultiplication~$\uD$ \,from  \eqref{brcop}  on~$\hs U$. 

Assume that $V$ is an algebra and $\Psi_{VH}$ is compatible with the multiplication on \hs$V$. 
Then $\Psi_{UV}$ is compatible with the multiplication on \hs$V$. 
If \hs$V$ is a coalgebra and $\Psi_{VH}$ is compatible with the comultiplication on \,$V$, then 
$\Psi_{UV}$ is also compatible with the comultiplication on \hs$V$. 

Given a  left $H$-braided vector space $V$  such that 
$$
\Psi_{VU} := \Psi_{VH'}\!\!\upharpoonright_{V\ot U}  \; : V \ot U \lra U\ot V\subset  (H \ot V')' 
$$ 
is bijective, where 
\[ \label{bVpH}
\Psi_{VH'} : V \ot H' \lra  (H \ot V')' , \quad  \Psi_{VH'}(v\ot f)(a\ot e):= \ipp{f\ot e}{\Psi_{HV}(a\ot v)}, 
\]
the map $\Psi_{VU}$ defines a braiding that turns $V$ into a right $U$-braided vector space    
and the opposite versions of above compatibility statements remain true.   

The analogous assertions hold if 
\[ \label{cbVU}
\Psi_{VH'}^{\circ} : V \ot H' \lra (H \ot V')' , \quad  
\Psi_{VH'}^{\circ}(v\ot f)(a\ot e):= \ipp{f\ot e}{\Psi_{VH}^{-1}(a\ot v)}, 
\] 
yields a bijective map 
\[
\Psi_{VU}^{\circ} := \Psi_{VH'}^{\circ}\!\!\upharpoonright_{V\ot U}  \; : V \ot U \lra U\ot V\subset  (H \ot V')',   
\]
and if 
\[   \label{cPHV} 
\Psi_{H'V}^{\circ}  : H' \ot V \lra  (V' \ot H)', \quad
\Psi_{H'V}^{\circ} (f\ot v)(e\ot a):=  \ipp{e\ot f}{\Psi_{HV}^{-1}(v\ot a)}, 
\] 
yields a bijective map 
\[  \label{cPUV} 
\Psi_{UV}^{\circ} := \Psi_{H'V}^{\circ} \!\!\upharpoonright_{U\ot V}  \; : U \ot V \lra V\ot U\subset (V' \ot H)'. 
\]
\end{lemma} 
\begin{proof} 
Although the lemma is proven along the lines of the previous ones, 
we will state the proof in order to demonstrate 
where the duality between $H$ and $U$ is used. 
Let $f,g\in U$, $u,v\in V$, $a,b\in H$ and $e\in W$. 
First note that 
$e(v^{\Br{1}}) \hs f^{\Br{2}}(a) \oeq{\eqref{bUVW}} e(v^{\Br{2}}) \hs f(a^{\Br{1}})$ implies  
\[  \label{vfa}
 f^{\Br{2}}\hsp(a) \,  v^{\Br{1}} =  f(a^{\Br{1}}) \, v^{\Br{2}}
\]  
and 
\[  \label{vfga}
 g^{\Br{2}_2}(b)\, f^{\Br{2}_1}(a) \, v^{\Br{1}_1\Br{1}_2}  
  \oeq{\eqref{vfa}}  
  g^{\Br{2}_2}(b)\, f(a^{\Br{1}_1}) \, v^{\Br{2}_1\Br{1}_2}  
  \oeq{\eqref{vfa}}  
   g(b^{\Br{1}_2})\, f(a^{\Br{1}_1}) \, v^{\Br{2}_1\Br{2}_2}. 
\]
Thus  
\begin{align*}
&f^{\br{2}\Br{2}_1}(a)\, g^{\br{1}\Br{2}_2}(b)\, v^{\Br{1}_1\Br{1}_2} 
\oeq{\eqref{phh},\eqref{vfga}} f(a^{\Br{1}_1 \br{2}})\, g(b^{\Br{1}_2\br{1}})\, v^{\Br{2}_1\Br{2}_2}\\
&\oeq{\eqref{WV}}  f(a^{ \br{2}\Br{1}_2})\, g( b^{\br{1}\Br{1}_1})\, v^{\Br{2}_1\Br{2}_2}
\oeq{\eqref{phh},\eqref{vfga}} f^{\Br{2}_2\br{2}}(a)\, g^{\Br{2}_1\br{1}}(b)\, v^{\Br{1}_1\Br{1}_2}, 
\end{align*} 
which proves \eqref{VW}, so $V$ becomes a left $U$-braided vector space with respect to 
the braidings $\Psi_{UU}$ and $\Psi_{UV}$. 

To prove compatibility with the multiplication $\ust$ from \eqref{ust} on $U$, we compute that 
\begin{align*}
&(f\ust g)^{\Br{2}}(a)\, v^{\Br{1}}  
\oeq{\eqref{vfa}} 
(f\ust g)(a^{\Br{1}})\, v^{\Br{2}} 
 \oeq{\eqref{ust}}  
 f(a^{\Br{1}}{\!}_{(1)}{\!}^{\br{2}})\, g(a^{\Br{1}}{\!}_{(2)}{\!}^{\br{1}}) \, v^{\Br{2}} \\
 &\oeq{\eqref{WV},\eqref{PVH}}  
 f(a_{(1)}{\!}^{\br{2}\Br{1}_2})\, g(a_{(2)}{\!}^{\br{1}\Br{1}_1}) \, v^{\Br{2}_1\Br{2}_2}  
 \oeq{\eqref{vfga}} f^{\Br{2}_2}(a_{(1)}{\!}^{\br{2}})\, g^{\Br{2}_1}(a_{(2)}{\!}^{\br{1}})\, v^{\Br{1}_1\Br{1}_2}\\ 
  &\oeq{\eqref{ust}}  (f^{\Br{2}_2}\ust  g^{\Br{2}_1})(a)\, v^{\Br{1}_1\Br{1}_2}. 
\end{align*} 
This implies \eqref{AWm}, 
and \eqref{1W} follows from 
$\eps^{\Br{2}}\hsp (a)\hs  v^{\Br{1}} \oeq{\eqref{vfa}} \eps(a^{\Br{1}})\hs v^{\Br{2}}   
\oeq{\eqref{PVH}} \eps(a)\hs v$ 
since $\vare$ yields the unit element in dual algebra $U \subset H'$.  

To prove the compatibility with the comultiplication of $U$, we proceed in the same manner.    
\begin{align*}
&f^{\Br{2}}{\!}_{(1)}\hsp (b) \, f^{\Br{2}}{\!}_{(2)}\hsp (a)\, v^{\Br{1}} 
 \oeq{\eqref{brcop}}  
 f^{\Br{2}}(b^{\br{1}}a^{\br{2}}) \,  v^{\Br{1}} 
 \oeq{\eqref{vfa}}  f((b^{\br{1}}a^{\br{2}})^{\Br{1}}) \, v^{\Br{2}} \\
  &  \oeq{\eqref{VAm}} f(b^{ \br{1} \Br{1}_1}a^{\br{2} \Br{1}_2}) \,  v^{\Br{2}_1\Br{2}_2}  
   \oeq{\eqref{WV}} f(b^{\Br{1}_2\br{1}}a^{\Br{1}_1\br{2}}) \,  v^{\Br{2}_1\Br{2}_2}  \\
&\oeq{\eqref{brcop}} f_{(1)}(b^{\Br{1}_2} ) \,  f_{(2)}(a^{\Br{1}_1})\, v^{\Br{2}_1\Br{2}_2} 
  \oeq{\eqref{vfga}}   f_{(1)}{\!}^{\Br{2}_2}(b ) \,  f_{(2)}{\!}^{\Br{2}_1} (a)\, v^{\Br{1}_1 \Br{1}_2} 
\end{align*} 
shows the first relation of \eqref{PHW}. 
The second relation of  \eqref{PHW} follows from 
$$
\vare(f^{\Br{2}})\, v^{\Br{1}} = f^{\Br{2}}(1) \, v^{\Br{1}} 
 \oeq{\eqref{vfa}} f(1^{\Br{1}}) \, v^{\Br{2}} 
  \oeq{\eqref{V1}}  f(1)  \, v  = \vare(f) \, v . 
$$

If $V$ is an algebra and $\Psi_{VH}$ is compatible with the multiplication, then 
\begin{align*} 
f^{\Br{2}}(a) \hs (vu)^{\Br{1}}  \oeq{\eqref{vfa}} f(a^{\Br{1}}) \hs (vu)^{\Br{2}} 
 \oeq{\eqref{AWm}}  f(a^{\Br{1}_1\Br{1}_2}) \hs v^{\Br{2}_2} u^{\Br{2}_1} 
  \oeq{\eqref{vfa}}  f^{\Br{2}_1 \Br{2}_2 }(a) \hs v^{\Br{1}_1} u^{\Br{1}_2} . 
\end{align*} 
Furthermore, if $1\in V$, we have 
$f^{\Br{2}}(a)\hs 1^{\Br{1}}  \oeq{\eqref{vfa}} f(a^{\Br{1}}) \hs 1^{\Br{2}} \oeq{\eqref{1W}} f(a) \hs 1$.
This and the previous computation show the compatibility of $\Psi_{UV}$ 
with the multiplication of $V$. 

Assume now that $V$ is a coalgebra and that $\Psi_{VH}$ is compatible with the comultiplication. Then  
\begin{align*}
f^{\Br{2}_1 \Br{2}_2 }(a) \hs v_{(1)}{\!}^{\Br{1}_1} \ot v_{(2)}{\!}^{\Br{1}_2} 
&\oeq{\eqref{vfa}} f(a^{\Br{1}_1 \Br{1}_2 }) \hs v_{(1)}{\!}^{\Br{2}_2} \ot v_{(2)}{\!}^{\Br{2}_1} 
\oeq{\eqref{PHW}} f(a^{\Br{1}}) \hs v^{\Br{2}}{\!}_{(1)} \ot v^{\Br{2}}{\!}_{(2)} \\
& \oeq{\eqref{vfa}} f^{\Br{2}}(a) \hs v^{\Br{1}}{\!}_{(1)} \ot v^{\Br{1}}{\!}_{(2)}, 
\end{align*} 
and 
$ f^{\Br{2}}\hsp(a)\, \vare(v^{\Br{1}}) \oeq{\eqref{vfa}}  f(a^{\Br{1}})\, \vare(v^{\Br{2}})
\oeq{\eqref{PHW}}  f(a)\, \vare(v)$. Hence $\Psi_{UV}$ is also compatible 
with the comultiplication of $V$. 

The opposite versions are proven analogously, and the last part of the lemma 
follows from the first part by applying Lemma \ref{Linv}.  
\end{proof} 

Given a left $H$-braided vector space $V$ and subspaces $U\subset H'$  and $W\subset V'$ 
satisfying the assumptions of Lemmas \ref{dbvs} and \ref{LUVH}, 
there are two ways of constructing a braiding $\Psi_{UW}$ on $U \ot W$, either 
by the restriction of $\Psi_{UV'}$ from \eqref{bHV} or by the restriction of $\Psi_{H'W}$ from \eqref{bUVW}. 
The next lemma shows that both constructions coincide whenever one of them can be realized.   
Equally, we can use either $\Psi_{VU}^\circ$ or $\Psi_{WH}^\circ$ to construct a braiding on $U\ot W$. 
In this case, the resulting braiding will be denoted by $\Psi_{UW}^{\bullet}$. 
Analogous results hold for right $H$-braided vector spaces. 

\begin{lemma} \label{bUW}
Let $V$ be a left $H$-braided vector space. 
Assume that $U\subset H'$ and $W\subset V'$ satisfy the conditions of 
Lemmas \ref{dbvs} and \ref{LUVH} such that the braidings $\Psi_{WH}$ and $\Psi_{VU}$ 
are well-defined. 
If either $\Psi_{H'W}\!\!\upharpoonright_{U\ot W}: U\ot W\to W\ot U$ or 
$\Psi_{UV'}\!\!\upharpoonright_{U\ot W}: U\ot W \to W\ot U$ is bijective, 
then so is the other and 
$\Psi_{H'W}\!\!\upharpoonright_{U\ot W} \,= \Psi_{UV'}\!\!\upharpoonright_{U\ot W}\,=: \Psi_{UW}$. 

Similarly, if $\Psi_{HW}^\circ$ and $\Psi_{UV}^\circ$ are well-defined 
and if either $\Psi_{V'U}^{\bullet}\!\!\upharpoonright_{W\ot U}$ or 
$\Psi_{WH'}^{\bullet}\!\!\upharpoonright_{W\ot U}$ yields a bijective map between 
$W\ot U$ and $U\ot W$, 
then $\Psi_{V'U}^{\bullet}\!\!\upharpoonright_{W\ot U}\,
= \Psi_{WH'}^{\bullet}\!\!\upharpoonright_{W\ot U}\, =: \Psi_{WU}^{\bullet}$, 
where  
\begin{align} \label{VUb}
&\Psi_{V'U}^{\bullet}: V' \ot U \to ( H \ot V)', &
&\Psi_{V'U}^{\bullet}(e\ot f)(a\ot v) :=  \ipp{a\ot e}{\Psi_{UV}^\circ(f\ot v)},\\
&\Psi_{WH'}^{\bullet} : W\ot H' \to (H\ot V)',  &
&\Psi_{WH'}^{\bullet}(e\ot f)(a\ot v) := \ipp{f\ot v}{\Psi_{HW}^\circ(a\ot e)}.  \label{WHb}
\end{align} 

If \,$V$ is a right $H$-braided vector space, then 
the analogous statements holds for $\Psi_{WU}$ and $\Psi_{UW}^{\bullet}$ 
under homologous assumptions, where 
\[
\Psi_{UW}^{\bullet}(f\ot e)(a\ot v) = \ipp{e\ot a}{\Psi_{VU}^\circ(v\ot f)} 
=  \ipp{v\ot f}{\Psi_{WH}^\circ(e\ot a)}. 
\]
\end{lemma} 
\begin{proof} 
The claim $\Psi_{H'W}\!\!\upharpoonright_{U\ot W} = \Psi_{UV'}\!\!\upharpoonright_{U\ot W}$ 
follows from the equation 
\begin{align*} 
\Psi_{UV'}(f\ot e)(v\ot a) 
&\oeq{\eqref{bHV}}  e(v^{\Br{2}}) \, f^{\Br{1}}(a) 
\oeq{\eqref{bVpH}}  e(v^{\Br{1}}) \, f(a^{\Br{2}}) 
\oeq{\eqref{bHpV}}   e^{\Br{2}}(v) \, f(a^{\Br{1}}) \\
&\oeq{\eqref{bUVW}} \Psi_{H'W}( f \ot e) (v \ot a) 
\end{align*} 
for $a\in H$, $v\in V$, $f\in U$ and $e\in W$. 
Similarly,  
\begin{align}  \nonumber
\Psi_{V'U}^{\bullet}(e\ot f)(a\ot v) 
&\oeq{\eqref{VUb}}  e(v^{\cBr{1}}) \, f^{\cBr{2}}\!(a) 
\oeq{\eqref{cPHV}}  e(v^{\pBr{2}}\,) \, f(a^{\pBr{1}}\,) 
\oeq{\eqref{cHV}}   e^{\cBr{1}}\!(v) \, f(a^{\cBr{2}}) \\ \nonumber 
&\oeq{\eqref{WHb}} \Psi_{WH'}^{\bullet}( e \ot f) (a \ot v) 
\end{align} 
implies that $\Psi_{V'U}^{\bullet}\!\!\upharpoonright_{W\ot U}
= \Psi_{WH'}^{\bullet}\!\!\upharpoonright_{W\ot U}$. 
The opposite versions are proven analogously. 
\end{proof} 

The next theorem shows how to transform a comodule $V$ of a coalgebra $H$ 
into a module of a dual algebra of $H$.  
Note  that we will use again a braiding that is constructed 
from the inverse of the given one as it happened in  Theorem \ref{dH}. 
The same observation can be made for succeeding results.   

\begin{theorem}   \label{mcom}
Let $H$ be a braided coalgebra and $V$ a braided right $H$-comodule 
with co\-action $\rho_R: V\to V\ot H$. 
Let $U\subset H'$ be a non-degenerate (unital) subalgebra with product given by \eqref{us}
such that $\Psi_{UU}$ and $\Psi_{UV}^\circ$ 
introduced in \eqref{pbbraid} and \eqref{cPUV}, respectively, are bijective.  
Then $V$ becomes a braided left $U$-module with respect to the 
braiding $\Psi_{UV}^\circ$ and the left action $\nu_L : U\ot V \to V$ given by 
\[  \label{nuL}
 \nu_L(f\ot v):=  \ip{f}{v_{(1)}{\!}^{\pBr{1}\,\hs}}\hs v_{(0)}{\!}^{\pBr{2}}\hs. 
\]
If $H$ is a braided bialgebra, $V$ is a braided right $H$-comodule algebra, 
and  $U$ is a left dual of $H$ as in Theorem \ref{dH}, 
then the action $\nu_L$ turns $V$ into a braided left $U$-module algebra.  

In case $V$ is a braided left $H$-comodule, the analogous statements for the opposite versions hold 
under homologous assumptions 
for the right action $\nu_R: V\ot U\to V$ given by 
\[ \label{nuR}
\nu_R(v\ot f):=  \ip{f}{v_{(-1)}{\!}^{\pBr{2}\,\hs}}\hs v_{(0)}{\!}^{\pBr{1}}\hs, 
\]
and with respect to the braiding $\Psi_{VU}^\circ$. 
 \end{theorem} 
\begin{proof} 
From Proposition \ref{lust} and Lemma \ref{LUVH}, 
we conclude that $U$ is a braided algebra with respect to braiding $\Psi_{UU}$ 
and $V$ is a left $U$-braided vector space 
such that the braiding $\Psi_{UV}^\circ$ is compatible with the multiplication of $U$. 
To show that the action $\nu_L$ equips $V$ with the structure of a left braided $U$-module, 
we need to prove \eqref{anu} and \eqref{bnu}. 

Let $f,g\in U$ and $v\in V$. Using the notation \eqref{na}, we compute that 
\begin{align*}
&f\la(g\la v)  \oeq{\eqref{nuL}} 
g(v_{(1)}{\!}^{\pBr{1}_1}) \, f(v_{(0)}{\!}^{\pBr{2}_1}{\hsp}_{(1)}{\hsp}^{\pBr{1}_2}) \, 
v_{(0)}{\!}^{\pBr{2}_1}{}_{(0)}{\!}^{\pBr{2}_2} \\ 
&\oeq{\eqref{rohD},\eqref{invPHV} } 
g(v_{(2)}{\!}^{\pbr{1}\pBr{1}_1}) \, f(v_{(1)}{\!}^{\pbr{2}\pBr{1}_2}) \, v_{(0)}{\!}^{\pBr{2}_1\pBr{2}_2} 
 \oeq{\eqref{WV}} 
 g(v_{(2)}{\!}^{\pBr{1}_2\pbr{1}}) \, f(v_{(1)}{\!}^{\pBr{1}_1\pbr{2}}) \, v_{(0)}{\!}^{\pBr{2}_1\pBr{2}_2} \\
 &\oeq{\eqref{PVH}} 
 g(v_{(1)}{\!}^{\pBr{1}}{\!}_{(2)}{\!}^{\pbr{1}}) \, f(v_{(1)}{\!}^{\pBr{1}}{\!}_{(1)}{\!}^{\pbr{2}}) \, v_{(0)}{\!}^{\pBr{2}} 
 \,\oeq{\eqref{us}} f\ust g(v_{(1)}{\!}^{\pBr{1}}\,) \, v_{(0)}{\!}^{\pBr{2}} 
  \,\oeq{\eqref{nuL}} (f\ust g) \la v. 
\end{align*} 
This yields \eqref{anu}. 

Note that, since $V'$ separates the points of $V$, 
Equations \eqref{cPHV}   and \eqref{cPUV} 
give 
\[  \label{fav}
f^{\cBr{2}}\! (a)\, v^{\cBr{1}} =  f(a^{\pBr{1}\,\hs})\, v^{\pBr{2}}\, , \quad f\in U,\ v\in V, \ a\in H. 
\]
Therefore, for $f,g\in U$, $v\in V$ and $a\in H$, 
\begin{align*}
&f^{\br{2}\cBr{2}\!\hsp}(a)\hs (g^{\br{1}}\la v^{\cBr{1}\!}) \oeq{\eqref{nuL}} 
f^{\br{2}\cBr{2}\!\hsp}(a)\, g^{\br{1}\hsp}(v^{\cBr{1}\!}{\!}_{(1)}{\!}^{\pBr{1}}\,)\, v^{\cBr{1}\!}{\!}_{(0)}{\!}^{\pBr{2}} \\
&\oeq{\eqref{fav}} 
f^{\br{2}}(a^{\pBr{1}_1})\, g^{\br{1}}(v^{\pBr{2}_1}{\hsp}_{(1)}{\!}^{\pBr{1}_2})\, 
v^{\pBr{2}_1}{\hsp}_{(0)}{\!}^{\pBr{2}_2}\,
\oeq{\eqref{invPHV}} 
f^{\br{2}}(a^{\pbr{1}\pBr{1}_1}\,)\, g^{\br{1}}(v_{(1)}{\!}^{\pbr{2}\pBr{1}_2}\,)\, v_{(0)}{\!}^{\pBr{2}_1\pBr{2}_2}\\
& \oeq{\eqref{WV}} 
f^{\br{2}}(a^{\pBr{1}_2 \pbr{1}})\, g^{\br{1}}(v_{(1)}{\!}^{\pBr{1}_1\pbr{2}})\, v_{(0)}{\!}^{\pBr{2}_1\pBr{2}_2} 
 \oeq{\eqref{brpbr},\eqref{fagb}} 
 f(a^{\pBr{1}_2 })\, g(v_{(1)}{\!}^{\pBr{1}_1})\, v_{(0)}{\!}^{\pBr{2}_1\pBr{2}_2} \\
 & \oeq{\eqref{nuL}} 
  f(a^{\pBr{1}}\,)\,  (g\la v)^{\pBr{2}} \, 
  \oeq{\eqref{fav}}  f^{\cBr{2}\!}(a)\, (g\la v)^{\cBr{1}} , 
\end{align*} 
which implies \eqref{bnu}. 

Now assume that 
$H$ is a braided bialgebra, $V$ is a braided right $H$-comodule algebra, and  $U$ is a left dual of $H$. 
Then, for $u,v\in V$ and $f\in U$, 
\begin{align*}
&(f_{(1)}\la u^{\cBr{1}\!}) \hs (f_{(2)}{\!}^{\cBr{2}\!}\la v) \oeq{\eqref{nuL}} 
f_{(1)}( u^{\cBr{1}\!}{\!}_{(1)}{\!}^{\pBr{1}_2})\, f_{(2)}{\!}^{\cBr{2}\!}( v_{(1)}{\!}^{\pBr{1}_1})\, 
u^{\cBr{1}\!}{\!}_{(0)}{\!}^{\pBr{2}_2} \, v_{(0)}{\!}^{\pBr{2}_1} \\
&\oeq{\eqref{fav}} 
f_{(1)}( u^{\pBr{2}_2}{}_{(1)}{\!}^{\pBr{1}_3})\, f_{(2)}( v_{(1)}{\!}^{\pBr{1}_1 \pBr{1}_2})\,
u^{\pBr{2}_2}{}_{(0)}{\!}^{\pBr{2}_3} \,\, v_{(0)}{\!}^{\pBr{2}_1} \\
&\oeq{\eqref{invPHV}} 
f_{(1)}( u_{(1)}{\!}^{\pbr{2}\pBr{1}_3})\, f_{(2)}( v_{(1)}{\!}^{\pBr{1}_1\pbr{1} \pBr{1}_2})\, 
u_{(0)}{\!}^{\pBr{2}_2\pBr{2}_3} \,\, v_{(0)}{\!}^{\pBr{2}_1} \\
&\oeq{\eqref{WV}} 
f_{(1)}( u_{(1)}{\!}^{\pBr{1}_2\pbr{2}})\, f_{(2)}( v_{(1)}{\!}^{\pBr{1}_1\pBr{1}_3\pbr{1}} )\, 
u_{(0)}{\!}^{\pBr{2}_2\pBr{2}_3} \,\, v_{(0)}{\!}^{\pBr{2}_1} \\
&\oeq{\eqref{brpbr},\eqref{brcop}}  
f( u_{(1)}{\!}^{\pBr{1}_2} \, v_{(1)}{\!}^{\pBr{1}_1\pBr{1}_3})\,u_{(0)}{\!}^{\pBr{2}_2\pBr{2}_3}\,v_{(0)}{\!}^{\pBr{2}_1} 
\oeq{\eqref{VAm}} 
f\big( (u_{(1)} \hs v_{(1)}{\!}^{\pBr{1}_1})^{\pBr{1}_2}\big) 
u_{(0)}{\!}^{\pBr{2}_2} \,v_{(0)}{\!}^{\pBr{2}_1} \\
& \oeq{\eqref{ipBr}} \!  
f\big(\hsp (u_{(1)}{\!}^{\Br{2}\pBr{1}_1} \hs v_{(1)}{\!}^{\pBr{1}_2})^{\pBr{1}_3}\hsp\big) 
u_{(0)}{\!}^{\pBr{2}_3} \hs v_{(0)}{\!}^{\Br{1}\pBr{2}_1 \pBr{2}_2} 
 \oeq{\eqref{VAm}}  
 f\big(\hsp (u_{(1)}{\!}^{\Br{2}} \hs v_{(1)})^{\pBr{1}_1\pBr{1}_2}\hsp\big) 
u_{(0)}{\!}^{\pBr{2}_2} \hs v_{(0)}{\!}^{\Br{1}\pBr{2}_1} \\
&  \oeq{\eqref{AWm}}  
f\big( (u_{(1)}{\!}^{\Br{2}} \hs v_{(1)})^{\pBr{1}\,}\hs\big) 
(u_{(0)} \,v_{(0)}{\!}^{\Br{1}})^{\pBr{2}\,} 
 \oeq{\eqref{rhoRm}}  
 f\big( (uv)_{(1)}{\!}^{\pBr{1}\,}\hs\big) (uv)_{(0)}{\!}^{\pBr{2}\,} 
 \oeq{\eqref{nuL}}  f\la (uv). 
\end{align*} 
This proves \eqref{num}. 
Furthermore, if $1\in V$, then 
$f\la 1 \oeq{\eqref{rR1}} f(1^{\pBr{1}\,\hs}) \, 1^{\pBr{2}}  \oeq{\eqref{1W}} f(1) \, 1 = \eps(f)\,1$. 
Since, by Lemma~\ref{LUVH}, 
$\Psi_{UV}^\circ$ is compatible with the comultiplication of $U$ and with 
the multiplications of $U$ and $V$, 
we conclude that $V$ is a braided left $U$-module algebra.

The opposite versions are shown analogously. 
\end{proof} 

Let $H$ be a braided bialgebra and $U$ a left dual of $H$. 
Since $U\subset H'$ is an algebra, there is a natural left (resp.\ right) $U$-action on $H$ 
given by right (resp.\ left) multiplication on $U$. On the other hand, the coproduct on $H$ 
equips $H$ trivially with the structure of a right (resp.\ left) $H$-comodule so that we may consider the 
left (resp.\ right) $U$-action on $H$ described in Theorem~\ref{mcom}. 
The next corollary shows that these actions coincide and turn $H$ into a 
left (resp.\ right) $U$-module algebra. Similar results hold for $U$ and $H$ interchanged. 
\begin{corollary} \label{nat}
Let $H$ be a braided bialgebra and let $U$ be a left dual of $H$ 
as in Theorem~\ref{dH}. 
Then the natural left $U$-action \,$\la: U\ot H \rightarrow H$ defined by 
 $f(g\la a):= (f g)(a)$, $a\in H$, $f,g\in U$, 
satisfies 
\[ \label{glaa}
g\la a=  \ip{g}{a_{(2)}{\!}^{\pbr{1}\,\hsp}}\hs a_{(1)}{\!}^{\pbr{2}}
\] 
and turns $H$ into a left $U$-module algebra. 
The natural right $U$-action \,$\ra: H\ot U \rightarrow H$  defined by  
$g(a \ra f):= (f g)(a)$  
satisfies 
\[ \label{graa}
a \ra f =  \ip{f}{a_{(1)}{\!}^{\pbr{2}\,\hsp}}\hs a_{(2)}{\!}^{\pbr{1}}
\]
and  turns $H$ into a right $U$-module algebra. 

The same formulas with $a\in U$, $f,g\in H$ 
hold for the natural left (resp.\ right) $H$-action on $U$ 
and  turn $U$ into a left (resp.\ right) $H$-module algebra.  
\end{corollary}
\begin{proof} 
Setting $V:=H$ in Theorem \ref{mcom} shows that the right-hand side in \eqref{glaa} (resp.\ \eqref{graa})  
defines a left (resp.\ right) $U$-action on $H$ such that $H$ becomes a left (resp.\ right) $U$-module algebra. 
The equalities in \eqref{glaa} and \eqref{graa}  follow from 
\eqref{mD} for the braiding $\Psi_{UH}^\circ$ by applying \eqref{cphh}. 
The proof for the $H$-actions on $U$ is identical but 
uses  \eqref{Dm} instead of \eqref{mD}. 
\end{proof}

Theorem \ref{mcom} shows how to turn a comodule of a braided coalgebra into a module 
of an appropiate dual algebra. The dual construction corresponds to turning a module of a braided 
algebra into a comodule of a dual coalgebra.  
This will be done in the next theorem. 
Similar to Proposition \ref{lcop}, we will need an additional  
condition to ensure that the coaction belongs to the correct (algebraic) tensor product. 
\begin{theorem}   \label{mmm}
Let $H$ be a braided unital algebra and $V$ a braided left $H$-module 
with action \mbox{$\nu_L : H\ot V\to V$} denoted by $\nu_L(a\ot v) := a \la v$.  
Let $U$ be a non-degenerate subcoalgebra of $H^{\uc}$ as defined in Proposition \ref{lcop} 
such that $\Psi_{UU}$ and $\Psi_{UV}^\circ$
given in \eqref{pbbraid} and \eqref{cPUV}, respectively, are bijective.  
Consider the linear map 
\[  \label{rR} 
\rho_R : V \lra (V'\ot H)', \quad \rho_R(v)(e\ot a) := e(a^{\pBr{1}}\la v^{\pBr{2}\,\hs}),  
\]
where $v\in V$, $a\in H$ and $e\in V'$. If 
\[ \label{rRVU}
\rho_R : V \lra V\ot U \subset (V'\ot H)', 
\]
then $\rho_R$ yields a right $U$-coaction on $V$ such that $V$ becomes a 
braided right $U$-comodule with respect to 
the braiding $\Psi_{UV}^\circ$ and 
the coproduct $\uD_{\hs -2}:= \Psi_{UU}^{-2}\circ \uD$ on $U$, where 
$\uD$ denotes the coproduct 
introduced in \eqref{brcop}.  

If $H$ is a braided bialgebra, $V$ is a braided left $H$-module algebra, 
and  $U$ is a left dual of $H$ as in Theorem \ref{dH}, 
then the coaction $\rho_R$ turns $V$ into a braided right 
$U^{(2,-2)}$-comodule algebra.  

In case $V$ is a braided right $H$-module, then  
the analogous statements hold under homologous assumptions for the opposite versions 
with respect to the left coaction determined by 
\[ \label{muR}
\rho_L: V\lra U\ot V\subset (H\ot V')', \quad \rho_L(v)(a\ot e) :=  e(v^{\pBr{1}}\ra a^{\pBr{2}\,\hs}), 
\]
and again for 
the braided coalgebra $(U,\uD_{\hs -2},\eps)$ and the braided bialgebra $U^{(2,-2)}$. 
 \end{theorem} 
\begin{proof}
As shown in Proposition \ref{lcop}, the braiding $\Psi_{UU}$ 
equips $(U, \uD, \eps)$ with the structure of a braided coalgebra. 
From Lemma \ref{LUVH}, 
we know that $\Psi_{UV}^\circ$ turns $V$ into a left $U$-braided 
vector space such that the braiding is compatible with the comultiplication of $U$ 
and, if defined, with the multiplications of $U$ and $V$.
Furthermore, by Proposition \ref{lemmD}, 
the compatibility conditions are also fulfilled with respect to  
the modified coproduct $\uD_{\hs -2}=\Psi_{UU}^{-2}\circ  \uD$ 
and the modified product $m_2= m\circ \Psi_{UU}^{2}$. 

Throughout this proof, let  $u,v\in V$, $a,b\in H$ and $f,g\in U$.
To show that $\rho_R$ turns $V$ into a braided right $U$-module, 
it remains to prove \eqref{rohD} and \eqref{brR}. 
First note that, since $V'$ separates the points of $V$, 
\eqref{rR}  is equivalent to 
\[    \label{SrR} 
v_{(0)}\, v_{(1)}\hsp(a) = a^{\pBr{1}}\la v^{\pBr{2}\,\hs}.  
\]
Moreover, to distinguish between $\uD_{\hs -2}$ and 
coproduct $\uD$ on $U$ defined in Proposition \ref{lcop}, 
we will employ the Sweedler notation 
$\uD_{\hs -2}(f) := f_{(1)^{\backprime}} \ot f_{(2)^{\backprime}}$. 
Then it follows from \eqref{UU}, \eqref{pbbraid} and \eqref{brcop} that 
\[ \label{Swp}
f_{(1)^{\backprime}}(b)\, f_{(2)^{\backprime}}(a)  = f(b^{\pbr{1}} a^{\pbr{2}}). 
\] 
Using  Lemma \ref{Linv} for the relations concerning the inverse braiding, we obtain 
\begin{align*} 
&(v_{(0)})_{(0)}\, v_{(1)}\hsp(a)\, (v_{(0)})_{(1)}\hsp(b) \oeq{\eqref{SrR}} 
b^{\pBr{1}_2}\hsp\la\hsp (a^{\pBr{1}_1}\hsp\la\hsp v^{\pBr{2}_1})^{\pBr{2}_2} 
\oeq{\eqref{nulinv}} b^{\pBr{1}_2\pbr{1}}\hsp\la\hsp (a^{\pBr{1}_1\pbr{2}}\hsp\la\hsp v^{\pBr{2}_1\pBr{2}_2}) \\
&\oeq{\eqref{VW}} 
b^{\pbr{1}\pBr{1}_1}\hsp\la\hsp (a^{\pbr{2}\pBr{1}_2}\hsp\la\hsp v^{\pBr{2}_1\pBr{2}_2}) 
\oeq{\eqref{anu}} (b^{\pbr{1}\pBr{1}_1}a^{\pbr{2}\pBr{1}_2})\hsp\la\hsp v^{\pBr{2}_1\pBr{2}_2} 
\oeq{\eqref{AWm}} 
 (b^{\pbr{1}}a^{\pbr{2}})^{\pBr{1}}\hsp\la\hsp v^{\pBr{2}} \\
 &\oeq{\eqref{SrR}} v_{(0)}\, v_{(1)}\hsp(b^{\pbr{1}}a^{\pbr{2}}) 
 \oeq{\eqref{Swp}}  v_{(0)}\, \hs(v_{(1)})_{(1)^{\backprime}}\hsp (b) \, \hs(v_{(1)})_{(2)^{\backprime}}\hsp(a). 
\end{align*} 
Therefore 
$$
v_{(0)}\ot (v_{(1)})_{(1)^{\backprime}} \ot (v_{(1)})_{(2)^{\backprime}}=(v_{(0)})_{(0)}\ot   (v_{(0)})_{(1)} \ot  v_{(1)}  
=: v_{(0)} \ot v_{(1)} \ot v_{(2)}, 
$$
which shows the first relation of \eqref{rohD}. 
The second relation of \eqref{rohD} follows 
from the definition of $\vare$ in Proposition \ref{lcop} and 
$$
v_{(0)}\, \vare(v_{(1)} ) = v_{(0)}\, v_{(1)}\hsp(1) \oeq{\eqref{SrR}}  1^{\pBr{1}}\la v^{\pBr{2}\,\hs} 
\oeq{\eqref{1W}}1\la v \oeq{\eqref{anu}} v. 
$$
Furthermore, 
\begin{align*} 
& v_{(0)}{\!}^{\cBr{1}} \, v_{(1)}{\!}^{\br{1}}\hsp(b)\, f^{\cBr{2}\br{2}}\hsp(a) 
\oeq{\eqref{phh}}  
v_{(0)}{\!}^{\cBr{1}} \, v_{(1)}\hsp(b^{\br{1}})\, f^{\cBr{2}}\hsp(a^{\br{2}}) 
\oeq{\eqref{fav}} v_{(0)}{\!}^{\pBr{2}} \, v_{(1)}\hsp(b^{\br{1}})\, f^{}\hsp(a^{\br{2}\pBr{1}\,\hs})  \\
&\oeq{\eqref{SrR}} (b^{\br{1}\pBr{1}_1}  \la  v^{\pBr{2}_1})^{\pBr{2}_2}\, \, f(a^{\br{2}\pBr{1}_2}) 
\oeq{\eqref{nulinv}}   (b^{\br{1}\pBr{1}_1 \pbr{2}}  \la  v^{\pBr{2}_1\pBr{2}_2}) \, f(a^{\br{2}\pBr{1}_2 \pbr{1}}) \\
&\oeq{\eqref{VW},\eqref{brpbr}} (b^{\pBr{1}_2 }  \la  v^{\pBr{2}_1\pBr{2}_2}) \, f(a^{\pBr{1}_1}) 
\oeq{\eqref{SrR}}   v^{\pBr{2}}{\!}_{(0)}\, v^{\pBr{2}}{\!}_{(1)}\hsp(b)\,  f(a^{\pBr{1}}\,) 
\oeq{\eqref{fav}} v^{\cBr{1}\!}{\!}_{(0)}\,  v^{\cBr{1}\!}{\!}_{(1)}\hsp(b) \,f^{\cBr{2}\!}\hsp(a), 
\end{align*} 
which proves \eqref{brR}. Hence $\rho_R$ turns $V$ into a braided right $U$-module. 

Now let $H$ be a braided bialgebra, $V$ a braided right $H$-module algebra,  
and $U$ a left dual of $H$ as in Theorem \ref{dH}. 
The compatibility of the braiding $\Psi_{UV}^\circ$ with the 
(modified) comultiplication on $U$ and the (modified) multiplications on $U$ and $V$ 
has been discussed in the beginning of the proof. 
To show that $V$ becomes a braided right $U^{(2,-2)}$-comodule algebra, we need 
to verify \eqref{rhoRm}.  

First note that 
\begin{eqnarray} \nonumber 
\um_{\hs 2}(f\ot g)(a)\!\!\!  &\oeq{\eqref{mkDn1},\eqref{us}} &\!\!\! 
\ipp{\Psi_{UU}^2(f\ot g)}{\Psi_{HH}^{-1}\! \circ\! \Delta(a)}
\oeq{\eqref{UU},\eqref{pbbraid}}  \ipp{f\ot g}{\Psi_{HH}\! \circ\! \Delta(a)} \\
&\oeq{\eqref{braidip}} &\!\!\! f({a_{(1)}}^{\!\br{2}})\hs g({a_{(2)}}^{\!\br{1}})   
\oeq{\eqref{ust}} f\ust g(a).      \label{vust}
\end{eqnarray}
 Starting with the right hand side of \eqref{rhoRm}, we compute that 
  \begin{align*}
 &u_{(0)} \hs v_{(0)}{\!}^{\cBr{1}} (u_{(1)}{\!}^{\cBr{2}}\ust v_{(1)})(a) 
 \oeq{\eqref{vust}} 
 u_{(0)} \hs v_{(0)}{\!}^{\cBr{1}} \, u_{(1)}{\!}^{\cBr{2}\!}\hsp(a_{(1)}{\!}^{\br{2}})\, 
 v_{(1)}\hsp (a_{(2)}{\!}^{\br{1}}) \\
& \oeq{\eqref{fav}} 
 u_{(0)} \hs v_{(0)}{\!}^{\pBr{2}\,} \, u_{(1)}\hsp(a_{(1)}{\!}^{\br{2}\pBr{1}\hs}\,)\, v_{(1)}\hsp (a_{(2)}{\!}^{\br{1}}) 
 \oeq{\eqref{SrR}} 
  u_{(0)} \hs \big(a_{(2)}{\!}^{\br{1}\pBr{1}_1} \la v^{\pBr{2}_1}\big)^{\!\pBr{2}_2}  \,\hs 
 u_{(1)}\hsp(a_{(1)}{\!}^{\br{2}\pBr{1}_2})\\
& \oeq{\eqref{nulinv}} 
 u_{(0)} \hs \big(a_{(2)}{\!}^{\br{1}\pBr{1}_1\pbr{2}} \la v^{\pBr{2}_1\pBr{2}_2}\big)^{\!}  \,\hs 
 u_{(1)}\hsp(a_{(1)}{\!}^{\br{2}\pBr{1}_2\pbr{1}}) \\
&\oeq{\eqref{VW},\eqref{brpbr}} 
u_{(0)} \hs \big(a_{(2)}{\!}^{\pBr{1}_2} \la v^{\pBr{2}_1\pBr{2}_2}\big)^{\!}  \,\hs 
 u_{(1)}\hsp(a_{(1)}{\!}^{\pBr{1}_1})
 \oeq{\eqref{PHW}} 
 u_{(0)} \hs \big(a^{\pBr{1}}{\!}_{(2)} \la v^{\pBr{2}\,}\hs\big)^{\!}  \,\hs 
 u_{(1)}\hsp(a^{\pBr{1}}{\!}_{(1)}) \\
 &\oeq{\eqref{SrR}} 
  (a^{\pBr{1}_1}{\hsp}_{(1)}{\!}^{\pBr{1}_2}\la u^{\pBr{2}_2}) \hs (a^{\pBr{1}_1}{\hsp}_{(2)} \la v^{\pBr{2}_1})  
  \oeq{\eqref{PHW}}  
   (a^{\pBr{1}_1\pBr{1}_2}{\hsp}_{(1)}\la u^{\pBr{2}_2\Br{1}}) \hs 
   (a^{\pBr{1}_1\pBr{1}_2}{\hsp}_{(2)}{\hsp}^{\Br{2}} \la v^{\pBr{2}_1})  \\
  &\oeq{\eqref{num}} 
   a^{\pBr{1}_1\pBr{1}_2} \la (u^{\pBr{2}_2}\hs v^{\pBr{2}_1} ) 
   \oeq{\eqref{VAm}} 
    a^{\pBr{1}\,} \la (u\hs v)^{\pBr{2}\,}  
    \oeq{\eqref{SrR}} 
    (u\hs v)_{(0)} \,  (u\hs v)_{(1)}\hsp(a), 
 \end{align*}
which implies \eqref{rhoRm}.  
Finally, if $1\in V$, then 
$1_{(0)}\, 1_{(1)}{\hsp}(a) \oeq{\eqref{SrR}}  a^{\pBr{1}\,} \la 1^{\pBr{2}\,}  
\oeq{\eqref{V1}}  a\la 1 \oeq{\eqref{nu1}}  1 \,\vare(a)$, from which we conclude 
that $\rho_R(1)=1\ot 1$. This finishes the proof that $V$ 
is a  braided right $U^{(2,-2)}$-comodule algebra. 
 
The opposite versions are proven analogously. 
\end{proof}

Taking $V:=H$ in the previous theorem and the multiplication of $H$ as left action, 
the right coaction $\rho_R : H \to H \ot U$ is determined by 
$$
\rho_R(a)(f\ot b)  \oeq{\eqref{rR}}  f(b^{\pbr{1}} a^{\pbr{2}}) 
\oeq{\eqref{brcop}} \ipp{\uD_{-2}(f)}{a\ot b}, 
$$
which is equivalent to 
$$
a_{(0)}(f)\, a_{(1)} = f_{(2)^{\backprime}}(a) \, f_{(1)^{\backprime}}, \qquad 
\uD_{\hs -2}(f) := f_{(1)^{\backprime}} \ot f_{(2)^{\backprime}}. 
$$
However, if we consider $H$ as a braided algebra with respect to the inverse braiding $\Psi_{HH}^{-1}$, then 
$\rho_R(a)(f\ot b) = f(b^{\br{1}} a^{\br{2}})$ and thus 
$$
f(a_{(0)})\, a_{(1)} = f_{(2)}(a) \, f_{(1)}, \qquad \uD(f):= f_{(1)}\ot f_{(2)}. 
$$
The corresponding left $U$-coaction $\rho_L(a) (b\ot f) =  f(a^{\br{1}} b^{\br{2}})$
satisfies 
$$
f(a_{(0)})\, a_{(-1)} = f_{(1)}(a) \, f_{(2)}.  
$$
These observations may be viewed as the dual version of Corollary \ref{nat}. 

Note that the coproduct of the dual coalgebra $(U,\uD,\eps)$ 
had to be changed to $\uD_{\hs -2}$ in Theorem \ref{mmm}. 
Therefore, combining repeatedly Theorems \ref{mcom} and \ref{mmm} may 
turn $V$ into a (co)module for a whole family of (co)algebras, 
each time with respect to a potentially different (co)action.   
The starting point for this observation is the next corollary. 

\begin{corollary} \label{corind} 
(i) \,Let $(H,\Delta, \eps)$ be a braided coalgebra and $V$ a braided right $H$-comodule 
with co\-action $\rho_{R} : V\to V\ot H$. 
Let $U\subset H'$ be a non-degenerate unital subalgebra with product given by \eqref{us}
such that $\Psi_{UU}$,  $\Psi_{UH}^\circ$ and $\Psi_{UV}^\circ$ 
introduced in \eqref{pbbraid}, \eqref{ipbb} and \eqref{cPUV}, respectively, are bijective.   
Assume that 
\[ \label{PHVcc}
\Psi_{HV}^{\circ\circ}: H \ot V \lra (V'\ot U)', \quad 
\Psi_{HV}^{\circ\circ}(a\ot v)(e\ot f):=\ipp{e\ot a}{\Psi_{UV}^{\circ-1}(v\ot f )}, 
\]
yields a bijection $\Psi_{HV}^{\circ\circ}: H \ot V \to V\ot H \subset (V'\ot U)'$. 
If the map 
\[ \label{coactUH}
\rho_{R}^{\backprime} : V\to (V'\ot U)', \quad 
\rho_{R}^{\backprime}(v)(e\ot f) := e(v_{(0)})\, f^{{\cBr{1}}^\backprime}\!(v_{(1)}{\!}^{{\cBr{2}}^\backprime}\hs)
\] 
satisfies $\rho_{R}^{\backprime} : V\to  V\ot H  \subset (V'\ot U)'$, then 
it defines a right $H$-coaction that turns $V$ into a braided right $H$-comodule with respect to the 
braiding $\Psi_{HV}^{\circ\circ}$ and the 
coproduct $\Delta_{-2} := \Psi_{HH}^{-2} \circ \Delta$ on $H$. 

If $H$ is a braided bialgebra, $V$ is a braided right $H$-comodule algebra, 
and  $U$ is a left dual of $H$ as in Theorem \ref{dH}, 
then the right coaction defined in \eqref{coactUH} 
turns $V$ into a braided right $H^{(2,-2)}$-comodule algebra. 

(ii) \,Let $(H,m)$ be a braided unital algebra 
and \,$V$ a braided left $H$-module 
with left action $\nu_{L} : H\ot V\to V$.    
Let $U$ be a non-degenerate subcoalgebra of $H^{\uc}$ as defined in Proposition~\ref{lcop} 
such that $\Psi_{UU}$ and $\Psi_{UV}^\circ$
given in \eqref{pbbraid} and \eqref{cPUV}, respectively, are bijective, 
and suppose that  $\rho_R$ defined in \eqref{rR} satisfies \eqref{rRVU}. 
Assume that the map $\Psi_{HV}^{\circ\circ}$ introduced in \eqref{PHVcc} yields a bijection 
$\Psi_{HV}^{\circ\circ}: H \ot V \to V\ot H$, and that 
there exists a non-degenerate subspace $W\subset V'$  such that 
\[ \label{PWHcc}
\Psi_{WH}^{\circ\circ}: W \ot H  \lra (U\ot V)', \quad 
\Psi_{WH}^{\circ\circ}(e\ot a)(f\ot v):=\ipp{f\ot e}{\Psi_{HV}^{\circ\circ}(a\ot v )}, 
\]
yields a bijection $\Psi_{WH}^{\circ\circ}: W \ot H \to H\ot W\subset (U\ot V)'$.
Then the map $\nu_{L}^\backprime : H\ot V \to W'$, 
\[ \label{actUH}
\nu_{L}^\backprime (a\ot v)(e)  :=  
\ip{\cdot}{\hsp\cdot} \circ (\nu_{L} \ot \id  )\circ  
(\Psi_{HV}^{-1} \ot \id )\circ (\id \ot \Psi_{WH}^{\circ\circ} )\circ(v\ot e \ot a), 
\]
turns $V$ into a braided left $H$-module with respect to the braiding $\Psi_{HV}^{\circ\circ}$ 
and the multiplication $m_{-2}:= m\circ \Psi_{HH}^{-2}$ on $H$.  

If $H$ is a braided bialgebra, $V$ is a braided left $H$-module algebra, 
and  $U$ is a left dual of $H$ as in Theorem \ref{dH}, 
then the left action defined in \eqref{actUH} 
turns $V$ into a braided left $H^{(-2,2)}$-module algebra. 

Analogous statements hold for the opposite versions. 
\end{corollary} 
\begin{proof} 
\,$(i)$ \,First note that the braided coalgebra $(H,\Delta, \eps)$ and the subalgebra $U\subset H'$
in {\it (i)} satisfy the assumptions of Theorem \ref{mcom}. Hence the left action \eqref{nuL} 
turns $V$ into a braided left $U$-module with respect to the braiding $\Psi_{UV}^\circ$. 
From \eqref{dud}, it follows that $H\cong \iota(H)\subset U^{\uc}$ is a non-degenerate subcoalgebra. 
Furthermore, $\Psi_{HV}^{\circ\circ}$ corresponds to $\Psi_{UV}^\circ$ in Theorem \ref{mmm}, 
where $H$ and $U$ swap the roles in the present situation. 
Applying Theorem \ref{mmm}, we conclude that the $H$-coaction defined in \eqref{rR} turns $V$ 
into a braided right $H$-comodule with respect to the braiding $\Psi_{HV}^{\circ\circ}$ 
and the coproduct $\Delta_{-2} = \Psi_{HH}^{-2}\circ \Delta$. But before we can apply 
Theorem~\ref{mmm}, we need prove that the coaction resulting from this construction 
coincides with  \eqref{coactUH} so that \eqref{rRVU} is fulfilled by assumption. 

Let $v\in V$. To distinguish the new coaction $\rho_{R}^{\backprime}$ from the original one, we use 
the Sweedler notation $\rho_{R}^{\backprime}(v) := v_{(0)^{\backprime}}\ot v_{(1)^{\backprime}}$. 
Evaluating  $\rho_{R}^{\backprime}(v)\in V\ot H$ 
on $e\ot f\in W\ot U$ yields 
\begin{eqnarray*} 
e(v_{(0)^{\backprime}})\, f(v_{(1)^{\backprime}}) 
\!\!&\oeq{\eqref{rR}} &\!\!
e( f^{{\cBr{1}}^\backprime\!} \la v^{{\cBr{1}}^\backprime\!}\,) 
 \oeq{\eqref{nuL}} 
 f^{{\cBr{1}}^\backprime\!}\hsp(v^{{\cBr{2}}^\backprime\!}{\hsp}_{(1)}{\!}^{\pBr{1}\,\hs} )\, 
 e( v^{{\cBr{2}}^\backprime\!}{\hsp}_{(0)}{\!}^{\pBr{2}\,\hs})\\
 \!\!&\oeq{\eqref{cPHV}} &\!\! 
 f^{{\cBr{1}}^\backprime\cBr{2}}\hsp(v^{{\cBr{2}}^\backprime\!}{\hsp}_{(1)})\, 
 e( v^{{\cBr{2}}^\backprime\!}{\hsp}_{(0)}{\!}^{\cBr{1}}) ,   
\end{eqnarray*} 
which is equivalent to 
\[ \label{rohn}
v_{(0)^{\backprime}}\, f(v_{(1)^{\backprime}}) 
= f^{{\cBr{1}}^\backprime\cBr{2}}\hsp(v^{{\cBr{2}}^\backprime\!}{\hsp}_{(1)})\, 
  v^{{\cBr{2}}^\backprime\!}{\hsp}_{(0)}{\!}^{\cBr{1}}.
\]
Furthermore, 
\begin{eqnarray*}
 f^{\cBr{2}}\!(a)\, v^{\cBr{1}}{\!}_{(0)} \ot v^{\cBr{1}}{\!}_{(1)} 
\!\!\!\!&\oeq{\eqref{fav}} &\!\!\!\!
 f(a^{\pBr{1}}\,\hs)\, v^{\pBr{2}}{\!}_{(0)} \ot v^{\pBr{2}}{\!}_{(1)} 
 \oeq{\eqref{invPHV}}  
  f(a^{\pbr{1}\pBr{1}}\,\hs)\, v_{(0)}{\!}^{\pBr{2}} \ot v_{(1)}{\!}^{\pbr{2}}\\
 \!\!\!\!&\oeq{\eqref{ophcph},\eqref{fav}}  &\!\!\!\! 
  f^{\cBr{2}_1\cBr{2}_2}(a)\, v_{(0)}{\!}^{\cBr{1}_1} \ot v_{(1)}{\!}^{\cBr{1}_2}
\end{eqnarray*}
gives  
$$
(\rho_R\ot\id)\circ \Psi_{UV}^\circ = ( \id \ot \Psi_{UH}^\circ)\circ (\Psi_{UV}^\circ \ot \id )\circ (\id\ot\rho_R). 
$$ 
Applying on both sides the corresponding inverse braidings, we obtain 
\[ \label{invn}
 ( \id \ot \Psi_{UH}^{\circ-1})\circ(\rho_R\ot\id)
= (\Psi_{UV}^{\circ} \ot \id )\circ (\id\ot\rho_R)\circ \Psi_{UV}^{\circ -1}. 
 \] 
Inserting \eqref{invn} into \eqref{rohn} yields 
\begin{eqnarray*}
v_{(0)^{\backprime}}\, f(v_{(1)^{\backprime}}) 
\oeq{\eqref{invn},\eqref{rohn}}   
f^{{\cBr{1}}^\backprime}\!(v_{(1)}{\!}^{{\cBr{2}}^\backprime}\hs)\, v_{(0)}. 
\end{eqnarray*}
This proves \eqref{coactUH} for the given right $H$-coaction $\rho_R(v) = v_{(0)} \ot v_{(1)}$ on $V$ 
and with the braiding $\Psi_{UH}^{\circ\,-1}(a\ot f) = f^{{\cBr{1}}^\backprime} \ot a^{{\cBr{2}}^\backprime}$ 
on $H\ot U$. 

Assume now that $H$ is a braided bialgebra and $V$ is a braided right $H$-comodule algebra. 
Then, by Theorem \ref{mcom}, $V$ becomes a left $U$-module algebra. 
In Proposition \ref{propdual}, it has been shown that $H\cong \iota(H)$  is a left dual of $U$. 
Applying Theorem \ref{mmm} with the roles of $U$ and $H$ reversed shows that the new 
coaction $\rho_{R}^{\backprime}$ turns $V$ into a 
braided right $H^{(2,-2)}$-comodule algebra. 

\,$(ii)$ \,Consider a braided unital algebra $(H,m)$, a braided left $H$-module $V$ and 
a non-degenerate subcoalgebra $U \subset H^{\uc}$ as described in {\it (ii)}. 
From Theorem \ref{mmm}, we conclude that $V$ becomes a 
braided right $U$-comodule with respect to 
the coproduct $\uD_{\hs -2}:= \Psi_{UU}^{-2}\circ \uD$ on $U$, 
the right coaction defined in \eqref{rR}, and 
the braiding $\Psi_{UV}^\circ$. 
As in the proof of  {\it (i)}, $\Psi_{HV}^{\circ\circ}$ 
corresponds to the braiding $\Psi_{UV}^\circ$ in Theorem~\ref{mcom} 
with the roles of $H$ and $U$ reversed.  
Furthermore, the multiplication on $H\cong \iota(H)\subset U'$ 
corresponding to the coproduct $\uD_{\hs -2}$ on $U$  in \eqref{us} 
is given by $m_{-2}= m\circ \Psi_{HH}^{-2}$.  Indeed, for all $a,b\in H$ and $f\in U$, we have  
\begin{eqnarray*}  
\ip{\um(\iota(a)\ot \iota(b))}{f}\!\!\!&\oeq{\eqref{us}}& \!\!\!\ipp{a\ot b}{\Psi_{UU}^{-1}\! \circ\! \uD_{\hs -2}(f)} 
 \oeq{\eqref{mkDn2},\eqref{UU}}  \ipp{\Psi_{HH}^{-3}(a\ot b)}{\uD(f)}  \\ 
 \!\!\!&\oeq{\eqref{brcop}}& \!\!\! \ip{m\circ \Psi_{HH}^{-2}(a\ot b)}{f} 
 \oeq{\eqref{mkDn1}} \ip{m_{-2}(a\ot b)}{f}.  
\end{eqnarray*} 
Applying now Theorem~\ref{mcom} shows that $V$ becomes a 
braided left $H$-module with respect to the 
 left $H$-action defined in \eqref{nuL}, the 
multiplication $m_{-2}:= m\circ \Psi_{HH}^{-2}$ on $H$, 
and the braiding $\Psi_{HV}^{\circ\circ}$. 

Let  $\nu_{L}^\backprime$ denote the new left $H$-action on $V$. 
To show that $\nu_{L}^\backprime$ is given by \eqref{actUH}, 
we use the Sweedler-type notation 
$$
\Psi_{HV}^{\circ\circ}(a\ot v) := v^{\Br{1}^{\circ\circ}} \ot a^{\Br{2}^{\circ\circ}} , \quad 
\Psi_{WH}^{\circ\circ}(e\ot a) := a^{\Br{1}^{\circ\circ}} \ot e^{\Br{2}^{\circ\circ}}, 
$$
for $a\in H$, $v\in V$ and $e\in W$,  
and compute that 
\begin{eqnarray} \nonumber 
e(\nu_{L}^\backprime(a\ot v))\!\! &\oeq{\eqref{nuL}} &\!\!
e(v_{(0)}{\!}^{{\cBr{2}}^\backprime}) \, v_{(1)}{\!}^{{\cBr{1}}^\backprime}\!(a)
 \,\oeq{\eqref{PHVcc}} \,
e(v_{(0)}{\!}^{\Br{1}^{\circ\circ}}) \,  v_{(1)}\hsp(a^{\Br{2}^{\circ\circ}}) \\ \label{bnuL}
&\oeq{\eqref{PWHcc}} &\!\!
e^{\Br{2}^{\circ\circ}}\!(v_{(0)}) \,  v_{(1)}\hsp(a^{\Br{1}^{\circ\circ}}) 
\oeq{\eqref{SrR}}  
e^{\Br{2}^{\circ\circ}}\!(a^{\Br{1}^{\circ\circ}\pBr{1}} \la  v^{\pBr{2}\,\hs}). 
\end{eqnarray} 
This shows that  $\nu_{L}^\backprime$ is given by \eqref{actUH}. 
In particular, we have  $\nu_{L}^\backprime : H\ot V \to V\subset W'$. 

If $H$ is a braided bialgebra, $V$ is a braided left $H$-module algebra, 
and $U$ is a left dual of $H$ as in Theorem \ref{dH}, then we conclude from 
Theorem \ref{mmm} that $V$ becomes a braided right $U^{(2,-2)}$-comodule algebra.  
Since $H^{(-2,2)}$ is a left dual of $U^{(2,-2)}$ by Proposition \ref{Pdual}, 
we deduce from Theorem \ref{mcom} that the left action $\nu_{L}^\backprime$  
turns $V$ into a braided left $H^{(-2,2)}$-module algebra. 

The opposite versions versions are proven analogously. 
\end{proof} 

The results of the last corollary may be seen as an induction step. 
Starting with an $H$-(co)module $V$, Corollary \ref{corind} produces another 
$H$-(co)action on $V$ but for the modified coproduct  $\Delta_{-2}$ or product $m_{-2}$. 
If the assumptions of the corollary are again satisfied, we can repeat the process,  
and so on, obtaining thus for all $n\in \N$ an $H$-(co)action on $V$ for the 
coproduct  $\Delta_{-2n}$ or the product $m_{-2n}$. 
If $H$ is a braided bialgebra and $V$ a braided (co)module algebra, 
then we get an coaction of the braided bialgebra $H^{(2n,-2n)}$ or an action 
of $H^{(-2n,2n)}$.

Theorem \ref{mcom} shows how to turn a braided comodule of a braided coalgebra into braided 
module of the dual algebra. However, 
it is more natural to dualize a coaction 
in such a way that a \emph{dual} space of the comodule becomes a module of the dual algebra. 
This will be done in the next proposition. 
Unlike Theorem \ref{mcom}, a braided comodule algebra 
will not dualize to braided module algebra. 
The correct way would be to dualize it to a braided module coalgebra but, 
as mentioned in Remark \ref{rem1}, we do not discuss these 
structures here. 
\begin{proposition}  \label{mco}
Let $H$ be a braided coalgebra and $V$ a braided right $H$-comodule 
with co\-action $\rho_R: V\to V\ot H$. 
Let $U\subset H'$ and $W\subset V'$ satisfy 
the assumptions of Lemmas \ref{lembraided},  \ref{dbvs}, \ref{LUVH} and \ref{bUW}  
which guarantee that the braidings $\Psi_{UU}$, $\Psi_{HW}^\circ$, $\Psi_{UV}^\circ$ and 
$\Psi_{WU}^\bullet$ are well-defined. Assume that $U\subset  H'$ is a (unital) subalgebra 
with respect to the product defined in \eqref{us}. Consider 
\[ \label{vRWU}
\nu_R : W\ot U \lra V', \quad \nu_R(e\ot f)(v) :=  f(v_{(1)}{\!}^{\pBr{1}}\,\hs)\, \hs e(v_{(0)}{\!}^{\pBr{2}}\,)
\] 
If $\nu_R : W\ot U \to W \subset V'$, then it defines a right $U$-action on $W$ such that $W$ becomes 
a braided right $U$-module with respect to the braidings $\Psi_{UU}^{-1}$  and  $\Psi_{WU}^\bullet$.  

For a braided left $H$-comodule $V$, it is required that $U$ and $W$ satisfy 
the assumptions of Lemmas \ref{lembraided},  \ref{dbvs}, \ref{LUVH} and \ref{bUW}  
which guarantee that the braidings 
$\Psi_{UU}$, $\Psi_{WH}^\circ$, $\Psi_{VU}^\circ$ and $\Psi_{UW}^\bullet$ 
are well-defined. If 
\[
\nu_L : U\ot W \lra V', \quad \nu_L(f\ot e)(v) :=  f(v_{(-1)}{\!}^{\pBr{2}}\,\hs)\, \hs e(v_{(0)}{\!}^{\pBr{1}}\,)
\] 
yields a map $\nu_L : U\ot W \to W\subset V'$, then it defines a left $U$-action on $W$ such that $W$ becomes 
a braided left $U$-module with respect to the braidings $\Psi_{UU}^{-1}$  and $\Psi_{UW}^\bullet$. 
\end{proposition} 
\begin{proof} 
As customary, we denote the map defined in \eqref{vRWU} by $\nu_R(e\ot f) := e\ra f$ 
and the left action given in \eqref{nuL} by  $\nu_L(f\ot v) := f\la v$. 
Since 
$$
(e\ra f)(v)  \oeq{\eqref{vRWU}}  f(v_{(1)}{\!}^{\pBr{1}}\,\hs)\, e(v_{(0)}{\!}^{\pBr{2}}\,\hs)
 \oeq{\eqref{nuL}}  e(f\la v),\quad  f\in U, \ \, v\in V, \ \, e\in V', 
$$ 
it follows immediately from Theorem \ref{mcom} 
that \eqref{vRWU} defines a right $U$-action on $V'$. 
Assuming that $\nu_R : W\ot U \to W$, we need to prove the 
compatibility with the braiding $\Psi_{WU}^\bullet$.  

Although the proof of \eqref{bmu} goes along the lines of previous ones, 
we present the computations in order to show where all the listed braidings 
and Lemma \ref{bUW} are needed.  
Let now $f,g\in U$, $a\in H$, $e\in W$ and $v\in V$. 
Employing the Sweedler-type notation  $\Psi_{WU}^\bullet(e\ot f) := f^{\Br{1}^\bullet}\!\ot e^{\Br{2}^\bullet}$, 
we have 
\begin{align*}
&f^{\Br{1}^\bullet}\!(a)\, (e\ra g)^{\Br{2}^\bullet}\!(v) 
\oeq{\eqref{VUb}}
f^{\cBr{2}}\!(a)\, (e\ra g)(v^{\cBr{1}}) 
\oeq{\eqref{cPHV}}  f(a^{\pBr{1}}\,)\, (e\ra g)(v^{\pBr{2}}\,) \\
&\oeq{\eqref{vRWU}}  
f(a^{\pBr{1}_1}\,)\, g(v^{\pBr{2}_1}{\hsp}_{(1)}{\!}^{\pBr{1}_2})\, e(v^{\pBr{2}_1}{\hsp}_{(0)}{\!}^{\pBr{2}_2})
\oeq{\eqref{VW},\eqref{invPHV}} 
f(a^{\pBr{1}_2\pbr{1}})\, g(v_{(1)}{\!}^{\pBr{1}_1\pbr{2}})\, e(v_{(0)}{\!}^{\pBr{2}_1\pBr{2}_2}) \\
&\oeq{\eqref{cphh}, \eqref{cHV}} 
f^{\pbr{1}}(a^{\cBr{2}})\, g^{\pbr{2}}(v_{(1)}{\!}^{\pBr{1}}\,\hs)\, e^{\cBr{1}}(v_{(0)}{\!}^{\pBr{2}}\,\hs) 
\oeq{\eqref{WHb}} 
f^{\pbr{1}\Br{1}^\bullet}\!(a)\, g^{\pbr{2}}\!(v_{(1)}{\!}^{\pBr{1}}\,)\, e^{\Br{2}^\bullet}\!(v_{(0)}{\!}^{\pBr{2}}\,) \\
&\oeq{\eqref{vRWU}}  
f^{\pbr{1}\Br{1}^\bullet}\!(a)\, (e^{\Br{2}^\bullet}\! \ra g^{\pbr{2}})(v), 
\end{align*} 
which implies \eqref{bmu} for the braidings $\Psi_{WU}^\bullet$ and $\Psi_{UU}^{-1}$. 
Furthermore, we know from Lemma \ref{Linv} that $U$ is a braided algebra with respect 
to the braiding $\Psi_{UU}^{-1}$, so this finishes the proof the first part. 
The second part is proven similarly. 
\end{proof} 

The analog of dualizing a coaction to an action as in the last proposition 
consists in dualizing an action to obtain a coaction of the dual coalgebra on the dual space. 
This is the purpose of our last proposition. 
\begin{proposition}  \label{com}
Let $H$ be a braided unital algebra and $V$ a braided left $H$-module 
with action $\nu_L: H\ot V\to V$.  
Let $U\subset H'$ and $W\subset V'$ satisfy 
the assumptions of Lemmas \ref{lembraided},  \ref{dbvs}, \ref{LUVH} and \ref{bUW}  
which guarantee that the braidings $\Psi_{UU}$, $\Psi_{HW}^\circ$, $\Psi_{UV}^\circ$ and 
$\Psi_{WU}^\bullet$ are well-defined. 
Assume that $U\subset H^{\uc}$ is a subcoalgebra as in Proposition \ref{lcop} but 
with respect to the braiding $\Psi_{HH}^{-1}$ on $H$, i.e., $\uD^{\!\circ}(f) \in U\ot U$ for all $f\in U$, where 
\[ \label{cuD} 
\ipp{\uD^{\!\circ}(f)}{a\ot b}:= \ip{f}{\,m\!\circ\! \Psi_{HH}^{-1} (a\ot b)}=f(b^{\pbr{1}}  a^{\pbr{2}}),  \qquad 
a,b\in H. 
\] 
Consider the linear map 
\[ \label{rLW}
\rho_L : W\lra (H\ot V)', \quad \rho_L (e)(a\ot v) := e(a^{\pBr{1}}\la v^{\pBr{2}}\,). 
\] 
If $\rho_L : W\to  U\ot W \subset (H\ot V)'$, then it defines a left $U$-coaction on $W$ such that $W$ becomes 
a braided left $U$-comodule with respect to the coalgebra $(U, \uD^{\!\circ}, \eps)$, 
the braiding $\Psi_{UU}^{-1}$ on $U$ and  the braiding $\Psi_{WU}^\bullet$ between $W$ and $U$.  

For a braided right $H$-comodule $V$, it is required that $U$ and $W$ satisfy 
the assumptions of Lemmas \ref{lembraided},  \ref{dbvs}, \ref{LUVH} and \ref{bUW}  
which guarantee that the braidings 
$\Psi_{UU}$, $\Psi_{WH}^\circ$, $\Psi_{VU}^\circ$ and $\Psi_{UW}^\bullet$ 
are well-defined. If the linear map 
\[
\rho_R : W\lra (V\ot H)', \quad 
 \rho_R (e)(v\ot a) := e( v^{\pBr{1}} \ra  a^{\pBr{2}}\,), 
\] 
fulfills $\rho_R : W\to  W \ot U \subset (V\ot H)'$, then it defines a right $U$-action on $W$ 
such that $W$ becomes a braided right $U$-comodule with respect to the coalgebra $(U, \uD^{\!\circ}, \eps)$, 
the braiding $\Psi_{UU}^{-1}$ on $U$ and  the braiding $\Psi_{UW}^\bullet$ between $U$ and $W$. 
\end{proposition} 
\begin{proof} As in the proof the previous proposition, 
we prove \eqref{Droh} and \eqref{brL} in order to show that the correct braidings 
and the correct coproduct are used. 
Let $f\in U$, $a,b\in H$, $e\in W$ and $v\in V$. Then 
\begin{align*}
&e_{(-1)}(a)\, (e_{(0)})_{(-1)}(b)\,(e_{(0)})_{(0)}(v) \oeq{\eqref{rLW}} 
e(a^{\pBr{1}_2}\la(b^{\pBr{1}_1}\la v^{\pBr{2}_1})^{\pBr{2}_2}) \\
&\oeq{\eqref{VW},\eqref{nulinv}}    
e(a^{\pbr{1}\pBr{1}_1}\la(b^{\pbr{2}\pBr{1}_2}\la v^{\pBr{2}_1\pBr{2}_2}))
 \oeq{\eqref{anu}} e((a^{\pbr{1}\pBr{1}_1}b^{\pbr{2}\pBr{1}_2})\la v^{\pBr{2}_1\pBr{2}_2})\\
&\oeq{\eqref{AWm}}   
e((a^{\pbr{1}}b^{\pbr{2}})^{\pBr{1}}\la v^{\pBr{2}}\,) 
\oeq{\eqref{rLW}}  e_{(-1)}(a^{\pbr{1}}b^{\pbr{2}}) \, e_{(0)}(v) 
\oeq{\eqref{cuD}} (e_{(-1)})_{(1)}\hsp (a)\, (e_{(-1)})_{(2)}\hsp (b)\, e_{(0)}\hsp (v), 
\end{align*} 
so that the first relation of \eqref{Droh} holds for the coproduct given in \eqref{cuD}. 
Moreover, 
$$
\eps(e_{(-1)})\, e_{(0)}\hsp (v) = e_{(-1)}(1)\,e_{(0)}\hsp (v) 
\oeq{\eqref{rLW}} e( 1^{\pBr{1}}\la v^{\pBr{2}}\,) \oeq{\eqref{1W},\eqref{anu}} e(v).  
$$
Hence $\rho_L$ defines a left $U$-coaction for the coalgebra $(U, \uD^{\!\circ}, \eps)$. 

Recall from \eqref{cphh} that $\Psi_{UU}^{-1}$ is the braiding on $U$ 
which corresponds to $\Psi_{HH}^{-1}$ on $H$ 
according to \eqref{UU}. Now, 
\begin{align*} 
& f^{\Br{1}^\bullet\pbr{1}}(a)\,  e_{(-1)}{\!}^{\pbr{2}}\!(b) \, e_{(0)}{\!}^{\Br{2}^\bullet}\! (v)
 \oeq{\eqref{cphh}}
  f^{\Br{1}^\bullet}\!(a^{\pbr{1}}\hs)\,  e_{(-1)}\hsp (b^{\pbr{2}}) \, e_{(0)}{\!}^{\Br{2}^\bullet}\! (v) \\
 & \oeq{\eqref{cPHV},\eqref{VUb}} 
  f(a^{\pbr{1}\pBr{1}}\,)\,  e_{(-1)}\hsp (b^{\pbr{2}}) \, e_{(0)}\hsp (v^{\pBr{2}}\,) 
  \oeq{\eqref{rLW}}
   f(a^{\pbr{1}\pBr{1}_1})\,  e(b^{\pbr{2}\pBr{1}_2}\la v^{\pBr{2}_1\pBr{2}_2}) \\
   & \oeq{\eqref{VW},\eqref{nulinv}} 
    f(a^{\pBr{1}_2})\,  e((b^{\pBr{1}_1}\la v^{\pBr{2}_1})^{\pBr{2}_2}) 
   \oeq{\eqref{cPHV},\eqref{VUb}}  f^{\Br{1}^\bullet}\!(a)\, e^{\Br{2}^\bullet}\!(b^{\pBr{1}}\la v^{\pBr{2}}\,)\\
   &\oeq{\eqref{rLW}} 
   f^{\Br{1}^\bullet}\!(a)\,  e^{\Br{2}^\bullet}{\!\!}_{(-1)}(b)\, e^{\Br{2}^\bullet}{\!\!}_{(0)}(v), 
\end{align*} 
which proves \eqref{brL} for the braidings  $\Psi_{UU}^{-1}$ on $U$ and 
$\Psi_{WU}^\bullet$ between $W$ and $U$. 

The proof of the opposite version is similar. 
\end{proof}
Note that we used all lemmas of this section 
in the proofs of Propositions \ref{mco} and \ref{com}. 
As a closing remark, 
let us point out that the relevance of the inverse braidings  
is undeniable throughout this paper, not only for turning right braided vector spaces into  
left braided vector spaces and vice versa, but also 
in the definitions of the braidings of the type $\Psi_{XY}^\circ$, 
in the definitions of actions and coactions in Theorems \ref{mcom} and \ref{mmm}, 
and in our final propositions.

\section{Examples} 

\begin{example} \label{ex1} 
Finite-dimensional braided (co)algebras, bialgebras and Hopf algebras and their 
finite-dimensional (co)modules yield examples of all presented structures \cite{Maj95,T}. 
In the finite-dimensional situation, the non-degenerated dual spaces are obviously unique. 
Explicit formulas can be deduced by using the coevaluation map 
\[ \label{coev} 
\mathrm{coev}_H : \K \lra H \ot H', \quad \mathrm{coev}_H(1) = \msum{j=1}{n}\, e_j \ot e^j, 
\]
where $\{e_1,\ldots, e_n\} \subset H$ is a linear basis 
and $\{e^1,\ldots, e^n\} \subset H'$ its dual basis. 
Under the identification $\K\ot H\cong H\cong H\ot \K$, \eqref{coev} yields 
$ (\id\ot \mathrm{ev})  \circ (\mathrm{coev}_H\ot \id) = \id : H\to H$ 
and $ (\mathrm{ev}\ot\id )  \circ (\id\ot \mathrm{coev}_{H'}) = \id : H\to H$. 
From the dual versions of these identities, 
we obtain for instances the following formulas for the induced braidings  
$\Psi_{H'H'}$ and $\Psi_{H'H}$,  
$$
g^{\br{1}} \ot f^{\br{2}} = \msum{j,k=1}{n}\,g(e_j^{\br{1}})\,f(e_k^{\br{2}})\, e^j \ot e^k , 
\qquad a^{\br{1}} \ot f^{\br{2}} =  \msum{j=1}{n}\, \,f(e_j^{\br{1}})\, a^{\br{2}} \ot e^j, 
$$
and similar formulas for all other induced braidings. Furthermore,  
the coproduct \eqref{brcop} and 
the product \eqref{us}  may be written in the form 
$$
\uD(f) = \msum{j,k=1}{n}\, f(e_j^{\br{1}} \hs e_k^{\br{2}})\, e^j \ot e^k
$$
and 
$$
\um(f\ot g) = \msum{j=1}{n}\, f(e_{j(1)}{\!}^{\pbr{2}})\,g(e_{j(2)}{\!}^{\pbr{1}})\, e^j,  
$$
respectively. 
Analogous expressions can be derived for actions and coactions. 
For example, the coaction $\rho_R: H\to H\ot H'$ in \eqref{rR} 
for $V=H$ and with the multiplication as left action is given by 
$$
\rho_R(a) = \msum{j=1}{n}\, e_j^{\pbr{1}} a^{\pbr{2}} \ot e^j. 
$$
Similarly, the left action $\nu_L: H\ot H' \to H'$ in \eqref{nuL} for $U=H$, $V=H'$, and 
with the right $H'$-coaction on $H'$ given by the coproduct  
$\rho_R=\uD: H' \to H' \ot H'$ from \eqref{brcop}, becomes 
$$
\nu_L(a\ot f) =  \msum{j=1}{n}\,  f(e_ja)\,e^j. 
$$
\end{example} 
\begin{example}
Graded braided (co)algebras, bi- and Hopf algebras and their 
graded (co)\-modules \cite{NT} provide examples 
if the spaces of homogeneous elements are finite-dimen\-sional 
for all grades. The dual spaces may be given as the direct sum of the duals 
of the spaces of homogeneous elements so that   
the existence of the presented structures can be deduced grade by grade from Example \ref{ex1}. 
In this way, we obtain a large class of infinite-dimensional examples. 
With some care, these arguments can be generalized to filtered (co)algebras, bialgebras, and so on.  
\end{example}

 \section*{Acknowledgments}
 It is a pleasure to thank Andrzej Sitarz for stimulating discussions that inspired this report. 
The author gratefully acknowledges partial financial support 
from the CONACyT project A1-S-46784 and the CIC-UMSNH project 
 "Grupos cu\'anticos y geometr\'ia no conmutativa".

\end{document}